\documentclass[10pt,reqno]{amsart}

\usepackage{amsmath}
\usepackage{amsfonts}
\usepackage{amssymb}
\usepackage{amsthm}
\usepackage{paralist}
\usepackage{mathrsfs}
\usepackage{url}

\numberwithin{equation}{section}

\usepackage{color}

\usepackage{graphicx}
\usepackage{subfigure}
\usepackage{ifthen} 

\provideboolean{shownotes} 
\setboolean{shownotes}{false} 
\newcommand{\margnote}[1]{
\ifthenelse{\boolean{shownotes}}%
{\marginpar{\raggedright\tiny\texttt{#1}}}%
{}%
}

\newcommand{\hole}[1]{
\ifthenelse{\boolean{shownotes}}%
{\begin{center} \fbox{ \rule {.25cm}{0cm}
\rule[-.1cm]{0cm}{.4cm} \parbox{.85\textwidth}{\begin{center}
\texttt{#1}\end{center}} \rule {.25cm}{0cm}}\end{center}}
{}
}

\newcommand\xqed[1]{%
  \leavevmode\unskip\penalty9999 \hbox{}\nobreak\hfill
  \quad\hbox{#1}}
\newcommand\myendrmk{\xqed{$\triangleright$}}

\theoremstyle{plain}

\newtheorem{lemma}{Lemma}[section]
\newtheorem{theorem}[lemma]{Theorem}
\newtheorem{proposition}[lemma]{Proposition}
\newtheorem{corollary}[lemma]{Corollary}

\theoremstyle{definition}

\newtheorem{remark}[lemma]{$\triangleleft$ Remark}
\newtheorem{definition}[lemma]{Definition}
\newtheorem{assumption}[lemma]{Assumption}

\theoremstyle{remark}

\newcommand{\Id}{\mathbb{I}}

\newcommand{\br}{\mathbf{r}}
\newcommand{\bw}{\mathbf{w}}
\newcommand{\by}{\mathbf{y}}
\newcommand{\bz}{\mathbf{z}}
\newcommand{\bA}{\mathbf{A}}
\newcommand{\bM}{\mathbf{M}}
\newcommand{\bF}{\mathbf{F}}
\newcommand{\bG}{\mathbf{G}}
\newcommand{\bQ}{\mathbf{Q}}

\newcommand{\bB}{\mathbf{B}}

\newcommand{\bT}{\mathbf{T}}
\newcommand{\bX}{\mathbf{X}}
\newcommand{\bU}{\mathbf{U}}

\newcommand{\R}{\mathbb{R}}
\newcommand{\C}{\mathbb{C}}
\newcommand{\Z}{\mathbb{Z}}

\newcommand{\cJ}{{\mathcal{J}}}
\newcommand{\cT}{{\mathcal{T}}}

\newcommand{\cL}{{\mathcal{L}}}

\newcommand{\cS}{{\mathcal{S}}}

\newcommand{\cH}{{\mathcal{H}}}
\newcommand{\region}{{\mathbb{G}}}

\renewcommand{\Re}{\mathrm{Re}\,} 
\renewcommand{\Im}{\mathrm{Im}\,}
\newcommand{\tr}{\mathrm{tr}\,}
\newcommand{\sgn}{\mathrm{sgn}\,}

\newcommand{\ess}{\sigma_\mathrm{\tiny{ess}}}
\newcommand{\ptsp}{\sigma_\mathrm{\tiny{pt}}}

\newcommand{\<}{\langle}
\renewcommand{\>}{\rangle}

\begin{document}

\title[Spectral and Modulational Stability of Klein-Gordon Wavetrains]{Spectral and Modulational Stability of Periodic Wavetrains for the Nonlinear Klein-Gordon Equation}

\author[C.K.R.T. Jones]{Christopher K. R. T. Jones}

\address{{\rm (C. K. R. T. Jones)} Department of Mathematics\\University of North Carolina\\Chapel Hill, NC 27599 (U.S.A.)}

\email{ckrtj@amath.unc.edu}

\author[R. Marangell]{Robert Marangell}

\address{{\rm (R. Marangell)} School of Mathematics and Statistics F07\\University of Sydney\\Sydney NSW 2006 (Australia)}

\email{robert.marangell@sydney.edu.au}

\author[P.D. Miller]{Peter D. Miller}

\address{{\rm (P. D. Miller)} Department of Mathematics\\University of Michigan\\Ann Arbor, MI 48109 (U.S.A.)}

\email{millerpd@umich.edu}

\author[R.G. Plaza]{Ram\'on G. Plaza}

\address{{\rm (R. G. Plaza)} Departamento de Matem\'aticas y Mec\'anica\\Instituto de Investigaciones en Matem\'aticas Aplicadas y en Sistemas\\Universidad Nacional Aut\'onoma de M\'exico\\Apdo. Postal 20-726, C.P. 01000 M\'exico D.F. (Mexico)}

\email{plaza@mym.iimas.unam.mx}

\begin{abstract}
This paper is a detailed and self-contained study of the stability properties of periodic traveling wave solutions of the nonlinear Klein-Gordon equation $u_{tt}-u_{xx}+V'(u)=0$, where $u$ is a scalar-valued function of $x$ and $t$, and the potential $V(u)$ is of class $C^2$ and periodic.  Stability is considered both from the point of view of spectral analysis of the linearized problem (spectral stability analysis) and from the point of view of wave modulation theory (the strongly nonlinear theory due to Whitham as well as the weakly nonlinear theory of wave packets). The aim is to develop and present new spectral stability results for periodic traveling waves, and to make a solid connection between these results and predictions of the (formal) modulation theory, which has been developed by others but which we review for completeness.  
\end{abstract}

\maketitle

\setcounter{tocdepth}{1}

\tableofcontents

\section{Introduction}
Consider the nonlinear Klein-Gordon equation
\begin{equation}
\label{eqnlKG}
u_{tt} - u_{xx} + V'(u) = 0, 
\end{equation}
where $u$ is a scalar function of $(x,t) \in \R \times [0,+\infty)$ and the potential $V$ is a real periodic function. Such potentials are intended as a generalization of the case 
$V(u) = - \cos(u)$, for which equation \eqref{eqnlKG} becomes the well-known sine-Gordon equation \cite{Sco2} in laboratory coordinates,
\begin{equation}
\label{eqsineG}
u_{tt} - u_{xx} + \sin(u) = 0.
\end{equation}

From a different point of view, however,
the nonlinear Klein-Gordon equation \eqref{eqnlKG} generalizes the linear relativistic equation for a charged particle in an electromagnetic field derived by Klein \cite{Klein27} and Gordon \cite{Gordon26}. The sine-Gordon equation \eqref{eqsineG}, in contrast, did not first appear in the context of nonlinear wave propagation, but rather in the study of the geometry of surfaces with negative Gaussian curvature \cite{Eis0}. Given the form of the nonlinear term, its name was coined as an inevitable pun. Thereafter, the sine-Gordon equation has appeared in many physical applications such as the study of elementary particles \cite{PeSky}, the propagation of crystal dislocations \cite{FreKo}, the dynamics of fermions in the Thirring model \cite{Cole75}, the propagation of magnetic flux on a Josephson line \cite{Sco2,SCR}, the dynamics of DNA strands \cite{GaetaRPD94}, and the oscillations of a series of rigid pendula attached to a stretched rubber band \cite{Dra0}, among many others. To sum up, the general form of equation \eqref{eqnlKG} constitutes one of the simplest and most widely applied prototypes of nonlinear wave equations in mathematical physics (an abridged bibliographic literature list includes \cite{BEMS,BuM2,SCM,Wh} and the references therein).

In the context of wave propagation problems, the study of periodic solutions representing regular trains of waves is of fundamental interest.  The existence and variety of such solutions can be studied by reducing the nonlinear Klein-Gordon equation \eqref{eqnlKG} to an appropriate ordinary differential equation.  However, the question of stability, that is, the dynamics of solutions initially close to a regular train of waves, demands the study of \eqref{eqnlKG} as an infinite-dimensional dynamical system.

There are two common approaches to the stability question that address ostensibly different aspects of the problem.  Firstly, one can analyze the nonlinear initial-value problem governing the difference between an arbitrary solution of the nonlinear Klein-Gordon equation \eqref{eqnlKG} and a given exact solution representing a train of waves.  In the first approximation one typically assumes that the difference is small and linearizes.  The resulting linear equation can in turn be studied in an appropriate frame of reference by a spectral approach.  To our knowledge, the linearized spectral approach was first taken by Scott \cite{Sco1} in the special case of the sine-Gordon equation.  It turns out that there are serious logical flaws in the reasoning of \cite{Sco1}, but we were recently able to give a completely rigorous spectral stability analysis of regular trains of waves for the sine-Gordon equation \cite{JMMP1}.  In part we obtained these results with the help of an exponential substitution introduced by Scott, but we used the substitution in a new way.  While Scott's prediction of details of the spectrum proved to be incorrect, we showed that the basic result of which types of waves are spectrally stable and unstable is indeed as written in Scott's paper \cite{Sco1}.  Thus we have given the first correct confirmation of a fact long accepted in the nonlinear wave propagation literature.  

A second approach to stability of nonlinear wave trains is based on the idea that such trains come in families parametrized by constants such as the wave speed and amplitude.  Another way to formulate the stability question is to consider an initial wave form that locally resembles a regular wave train near each point, but for which the parameters of the wave vary slowly compared with the wave length.  The study of the dynamics of such slowly modulated wave trains leads to an asymptotic reduction of the original nonlinear equation \eqref{eqnlKG} to a quasilinear model system of equations, whose suitability (well-posedness) when formulated with initial conditions leads to a prediction of stability or instability.  This approach was originally developed by Whitham \cite{Wh1,Wh} and was also used by  Scott \cite{Sco2} to study the general nonlinear Klein-Gordon equation \eqref{eqnlKG}.  See also 
\cite{EFMc1,FoMc1,FoMc2,Mura1,Prk1}.   A different form of modulation theory arises if one assumes that the amplitude of the modulated waves is small; such weakly nonlinear analysis leads to models for  the dynamics that are equations of nonlinear Schr\"odinger type for which the focusing/defocusing dichotomy purports to determine stability.

%
%

The goal of this paper is to marry together the two types of stability analysis described above in the context of regular wave trains for the general nonlinear Klein-Gordon equation \eqref{eqnlKG}.
On the one hand, this requires that we generalize the spectral stability analysis we developed in \cite{JMMP1} for the special case of the sine-Gordon equation to more general potentials $V$.
While the methods we introduced in \cite{JMMP1} do not rely heavily on the complete integrability of the sine-Gordon equation, they also do not all immediately generalize for $V(u)\neq -\cos(u)$.  For this reason as well as for completeness and pedagogical clarity, we will give all details of the spectral analysis for general $V$, which the reader will find to be refreshingly classical in nature.  On the other hand, to explain the approach of wave modulation, we will review the calculations originally carried out by Whitham and Scott, and we will also explain the details of the weakly nonlinear expansion method.  Then, to connect the two approaches, we apply Evans function techniques to analyze the spectrum near the origin in the complex plane, and we find a common link with modulation theory by computing the terms in the Taylor expansion of the (Floquet) monodromy matrix for the spectral stability problem.  The main result is the introduction of a modulational stability index which determines exactly whether the spectral curves are tangent to the imaginary axis near the origin, or whether they are tangent to two lines making a nonzero angle with the imaginary axis, and consequently invading the unstable complex half plane. The latter situation corresponds to a particular type of linear instability that
turns out to correlate precisely with the formal modulational instability calculation of Whitham (ellipticity of the modulation equations) as well as to the weakly nonlinear analysis of waves near equilibrium (focusing instability).

There have been several papers recently on the subject of spectral stability properties of traveling wave solutions of nonlinear equations that, like the Klein-Gordon equation \eqref{eqnlKG}, are second-order in time.  As the reader will soon see, such problems require the analysis of eigenvalue problems in which the eigenvalue parameter does not appear in the usual linear way, and that cannot be reduced to selfadjoint form.  In particular, we draw the reader's attention to the works of Stanislavova and Stefanov \cite{StanislavovaS12,StanislavovaS13} in this regard.  One way that our paper differs from these works is that due to the periodic nature of the traveling waves we consider and our interest in determining the stability of these waves to localized perturbations (perturbations constructed as superpositions of modes with no predetermined fixed wavelength), we must deal with a spectrum that is fundamentally continuous rather than discrete.  The spectral stability properties of periodic traveling waves in first-order systems (e.g., the Korteweg-de Vries equation or the nonlinear Schr\"odinger equation) have also been studied recently by H\v{a}r\v{a}gu\c{s} and Kapitula \cite{HaKa1}.

The paper is organized as follows. In \S\ref{secstructure} we study the periodic traveling wave solutions to \eqref{eqnlKG}, which come in four types organized in terms of the energy parameter $E$ and the wave speed $c$.  A natural classification follows, and we analyze the dependence of the period on the energy $E$, as this turns out to relate to the connection between spectral stability analysis on the one hand and wave modulation theory on the other.   The next several sections of our paper concern the linearized spectral stability analysis of these traveling waves.
In \S\ref{secspectralproblem} we establish the linearized problem, we define its spectrum and show that it can be computed from a certain monodromy matrix.  Then we establish elementary properties of the spectrum and interpret it in the context of dynamical stability theory for localized perturbations.  It turns out to be useful to compare the spectrum with that of a related Hill's equation (the approach introduced by Scott \cite{Sco1}),  details of which are presented in \S\ref{secHill}. In \S\ref{sec:monodromy} we analyze the monodromy matrix in a neighborhood of the origin.  Based on these calculations, 
in \S\ref{section:stability-indices} we define two different stability indices that are capable of detecting (unstable) spectrum in the right half-plane under a certain non-degeneracy condition. One of these indices detects a special type of spectral instability that we call a \emph{modulational instability}.
 In \S\ref{secfurther} we complete the stability analysis for periodic traveling waves by showing spectral stability and instability in two remaining cases where the instability indices prove to be inconclusive.  A brief summary statement of our main results on spectral stability of periodic traveling waves for the nonlinear Klein-Gordon equation \eqref{eqnlKG} is formulated in \S\ref{section:summary}.  With the spectral stability analysis complete, in \S\ref{sec:modulation} we turn to a review of wave modulation theory both in the strongly nonlinear sense originally developed by Whitham and in the weakly nonlinear sense.  
 In \S\ref{secWhitham} we show that the analytic type of the system of quasilinear modulation equations, which is the key ingredient in Whitham's fully nonlinear theory of modulated waves and that does not appear to be directly connected with linearized (spectral) stability,  is in fact determined by the modulational instability index. Then, in \S\ref{sec:NLS} we show that the focusing/defocusing type of a nonlinear Schr\"odinger equation arising in the weakly nonlinear modulation theory of
near-equilibrium waves is also determined by the same modulational instability index. Finally, in \S\ref{sec:extensions}, we discuss how some of our methods and results can be extended to other potentials more general than those that we consider in detail. 

\subsection*{On notation}  We denote complex conjugation with an asterisk (e.g., $\lambda^*$) and denote the real and imaginary parts of a complex number $\lambda$ by $\Re\lambda$ and $\Im\lambda$ respectively.
We use lowercase boldface roman font to indicate column vectors (e.g., $\bw$), and with the exception of the identity matrix $\Id$ and the Pauli matrix $\sigma_-$ we use upper case boldface roman font to indicate square matrices (e.g., $\bM$).  Elements of a matrix $\bM$ are denoted $M_{ij}$. Linear operators acting on infinite-dimensional spaces are indicated with calligraphic letters (e.g., $\cL$, $\cT$, and $\cH$). 

\section{Structure of periodic wavetrains}
\label{secstructure}
\subsection{Basic assumptions on $V$}
To facilitate a classification of different types of traveling waves in the nonlinear Klein-Gordon equation, it is convenient to make the following assumptions on the periodic potential function $V$.

\begin{assumption}
\label{assumptionsV}
The potential function $V$ satisfies the following:
\begin{enumerate}[(a)]
\item \label{assumea}$V : \R \to \R$ is a periodic function of class $C^2$; 
\item \label{assumeb} $V$ has exactly two non-degenerate critical points per period.
\end{enumerate} 
\end{assumption}
The Klein-Gordon equation \eqref{eqnlKG} involves $V'$ and not $V$ directly, so $V$ may be
augmented by an additive constant at no cost.  Moreover, by independent scalings $u\mapsto \alpha u$ and $(x,t)\mapsto (\beta x,\beta t)$ by positive constants $\alpha>0$ and $\beta>0$, we may freely modify the period and amplitude of $V$.  Therefore, \emph{without loss of generality} we will make the following additional assumptions.
\begin{assumption}
\label{assumptionsnormalize}
The potential $V$ has fundamental period $2\pi$, and 
\begin{equation}
\min_{u\in\mathbb{R}}V(u)=-1\quad\text{while}\quad\max_{u\in\mathbb{R}}V(u)=1.
\end{equation}
\end{assumption}
The maximum and minimum are therefore the only two critical values of $V$, and they are each achieved at precisely one non-degenerate critical point per period $2\pi$.

\begin{remark}
The sine-Gordon potential $V(u):=-\cos(u)$ clearly satisfies all of the above hypotheses.
\myendrmk
\end{remark}

While the above assumptions on $V$ allow for an easier exposition, many of our results also hold for more general periodic $V$, and even the assumption of periodicity of $V$ can be dropped in some cases.  Some generalizations along these lines are discussed in \S\ref{sec:extensions}.

Equation \eqref{eqnlKG} has traveling wave solutions of the form
\begin{equation}
\label{tws}
u(x,t) = f(z),\quad z:=x-ct 
\end{equation}
where $c \in \R$ is the wave speed. In what follows we shall assume that $c \neq \pm 1$. Substituting into \eqref{eqnlKG} we readily see that the profile function $f : \R \to \R$ satisfies the nonlinear ordinary differential equation
\begin{equation}
\label{eqnlp}
(c^2 -1) f_{zz} + V'(f) = 0. 
\end{equation}
One of the implications of the simple assumptions in force on $V$ is that in the phase portrait of
\eqref{eqnlp} in the $(f,f_z)$-plane, the separatrix composed of the unstable fixed points (all necessarily of saddle type) and their corresponding stable and unstable manifolds will be a (multiply) connected set.  Indeed, there are no isolated components of the separatrix, which qualitatively resembles that of the simple pendulum, consisting of a $2\pi$-periodic array of unstable fixed points on the $f_z=0$ axis and the heteroclinic orbits connecting them in pairs.  See Figure~\ref{fig:PhasePortraitPlot1}.
\begin{figure}[h]
\begin{center}
\includegraphics{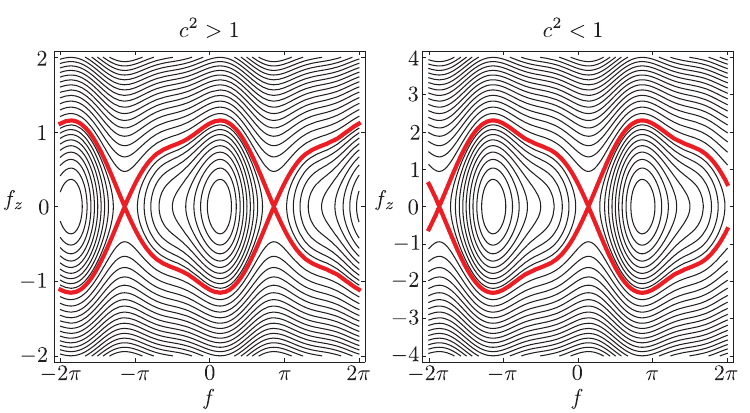}
\end{center}
\caption{Phase portraits of equation \eqref{eqnlp} for $c=2$ (left) and $c=1/2$ (right), where the potential is $V(u)=-0.861[\cos(u) +\tfrac{1}{3}\sin(2u)]$, which satisfies Assumptions~\ref{assumptionsV} and \ref{assumptionsnormalize}.  The separatrices are the thicker red curves. (Color online.)}
\label{fig:PhasePortraitPlot1}
\end{figure}

\subsection{Types of periodic traveling waves}

All non-equilibrium solutions to \eqref{eqnlp}, with the exception of 
the heteroclinic orbits in the separatrix,
are such that $f_z$ is a non-constant periodic function with some finite (fundamental) period $T > 0$; thus, we are interested in solutions $f$ to \eqref{eqnlp} satisfying $f(z + T) = f(z)$ or $f(z+T) = f(z) \pm 2\pi$. There are four distinct types of solutions to \eqref{eqnlp} for which $f_z$ is periodic. 
We will refer to all such solutions of \eqref{eqnlKG} as \emph{periodic traveling waves}.

To begin classifying the periodic traveling waves, notice that the phase portrait of equation \eqref{eqnlp} is (qualitatively) shifted according to the sign of $c^2 -1 $; indeed a change of sign of $c^2-1$ induces an exchange between the stable and unstable fixed points. Therefore, the first dichotomy concerns the wave speed.
\begin{definition}[subluminal and superluminal periodic traveling waves]
Each periodic traveling wave $f$ is of exactly one of the following two types.
\begin{itemize}
\item[(i)]
If $c^2 < 1$ then $f$ is called a \textit{subluminal} periodic traveling wave.
\item[(ii)] 
If $c^2 > 1$ then $f$ is called a \textit{superluminal} periodic traveling wave.
\end{itemize}
\end{definition}

The second dichotomy applies to both super- and subluminal waves and concerns the nature of the orbit of $f$ in the phase portrait.
\begin{definition}[librational and rotational periodic traveling waves]
Each periodic traveling wave $f$ is of exactly one of the following two types.
\begin{itemize} 
\item[(i)] If $f(z+T) = f(z)$ for all $z \in \R$ then $f$ is called a \textit{librational} wavetrain.  Its orbit in the phase plane is a closed trajectory surrounding a single critical point of $V(u)$ and it is enclosed by the separatrix.  In this case $f_z$ is a function of exactly two zeroes per fundamental period.
\item[(ii)] If $f(z+T) = f(z) \pm 2\pi$ for all $z \in \R$ then $f$ is called a \textit{rotational} wavetrain.  Its orbit in the phase plane is an open trajectory lying outside of the separatrix.  In this case $f_z$ is a function of fixed sign.  These waves are also called \textit{kink trains} (or \emph{antikink trains}, depending on the sign of $f_z$) in the literature.
\end{itemize}
\end{definition}

\begin{remark}
In the classical mechanics literature (see Goldstein \cite{Golds2ed}; see also \cite{BuM2}) the term \textit{rotation} (sometimes designated as \textit{circulation} or \textit{revolution}) is used to characterize the kind of periodic motion in which the position is not bounded, but the momentum is periodic. Increments by a period in the position produce no essential change in the state of the system. Examples of rotations are the highly-energetic motions of the simple pendulum in which the pendulum mass perpetually spins about its pivot point.
The name \textit{libration} is a term borrowed from the astronomical literature describing periodic motions in which both the position and the momentum are periodic functions with same frequency. Examples of librations are the relatively small-amplitude oscillations of the simple pendulum about
its gravitationally-stable downward equilibrium.
\myendrmk
\end{remark}

Equation \eqref{eqnlp} can be integrated once to obtain
\begin{equation}
\label{eq:waveeq}
\tfrac{1}{2}(c^2 -1) f_z^2 = E - V(f), 
\end{equation}
where $E$ is an integration constant with the interpretation of total (kinetic plus potential) energy. Given the value of $c$, it is the energy parameter $E$ that distinguishes between waves of librational and rotational types.  

\subsubsection{Superluminal librational waves} Suppose $c^2 > 1$ and $|E| < 1$. These parameter values correspond to librational motion because according to \eqref{eq:waveeq}, $f$ must be confined to a bounded interval when $|E|<1$ in order that the effective potential $\tilde{V}(f) := (V(f) - E)/(c^2 -1)$ be non-positive. 
There is exactly one such maximal interval per $2\pi$-period of $V$.  Let $[f^-,f^+]$ denote one such interval (well-defined modulo $2\pi$).  This interval contains exactly one critical point of $V$, a minimizer where $V=-1$ corresponding to a stable equilibrium of \eqref{eqnlp}, necessarily in its interior.  
We thus have a closed periodic orbit of \eqref{eqnlp} in the $(f,f_z)$-plane, an orbit that encloses a single stable equilibrium point and that crosses the $f_z=0$ axis exactly at the two points $f=f^\pm$.  Each of the closed orbits in the left-hand panel of Figure~\ref{fig:PhasePortraitPlot1} is of this type.  The amplitude $f^+-f^-$ of this solution is strictly less than $2\pi$.

The orbit is symmetric with respect to reflection in the $f_z=0$ axis, and the part of the orbit in the upper half phase plane is given by the graph
\begin{equation}\label{eq:fzsuperlib}
f_z=\frac{\sqrt{2}}{\sqrt{c^2-1}}\sqrt{E-V(f)},\quad f^-\le f\le f^+.
\end{equation}
Integrating $dz/df = (f_z)^{-1}$, with $f_z$ given in terms of $f$ by \eqref{eq:fzsuperlib}, over the interval $f^-\le f\le f^+$ gives half the value of the fundamental period $T$ of the motion.  Therefore,
\begin{equation}
\label{periodsuplib}
T = \sqrt{2} \sqrt{c^2-1} \int_{f^-}^{f^+} \frac{d\eta}{\sqrt{E-V(\eta)}}.
\end{equation}
Note that $f^-$ and $f^+$ depend on $E\in (-1,1)$ but not on $c$ with $c^2>1$.
For example, in the sine-Gordon case with the potential $V(u) = - \cos(u)$, the interval of oscillation is $[f^-,f^+]=[-\arccos(-E),\arccos(E)]\pmod{2\pi}$ and the contained stable equilibrium occurs at $f=0\pmod{2\pi}$.  

\subsubsection{Subluminal librational waves} Likewise, when $c^2 < 1$ and $|E|<1$, we have  a system of maximal intervals (complementary to the system $[f^-,f^+]\pmod{2\pi}$) of the form 
$[f^+,f^-+2\pi]\pmod{2\pi}$ on which $E \leq V(f)$ and hence effective potential $\tilde{V}(f):=(V(f)-E)/(c^2-1)$ is again non-positive, with the orbit oscillating in the interval $f^+\le f\le f^-+2\pi$ closing about a different equilibrium point, now a maximizer of $V$ where $V=1$. This is also a librational motion, and the closed orbits illustrated in the right-hand panel of Figure~\ref{fig:PhasePortraitPlot1} are of this type. The part of the orbit in the upper half phase plane is given by the graph
\begin{equation}
 f_z = \frac{\sqrt{2}}{\sqrt{1-c^2}} \sqrt{V(f)-E},\quad f^+\le f\le f^-+2\pi,
\end{equation}
and by a similar argument the fundamental period of the librational motion is given by the formula
\begin{equation}
\label{periodsublib}
T =  \sqrt{2} \sqrt{1-c^2} \int_{f^+}^{f^-+2\pi} \frac{d\eta}{\sqrt{V(\eta)-E}}.
\end{equation}
In the sine-Gordon example with $V(u)=-\cos(u)$, we have $[f^+,f^-+2\pi]=[\arccos(-E),2\pi-\arccos(-E)]\pmod{2\pi}$ and the enclosed stable equilibrium occurs at $f=\pi \pmod{2\pi}$.  

We stress that in either of the two librational wave cases, there is exactly one critical point of $V$ in the interval of oscillation of $f$. Moreover, as a fundamental period $T$ can be interpreted as the total length traversed by the parameter $z$ as $f$ goes from the initial point of this interval to the final point and back again, the critical point (equilibrium position) will be passed exactly twice per period, i.e.\@ $V'(f(z)) = 0$ exactly twice per period $T$.  The orbit of each librational wave in the $(f,f_z)-$plane of the phase portrait of equation \eqref{eqnlp} is symmetric with respect to the $f$-axis. Lastly, 
$f_z = 0$ exactly when $f$ coincides with one of the endpoints of its interval of oscillation.
That is, the orbit of the wave in the phase plane crosses the $f$-axis at exactly these two points.

\subsubsection{Superluminal rotational waves} When $c^2 > 1$ and $E > 1$, the effective potential $\tilde{V}(f):=(V(f)-E)/(c^2-1)$ is strictly negative for all $f\in\mathbb{R}$.
Hence, according to \eqref{eq:waveeq}, $f$ is not confined to any bounded interval, and $f_z$ has a fixed sign.  These features show that the corresponding motion is of rotational type, and such motions correspond to the open orbits in the left-hand panel of Figure~\ref{fig:PhasePortraitPlot1}.  The ``period'' of this motion is defined as the smallest positive value of $T$ for which $f(z+T)=f(z)\pm 2\pi$ holds for all $z\in\mathbb{R}$ (the choice of sign corresponds to the fixed sign of $f_z$).
Integrating $dz/df$ with respect to $f$ over the representative interval $0\le f\le 2\pi$ then gives
the period as
\begin{equation}
\label{periodsuprot}
T = \frac{\sqrt{c^2-1}}{\sqrt{2}} \int_0^{2\pi} \frac{d\eta}{\sqrt{E-V(\eta)}}. 
\end{equation}

\subsubsection{Subluminal rotational waves} If $c^2 < 1$ and $E < -1$ the effective potential $\tilde{V}(f):=(V(f)-E)/(c^2-1)$ is again always strictly negative for all $f\in\mathbb{R}$, so again
$f_z$ never vanishes and hence has a fixed sign corresponding to rotational motion in which $f$ changes by $2\pi$ over a fundamental ``period''. The period in this case is given by
\begin{equation}
\label{periodsubrot}
T = \frac{\sqrt{1-c^2}}{\sqrt{2}} \int_0^{2\pi} \frac{d\eta}{\sqrt{V(\eta) - E}}. 
\end{equation}
The open orbits illustrated in the right-hand panel of Figure~\ref{fig:PhasePortraitPlot1} correspond to subluminal rotational waves.

\subsubsection{Classification scheme}
We present the following parametric classification of periodic traveling wave solutions of \eqref{eqnlKG}.
\begin{definition}
\label{defregionsR}
The open set of $(E,c)$ such that equation \eqref{eqnlKG} has at least one periodic traveling wave solution of one of the four types described above is denoted $\region \subset \R^2$. We distinguish the following four disjoint open subsets:
\begin{align*}
\region_<^\mathrm{lib} &= \{ c^2 < 1, \;|E|<1\}, & \text{(subluminal librational)},\\
\region_<^\mathrm{rot} &= \{ c^2 < 1, \; E < -1 \}, & \text{(subluminal rotational)},\\
\region_>^\mathrm{lib} &= \{ c^2 > 1, \; |E|<1 \}, & \text{(superluminal librational)},\\
\region_>^\mathrm{rot} &= \{ c^2 > 1, \; E > 1 \}, & \text{(superluminal rotational)},
\end{align*}
corresponding to each of the four types of periodic waves and such that 
$\region=\region_<^\mathrm{lib}\cup\region_<^\mathrm{rot}\cup\region_>^\mathrm{lib}\cup\region_>^\mathrm{rot}$.
\end{definition}
Due to translation invariance of the Klein-Gordon equation \eqref{eqnlKG}, i.e., invariance under $x\mapsto x-x_0$, any translate in $z$ of a periodic traveling wave solution is again a periodic traveling wave solution corresponding to exactly the same values of the parameters $(E,c)$.  Another symmetry of the periodic traveling wave solutions corresponding to given values of $(E,c)$ is the involution $z\mapsto -z$.  For librational  waves, the involution group is a discrete subgroup of the translation group, but for rotational waves it is an independent symmetry.  There are no other symmetries of the periodic traveling waves for fixed $(E,c)$.  

It will be convenient for us to single out a unique solution $f(z)$ of the differential equation \eqref{eq:waveeq} for each $(E,c)\in\region$.  We do this by fixing the initial conditions at $z=0$ as follows.  Fix a minimizer $\underline{u}$ and a maximizer $\overline{u}$ of $V$.  Then for subluminal waves, we determine $f$ given $(E,c)\in\region$ by
\begin{equation}
f(0)=\overline{u},\quad f_z(0)>0,\quad (E,c)\in \region_<^\mathrm{lib}\cup\region_<^\mathrm{rot},
\label{eq:sub-normalize}
\end{equation}
while for superluminal waves,
\begin{equation}
f(0)=\underline{u},\quad f_z(0)>0,\quad (E,c)\in\region_>^\mathrm{lib}\cup\region_>^\mathrm{rot}.
\label{eq:super-normalize}
\end{equation}
With these conditions, the periodic traveling wave $f(z)=f(z;E,c)$ is well-defined for $(E,c)\in\region$.
In the case of rotational waves, we are selecting a particular kink train (antikink trains are obtained by the involution symmetry).

\begin{lemma}
\label{lemfC2E}
For each $z \in \R$, $f(z;E,c)$ is a function of class $C^2$ of $(E,c) \in \region$. 
\end{lemma}
\begin{proof}
First, suppose that $f$ is a rotational wave, in which case $f=f(z;E,c)$
may be obtained for all $z\in\mathbb{R}$  by inverting the relation
\begin{equation}
 z = \frac{\sqrt{|c^2 -1|}}{\sqrt{2}} \int_{f(0)}^{f} \frac{d \eta}{\sqrt{|E - V(\eta)|}}
\end{equation}
where either $f(0)=\underline{u}$ or $f(0)=\overline{u}$.  The right-hand side is a $C^3$ function of $(f,E,c)$,
and its derivative with respect to $f$ is strictly positive.  Hence it follows by the Implicit Function Theorem that $f$ can be solved for uniquely as a function of class $C^3$ in $(E,c)$.

For librational waves, we modify the argument as follows.  There is a maximal open interval containing $z=0$ on which $f$ is strictly increasing, and for $z$ in this interval exactly the same argument given above for rotational waves applies.  Adding half the period $T$ produces another open interval of $z$-values for which $f$ is strictly decreasing, a case again handled by a simple variation of the same argument.  By periodicity of $f$ with period $T$ it remains to consider the values of $z$ for which $f_z=0$.  But for such $z$ we have $V(f)=E$ while $V'(f)\neq 0$ and we see by the Implicit Function Theorem that $f$ is a $C^2$ function of $E$ (and is independent of $c$).  This completes the proof of the Lemma.
\end{proof}

Note that given the assumption that $V$ is twice continuously differentiable, the only obstruction to $f$ being a  class $C^3$ function of $(E,c)\in\region$ occurs for librational waves and comes from the points $z$ where $f_z=0$.

\begin{remark}
\label{remark:infinite-speed}
A degenerate case of periodic traveling waves that is also of physical importance is
the limiting case of waves of infinite velocity, that is, solutions $u=f$ of the Klein-Gordon equation \eqref{eqnlKG}
that are independent of $x$ and periodic (modulo $2\pi$) functions of $t$.  Substituting $u(x,t)=f(t)$ into \eqref{eqnlKG}, multiplying by $f_t$ and integrating yields the energy conservation law:
\begin{equation}
\tfrac{1}{2}f_t^2=E-V(f).
\label{eq:pendulum}
\end{equation}
Comparing with \eqref{eq:waveeq}, the solutions of \eqref{eq:pendulum} coincide with the superluminal waves for $c^2=2$, but considered as functions of $t$ rather than the galilean variable $z=x-ct$.  Both librational ($|E| <1$) and rotational ($E>1$) motions are possible.  
\myendrmk
\end{remark}

\subsection{Monotonicity of the period $T$ with respect to energy $E$}
The period $T$ of each type of periodic traveling wave solution of \eqref{eqnlKG} depends on
both $c$ and $E$.  The dependence on $c$ is always via multiplication by the factor $\sqrt{|c^2-1|}$, but the dependence on $E$ is nontrivial.  By straightforward differentiation of \eqref{periodsuprot} and \eqref{periodsubrot},
\begin{equation}
T_E=\frac{\partial T}{\partial E}=\frac{1-c^2}{2\sqrt{2}\sqrt{|c^2-1|}}\int_0^{2\pi}\frac{d\eta}{|V(\eta)-E|^{3/2}},\quad \text{for $f$ rotational},
\end{equation}
which proves the following:
\begin{proposition}
\label{prop:rotational-monotone}
For rotational waves, the period $T$ is a strictly monotone function of the energy, and
$(c^2-1)T_E<0$.
\end{proposition}

In the case of librational waves, the limits of integration in the formulae \eqref{periodsuplib} and \eqref{periodsublib} for $T$ (the endpoints of the interval of oscillation for $f$) depend on $E$,  
and it would appear natural to try to calculate $T_E$ via Leibniz' rule; however the integrand is singular at the endpoints and hence the rule does not apply in this context.  The problem, however, is not merely technical; for librational waves the assumptions in force on $V$ are simply insufficient to deduce the sign of $T_E$.  In general, it is possible for the sign of $T_E$ to change one or more times within the librational energy interval $|E|<1$.

We now give a condition for which the generic case $T_E \neq 0$ holds. In \cite{Chi1}, Chicone studied classical Hamiltonian flows in the plane for which a family of closed orbits surrounds an equilibrium point, and obtained a condition on the potential sufficient to ensure monotonicity of the period with respect to the energy for the orbit family.   More precisely,
he studied energy conservation laws of the form 
\begin{equation}
\frac{1}{2}\left(\frac{dx}{dt}\right)^2 + \tilde{V}(x)=\tilde{E}
\label{eq:ChiconeODE}
\end{equation}
where $\tilde{V}$ is a potential function having a non-degenerate local minimum at some point $x^0\in\mathbb{R}$.  Consider a periodic solution $x=x(t)$ close enough to the stable equilibrium $x=x^0$ that there are no other critical points of $\tilde{V}$ besides $x^0$ for $x$ in the interval $[x^-,x^+]$ over which $x(t)$ oscillates.  Chicone proved that if the function 
\begin{multline}
N(x):=6[\tilde{V}(x)-\tilde{V}(x^0)]\tilde{V}''(x)^2-3\tilde{V}'(x)^2\tilde{V}''(x)\\
{}-2[\tilde{V}(x)-\tilde{V}(x^0)]\tilde{V}'(x)\tilde{V}'''(x)
\end{multline}
is not identically zero and sign semi-definite for all $x\in [x^-,x^+]$, then the fundamental period $T$ of the periodic orbit $x(t)$ is a monotone function of $\tilde{E}$, and $T_{\tilde{E}}$ has the same sign as $N$.  Chicone's condition is equivalent to the ratio $[\tilde{V}(x)-\tilde{V}(x^0)]/V'(x)^2$ being either semi-concave or semi-convex on the interval $[x^-,x^+]$.

Comparing \eqref{eq:ChiconeODE} with \eqref{eq:waveeq} and applying Chicone's criterion to our problem, we may obtain sufficient conditions on the Klein-Gordon potential $V$ to guarantee that $T_E$ is nonzero for all librational waves.  Since
$V$ has two critical values, each of which can correspond to a family of stable equilibria depending on the sign of $c^2-1$, we require two different conditions on $V$.  Define
\begin{equation}
N^\pm(f):=6[V(f)\pm 1]V''(f)^2-3V'(f)^2V''(f)-2[V(f)\pm 1]V'(f)V'''(f).
\end{equation}
Then Chicone's theory implies the following:  
\begin{proposition}[Chicone's criterion]
Suppose that $V:\mathbb{R}\to\mathbb{R}$ satisfies the conditions of Assumptions~\ref{assumptionsV} and \ref{assumptionsnormalize}.  If also $V$ is
of class $C^3$ and the functions $N^\pm:\R\to\R$ are both not identically zero and 
semidefinite, then $T_E\neq 0$ holds for all librational waves, and the sign of $T_E$ coincides with that of $N^+$ (resp., $N^-$) for superluminal (resp., subluminal) waves.
\label{prop:Chicone}
\end{proposition}

\begin{remark}
\label{remark:Chicone-sine-Gordon}
For the sine-Gordon potential, $V(u)=-\cos(u)$,
\begin{equation}
N^+(f)=4(2-\cos f)\sin^4(\tfrac{1}{2}f)\quad\text{and}\quad N^-(f)=-4(2+\cos f)\cos^4(\tfrac{1}{2}f)
\end{equation}
so $N^+(f)\ge 0$ and $N^-(f)\le 0$ for all $f\in\mathbb{R}$.  Therefore, by Chicone's criterion, $T$ is strictly increasing (resp., decreasing) in $E$ for superluminal (resp., subluminal) librational waves, i.e., $(c^2-1)T_E>0$ holds for all librational waves.
\myendrmk
\end{remark}

While Chicone's criterion is sufficient for monotonicity, we make no claim that it is necessary.  Another condition that is \emph{equivalent} to monotonicity of $T$ with respect to $E$ and that can be checked given just the wave profile $f$ itself will be developed in \S\ref{sec:alternate} (cf.\@ equations \eqref{eq:delta-analytic}, \eqref{defDeltah}, and \eqref{eq:delta-identity}).  

\begin{remark}
The period $T$ blows up to $+ \infty$ as $E \to \sgn(c^2-1)$ (the value of the separatrix). This blowup is consistent with  
Proposition~\ref{prop:rotational-monotone} for the rotational orbits on one side of the separatrix.
However, it also guarantees that there always exist librational orbits close to the separatrix for which $T$ is indeed monotone in $E$ and moreover, for which $(c^2-1)T_E>0$.  
As will be seen, the stability analysis for librational waves is most conclusive in the case when $(c^2-1)T_E>0$.
\label{remark:TE-near-separatrix}
\myendrmk
\end{remark}

It will also be convenient later for us to have available information about the sign of $T_E$ for librational waves in the limiting case of near-equilibrium oscillations.  In this direction, we have the following:
\begin{proposition}
\label{proposition-near-equilibrium}
For a family of librational periodic traveling wave solutions of the Klein-Gordon equation \eqref{eqnlKG} for which $f$ oscillates about a non-degenerate equilibrium point $u^0$ near which $V$ has four continuous derivatives, 
\begin{equation}
\sgn\left((c^2-1)\left.T_E\right|_{E=V(u^0)}\right)=\sgn\left(5V'''(u^0)^2-3V''(u^0)V^{(4)}(u^0)\right).
\label{eq:TE-near-equilibrium}
\end{equation}
\end{proposition}
In the setting of potentials $V$ satisfying Assumptions~\ref{assumptionsV} and \ref{assumptionsnormalize} we may choose $u^0=\underline{u}$ with $V(u^0)=-1$ for $c^2>1$ and $u^0=\overline{u}$ with $V(u^0)=1$ for $c^2<1$.  
\begin{proof}
We suppose for simplicity that $V$ is analytic near an equilibrium corresponding to one of the non-degenerate critical points of $V$.  Consider first the superluminal case, in which case $f$ oscillates about a critical point $u^0=\underline{u}$ for which $V'(\underline{u})=0$ but $V''(\underline{u})>0$ ($\underline{u}$ is a local minimizer of $V$).  We use Chicone's substitution \cite{Chi1} $\eta=\underline{m}(s)$ to write the formula \eqref{periodsuplib} in the form
\begin{equation}
T=\sqrt{2}\sqrt{c^2-1}\int_{-1}^1\frac{\underline{m}'(w\sqrt{E-V(\underline{u})})\,dw}{\sqrt{1-w^2}},
\label{eq:TsuplibChicone}
\end{equation}
where $\underline{m}$ is the analytic and monotone increasing function of $s$ determined uniquely from the relation
\begin{equation}
V(\underline{m})-V(\underline{u})=s^2.
\label{eq:superluminal-Chicone}
\end{equation}
Note that the function $\underline{m}$ is independent of $E$, but $\underline{m}'$ is evaluated at the scaled argument
$s=w\sqrt{E-V(\underline{u})}$ in \eqref{eq:TsuplibChicone}.  Now if $E-V(\underline{u})$ is sufficiently small, the Taylor series for $\underline{m}'(w\sqrt{E-V(\underline{u})})$ about $w=0$ converges uniformly for $|w|\le 1$, so it may be integrated term-by-term to yield an infinite-series formula for $T$:
\begin{equation}
\begin{split}
T&=\sqrt{2}\sqrt{c^2-1}\sum_{n=0}^\infty\left[\frac{\underline{m}^{(2n+1)}(0)}{(2n)!}\int_{-1}^1\frac{w^{2n}\,dw}{\sqrt{1-w^2}}\right](E-V(\underline{u}))^n \\&= \pi\sqrt{2}\sqrt{c^2-1}\left[\underline{m}'(0)+\sum_{n=1}^\infty
\frac{(2n-1)!\underline{m}^{(2n+1)}(0)}{2^{2n-1}n!(n-1)!(2n)!}(E-V(\underline{u}))^n\right].
\end{split}
\end{equation}
Evidently, $T$ is an analytic function of $E$ at $E=V(\underline{u})$.  Its derivative at the equilibrium is therefore 
\begin{equation}
\left.T_E\right|_{E=V(\underline{u})}=\frac{\pi}{4}\sqrt{2}\sqrt{c^2-1}\underline{m}'''(0).
\end{equation}
It remains to calculate $\underline{m}'''(0)$, a task easily accomplished by repeated implicit differentiation of \eqref{eq:superluminal-Chicone} with respect to $s$, and then setting $s=0$ and $\underline{m}=\underline{u}$.  In this way, we obtain
\begin{equation}
\underline{m}'''(0)=\sqrt{\frac{V''(\underline{u})}{2}}\frac{5V'''(\underline{u})^2-3V''(\underline{u})V^{(4)}(\underline{u})}{3V''(\underline{u})^4},
\end{equation}
from which \eqref{eq:TE-near-equilibrium} is obvious in the superluminal case $c^2>1$.

In the subluminal case, one considers instead oscillations $f$ about a critical point $u^0=\overline{u}$ for which $V'(\overline{u})=0$ but $V''(\overline{u})<0$ (a local maximizer of $V$).  Now letting $\overline{m}$ be the analytic and monotone increasing function of $s$ near $s=0$ defined by the equation
$
V(\overline{u})-V(\overline{m})=s^2,
$
the relevant formula \eqref{periodsublib} becomes
\begin{equation}
T=\sqrt{2}\sqrt{1-c^2}\int_{-1}^1\frac{\overline{m}'(w\sqrt{V(\overline{u})-E})\,dw}{\sqrt{1-w^2}}.
\end{equation}
Following the same procedure as in the superluminal case 
establishes \eqref{eq:TE-near-equilibrium} in the subluminal case as well.
\end{proof}

\section{The spectral problem}
\label{secspectralproblem}

\subsection{Definition of resolvent set and spectrum}

Let us now consider how a perturbation of the periodic traveling wave $f = f(z)$ evolves under the Klein-Gordon equation \eqref{eqnlKG}. Substituting $u=f+v$ into the Klein-Gordon equation \eqref{eqnlKG} written in the galilean frame associated with the independent variables $(z=x-ct,t)$ and using the equation \eqref{eqnlp} satisfied by $f$, one finds that the perturbation $v$ necessarily satisfies the nonlinear equation
\begin{equation}
 v_{tt} - 2 c v_{zt} + (c^2 -1) v_{zz} + V'(f(z)+v) - V'(f(z)) = 0.
 \label{eq:perturbexact}
\end{equation}
As a leading approximation for small perturbations, we replace \eqref{eq:perturbexact} by its linearization about $v=0$ and hence obtain the linear equation
\begin{equation}
\label{linearperturb}
v_{tt} - 2cv_{zt} + (c^2 -1) v_{zz} + V''(f(z))v = 0.
\end{equation}

Since $f$ depends on $z$ but not $t$, the equation \eqref{linearperturb} admits treatment by separation of variables, which leads naturally to a spectral problem.  Seeking particular solutions of  \eqref{linearperturb}  of the form $v(z,t) = w(z)e^{\lambda t}$, where $\lambda \in \C$, $w$ satisfies the linear ordinary differential equation
\begin{equation}
\label{eq:spectral}
(c^2 -1) w_{zz} - 2c\lambda w_z + (\lambda^2 + V''(f(z))) w = 0, 
\end{equation}
in which the complex growth rate $\lambda$ appears as a (spectral) parameter. Equation \eqref{eq:spectral} will only have a nonzero solution $w$ in a given Banach space $X$ for certain $\lambda\in\mathbb{C}$, and roughly speaking, these values of $\lambda$ constitute the spectrum for the linearized problem \eqref{linearperturb}.  A necessary condition for the stability of $f$ is
that there are no points of spectrum with $\Re\lambda>0$ (which would imply the existence of
a solution $v$ of \eqref{linearperturb} that lies in $X$ as a function of $z$ and grows exponentially in time). These notions will be made precise shortly.

Following Alexander, Gardner and Jones \cite{AGJ}, the spectral problem \eqref{eq:spectral} with $w \in X$ can be equivalently regarded as a first order system of the form
\begin{equation}
\label{eq:firstorder}
\bw_z = \bA(z,\lambda) \bw,
\end{equation}
where
$\bw := (
w, w_z )^\top \in Y
$
($Y$ is a Banach space related to $X$), and 
\begin{equation}
\label{eq:coeffA}
\bA(z,\lambda) := \left(\begin{matrix}
0 & 1 \\ & \\ - \displaystyle{\frac{(\lambda^2 + V''(f(z)))}{c^2 -1 }} & \displaystyle{\frac{2c\lambda}{c^2 - 1}}
\end{matrix}\right).
\end{equation}
Note that the coefficient matrix $\bA$ is periodic in $z$ with period $T$.

To interpret the problem \eqref{eq:firstorder} in a more functional analytic setting, we consider the closed, densely defined operators $\cT(\lambda) : D \subset Y \to Y$ whose
action is defined by
\begin{equation}
\label{eq:defofT}
\cT(\lambda)\bw := \bw_z - \bA(z,\lambda) \bw
\end{equation}
on a domain $D$ dense in $Y$.
The family of operators is parametrized by $\lambda \in \C$, but the domain $D$ is taken to be independent of $\lambda\in\C$. The resolvent set and spectrum associated with $\cT$ are then defined as follows \cite{San,SS6}.

\begin{definition}[resolvent set and spectrum of $\cT$]\label{defspect}
We define the following subsets of the complex $\lambda$-plane:
\begin{itemize}
\item[(i)]
the \emph{resolvent set} $\zeta$ is defined by
\[
\zeta := \{ \lambda \in \C\,: \;\text{$\cT(\lambda)$ is
one-to-one and onto, and $\cT(\lambda)^{-1}$ is
bounded} \};
\]
\item[(ii)]
the \emph{point
spectrum} $\ptsp$ is defined by
\[
\ptsp := \{ \lambda \in \C\,:
\; \text{$\cT(\lambda)$ is Fredholm with zero index and has a
non-trivial kernel} \};
\]
\item[(iii)]
the \emph{essential spectrum} $\ess$ is defined by
\[
\ess := \{
\lambda \in \C\,: \; \text{$\cT(\lambda)$ is either not Fredholm or has index different from zero} \}.
\]
\end{itemize}
The \emph{spectrum} $\sigma$ of $\cT$ is the (disjoint) union of the essential and point spectra, $\sigma = \ess \cup \ptsp$. Note
that since $\cT(\lambda)$ is closed for each $\lambda\in\C$, then $\zeta = \C
\backslash \sigma$ (see, e.g., \cite{Kat1}).
\end{definition}
As it is well-known, the definition of spectra and resolvent associated with periodic waves depends upon the choice of the function space $Y$. Motivated by the fact that the sine-Gordon equation is well-posed in $L^p$ spaces \cite{BuM1}, here we shall consider
\begin{equation}
D = H^1(\R;\C^2) \quad \text{and}\quad Y = L^2(\R;\C^2), 
\end{equation}
which corresponds to studying spectral stability of periodic waves with respect to \emph{spatially localized perturbations} in the galilean frame in which the waves are stationary.

\begin{remark}
Observe that the parameter $\lambda$ appears in \eqref{eq:spectral} both quadratically and linearly, making this spectral problem a non-standard one (that is, it does not have the form $\cL w = \lambda w$, where $\cL$ is a differential operator in a Banach space). Instead, the family of operators $\cT(\lambda)$ constitutes a so-called 
\emph{quadratic pencil} in the spectral parameter $\lambda$.
By introducing auxiliary fields it is possible to reformulate the spectral problem $\cT(\lambda)\bw=0$ as a spectral problem of standard form on an appropriate tensor product space, but not in such a way that the latter problem is selfadjoint (we will not take this approach).  
While it is clear that Definition~\ref{defspect} reduces in the special case
when $\cT(\lambda)$ has the form $\cL-\lambda$ to a more standard definition (in particular $\ptsp$ is the set of eigenvalues of $\cL$), it is sufficiently general to handle the quadratic pencil defined by \eqref{eq:defofT} of interest here.
\myendrmk
\end{remark}

It is well-known (\cite{Grd1}, see also \S\ref{section-Hill-spectrum}) that the $L^2$ spectrum of a differential operator with periodic coefficients contains no eigenvalues.  This fact persists for the quadratic pencil $\cT$ under the generalized definition of spectra considered here.   Indeed, we have the following.
\begin{lemma}
\label{lem:allcontinuous}
All $L^2$ spectrum of $\cT$ defined by \eqref{eq:defofT} is purely essential, that is, $\sigma = \ess$ and $\ptsp$ is empty. 
\end{lemma}
\begin{proof}
Let $\lambda \in \ptsp$. Then by definition, $\cT(\lambda)$ is Fredholm with zero index and has a non-trivial kernel. This implies that $N := \ker \cT(\lambda) \subset H^1(\R;\C^2)$ is a finite dimensional Hilbert space. Let us denote $\cS : L^2 \to L^2$ as the (unitary) shift operator with period $T$, defined as $\cS \bw(z) := \bw(z+T)$. Since the coefficient matrix $\bA(z,\lambda)$ is periodic with period $T$, there holds $\cS\cT(\lambda) = \cT(\lambda)\cS$ in $L^2$, making $N$ an invariant subspace of $\cS$. Let us define $\hat \cS$ as the restriction of $\cS$ to $N$. Then $\hat\cS :N \to N$ is a unitary map in a finite-dimensional Hilbert space. Therefore, $\hat{\cS}$ must have an eigenvalue $\alpha \in \C$ such that $\cS \bw^0 = \alpha \bw^0$ for some $\bw^0 \in N\subset L^2$, $\bw^0 \neq 0$. Since $\hat{\cS}$ is unitary, we have that $|\alpha| = 1$, whence 
\begin{equation}
 |\bw^0(z+T)| = |(\hat{\cS}\bw^0 )(z)| =|\alpha \bw^0(z)| = |\bw^0(z)|
\end{equation}
(here $|\cdot|$ means the Euclidean norm in $\C^2$), that is, $|\bw^0(z)|^2$ is $T$-periodic. But since $\bw^0\neq 0$, this is a contradiction with $\bw^0 \in L^2$. Thus, $\ptsp$ is empty and $\sigma=\ess$.
\end{proof}

\subsection{Floquet characterization of the spectrum.  The periodic Evans function}
\label{secEvans}
Let $\bF(z,\lambda)$ denote the $2\times 2$ identity-normalized fundamental solution matrix for the differential equation \eqref{eq:firstorder}, i.e., the unique solution of
\begin{equation}
\bF_z(z,\lambda)=\bA(z,\lambda)\bF(z,\lambda),\quad\text{with initial condition $\bF(0,\lambda)=\mathbb{I}$, $\forall\lambda\in\C$}.
\end{equation}
The $T$-periodicity in $z$ of the coefficient matrix $\bA$ then implies that
\begin{equation}
\bF(z+T,\lambda)=\bF(z,\lambda)\bM(\lambda),\quad\forall z\in\R, \quad \text{where} \quad \bM(\lambda) := \bF(T,\lambda). 
\label{eq:monodromy}
\end{equation}
The matrix $\bM(\lambda)$ is called the \emph{monodromy matrix} for the first-order system \eqref{eq:firstorder}.
The monodromy matrix is really a representation of the linear mapping taking a given solution 
$\bw(z,\lambda)$ evaluated for $z=0\pmod{T}$ to its value one period later.
Since the elements of the coefficient matrix $\bA$ are entire functions of $\lambda$, and since the Picard iterates for $\bF(z,\lambda)$ converge uniformly for bounded $z$, the elements of the monodromy matrix $\bM(\lambda)$ are also entire functions of $\lambda\in\C$.
Let $\mu(\lambda)$ denote an eigenvalue of $\bM(\lambda)$, and let $\bw^0(\lambda)\in\C^2$ denote
a corresponding (nonzero) eigenvector:  $\bM(\lambda)\bw^0(\lambda)=\mu(\lambda)\bw^0(\lambda)$.  Then $\bw(z,\lambda):=\bF(z,\lambda)\bw^0(\lambda)$ is a nontrivial solution of the first-order system \eqref{eq:firstorder} that satisfies
\begin{equation}
\begin{split}
\bw(z+T,\lambda) &=\bF(z+T,\lambda)\bw^0(\lambda) 
=\bF(z,\lambda)\bM(\lambda)\bw^0(\lambda)\quad\text{(by \eqref{eq:monodromy})}\\
&=\mu(\lambda)\bF(z,\lambda)\bw^0(\lambda) =\mu(\lambda)\bw(z,\lambda),\quad\forall z\in\mathbb{R}.
\end{split}
\end{equation}
Thus $\bw(z,\lambda)$ is a particular solution that goes into a multiple of itself upon translation by a period in $z$.  After G.\@ Floquet, such solutions are called \emph{Floquet solutions}, and the eigenvalue $\mu(\lambda)$ of the monodromy matrix $\bM(\lambda)$ is called a \emph{Floquet multiplier}.  If $R(\lambda)$ denotes any number (determined modulo $2\pi i$) for which $e^{R(\lambda)}=\mu(\lambda)$, it is evident that $e^{-R(\lambda)z/T}\bw(z,\lambda)$ is a $T$-periodic function of $z$,
or, equivalently (Bloch's Theorem) $\bw(z,\lambda)$ can be written in the form
\begin{equation}
\bw(z,\lambda)=e^{R(\lambda)z/T}\bz(z,\lambda),\quad\text{where $\bz(z+T,\lambda)=\bz(z,\lambda)$, $\forall z\in\R$}.
\label{eq:Bloch-form}
\end{equation}
The quantity $R(\lambda)$ is sometimes called a \emph{Floquet exponent}.  A further consequence of Floquet theory is that if the first-order system \eqref{eq:firstorder} has a nontrivial solution in $L^\infty(\mathbb{R},\mathbb{C}^2)$, it is necessarily a linear combination of solutions having Bloch form \eqref{eq:Bloch-form} with $R(\lambda)$ purely imaginary, that is, it is a superposition of Floquet solutions corresponding to Floquet multipliers $\mu(\lambda)$ with $|\mu(\lambda)|=1$.

The $L^2(\mathbb{R},\mathbb{C}^2)$ spectrum of $\cT$ given by \eqref{eq:defofT} is characterized in terms of the monodromy matrix as follows.
\begin{proposition}\label{prop:deteq0}
$\lambda \in \sigma$ if and only if there exists $\mu \in \C$ with $|\mu| = 1$ such that
\begin{equation}
 \label{eq:det0}
D(\lambda,\mu) := \det (\bM(\lambda) - \mu \Id) = 0,
\end{equation}
that is, at least one of the Floquet multipliers lies on the unit circle.
\end{proposition}
\begin{proof}
According to Lemma \ref{lem:allcontinuous}, $\sigma$ consists entirely of essential spectrum. Moreover, $\lambda \in \ess$ if and only if the system \eqref{eq:firstorder} admits a non-trivial, uniformly bounded solution \cite[pgs.\@ 138--140]{He}.  Any such solution is necessarily a linear combination of Floquet solutions with multipliers $\mu$ satisfying $|\mu|=1$. The condition \eqref{eq:det0} simply expresses that $\mu$ is a Floquet multiplier of the system \eqref{eq:firstorder}.  
It is not difficult to verify that $\cT(\lambda)$ has a bounded inverse provided all Floquet exponents have non-zero real part \cite[Proposition 2.1]{Grd1}.  Hence $\lambda \in \sigma = \ess$ if and only if there exists a eigenvalue of $\bM(\lambda)$, i.e., a solution of the quadratic equation \eqref{eq:det0}, of the form $\mu = e^{i\theta}$ with $\theta \in \R$.
\end{proof}
Following the foundational work of R.\ A.\ Gardner on stability of periodic waves \cite{Grd1,Grd3,Grd2} we make the following definition.
\begin{definition}[periodic Evans function]\label{def:gar}
The \textit{periodic Evans function} is the restriction of $D(\lambda,\mu)$ to the unit circle $S^1\subset\mathbb{C}$ in the second argument, which is to be regarded as a unitary parameter in this context.  Thus, for each  $\theta\in\mathbb{R}\pmod{2\pi}$, $D(\lambda,e^{i\theta})$ is an entire function of $\lambda\in\mathbb{C}$ whose (isolated) zeros are particular points of the spectrum $\sigma$.
\end{definition}

The obvious parametrization of the spectrum according to values of $\mu=e^{i\theta}\in S^1$, or equivalently $\theta\in\mathbb{R}\pmod{2\pi}$ can be made even clearer by introducing the set $\sigma_\theta$ of complex numbers $\lambda$ for which there exists a nontrivial solution of the boundary-value problem consisting of \eqref{eq:spectral} with the boundary condition
\begin{equation}\label{eq:periodicode}
\begin{pmatrix} w(T) \\ w_z(T) \end{pmatrix} =e^{i\theta} \begin{pmatrix} w(0) \\ w_z(0) \end{pmatrix} ,\quad\theta\in\mathbb{R}.
\end{equation}
Obviously the sets $\sigma_\theta$ and $\sigma_{\theta+2\pi n}$ coincide for all $n\in\Z$.
It follows from Proposition~\ref{prop:deteq0} that $\sigma$ may be written as a union of these \emph{partial spectra} as follows:
\begin{equation}
\sigma=\bigcup_{-\pi<\theta\le\pi}\sigma_\theta.
\end{equation}
Furthermore, it is clear that the set $\sigma_\theta$ is characterized as the zero set of the (entire in $\lambda$) periodic Evans function $D(\lambda,e^{i\theta})$ and hence is purely discrete.  The discrete partial spectrum $\sigma_\theta$ can therefore be detected for fixed $\theta\in\mathbb{R}$ by standard techniques based on the use of the Argument Principle.  However, the study of localized perturbations requires considering all of the partial spectra $\sigma_\theta$ at once.  

\begin{remark}
\label{remark:periodic}
Note that, in particular, the equation $D(\lambda,1)=0$ characterizes the part of the spectrum corresponding to perturbations that are $T$-periodic, and hence $\sigma_0$ is the \emph{periodic partial spectrum}.  Likewise, the equation $D(\lambda,-1)=0$ determines the \emph{antiperiodic partial spectrum} $\sigma_{\pi}$ corresponding to perturbations that change sign after translation by $T$ in $z$ (and hence that have fundamental period $2T$).  The points of $\sigma_0$ (resp., of $\sigma_\pi$) are frequently called \emph{periodic eigenvalues} (resp., \emph{antiperiodic eigenvalues}) although we stress that in neither case do the corresponding nontrivial solutions of \eqref{eq:firstorder} belong to $L^2(\R)$.
\myendrmk
\end{remark}

The real angle parameter $\theta$ is typically a local coordinate for the spectrum $\sigma$ as a
real subvariety of the complex $\lambda$-plane.  This explains the intuition that the $L^2$ spectrum is purely ``continuous'', and gives rise to the notion of \emph{curves of spectrum}:
\begin{proposition}
Suppose that $\lambda_0\in\sigma$ corresponding to a Floquet multiplier $\mu_0\in S^1$, and suppose that $D_\lambda(\lambda_0,\mu_0)\neq 0$ and $D_\mu(\lambda_0,\mu_0)\neq 0$.
Then there is a complex neighborhood $U$ of $\lambda_0$ such that $\sigma\cap U$ is
a smooth curve through $\lambda_0$.
\label{prop:curves}
\end{proposition}
\begin{proof}
Since $D_\lambda(\lambda_0,\mu_0)\neq 0$, it follows from the Analytic Implicit Function Theorem that
the characteristic equation $D(\lambda,\mu)=0$ may be solved locally for $\lambda$ as an analytic function $\lambda=l(\mu)$ of $\mu\in\mathbb{C}$ near $\mu=\mu_0=e^{i\theta_0}$ with $l(\mu_0)=\lambda_0$. The spectrum near $\lambda_0$ is therefore the image of the map $l$ restricted to the unit circle near $\mu_0$, that is, $\lambda=l(e^{i\theta})$ for $\theta\in\R$ near $\theta_0$.  But then $D_\mu(\lambda_0,\mu_0)\neq 0$ implies that $dl(e^{i\theta})/d\theta\neq 0$
at $\theta=\theta_0$, which shows that the parametrization is regular, i.e., the image is a smooth curve (in fact, an analytic arc) passing through the point $\lambda_0$.
\end{proof}

\begin{remark}
\label{remark:normal-forms}
Points of the spectrum $\sigma$ where at least one of the two first-order partial derivatives of $D(\lambda,\mu)$ vanishes correspond to singularities of the system of spectral arcs.  The nature of a given singularity can be characterized by a normal form obtained from the germ of $D(\lambda,\mu)$ near the singularity.  For example, if in suitable local coordinates $\tilde{\lambda}\in\mathbb{C}$ and $\tilde{\theta}\in\mathbb{R}$, the equation $D(\lambda,e^{i\theta})=0$ takes the normal form $\tilde{\lambda}^p=\tilde{\theta}$, 
 then $\sigma\cap U$ will consist of exactly $p$ analytic arcs crossing at $\lambda_0$ separated by equal angles.  
 It turns out that the origin $\lambda=0$ is a point of $\sigma$ that requires such specialized analysis, and we will carry this out in detail in \S\ref{secmodinst}.  (See in particular Remark~\ref{remark:one-curve}.)
\myendrmk
\end{remark}

\subsection{Basic properties of the spectrum}
\subsubsection{Spectral symmetries}
The Klein-Gordon equation \eqref{eqnlKG} is a real Hamiltonian system, and this forces certain
elementary symmetries on the spectrum $\sigma$.  
\begin{proposition}
The spectrum $\sigma$ is symmetric with respect to reflection in the real and imaginary axes, i.e., if $\lambda\in\sigma$, then also $\lambda^*\in\sigma$ and $-\lambda\in\sigma$ (and hence also $-\lambda^*\in\sigma$).  
\label{prop:spectral-symmetry}
\end{proposition}
\begin{proof}
Let $\lambda\in\sigma$.  Then there exists $\theta\in\mathbb{R}$ for which $\lambda\in\sigma_\theta$, that is, there is a nonzero solution $w(z)$ of the boundary-value problem \eqref{eq:spectral} with \eqref{eq:periodicode}.  Since $V''(f(z))$ is a real-valued function, it follows by taking complex conjugates that $w(z)^*$ is a nonzero solution of the same boundary-value problem but with $e^{i\theta}$ replaced by $e^{-i\theta}$ and $\lambda$ replaced by $\lambda^*$.  It follows that $\lambda^*\in\sigma_{-\theta}\subset\sigma$.  The fact that $-\lambda\in\sigma$ will follow from Lemma~\ref{lem:floquetrelation} and Remark~\ref{rem:negsim} below.  
\end{proof}
The reflection symmetry of $\sigma$ through the imaginary axis is the one that comes from the Hamiltonian nature of the Klein-Gordon equation, and it implies that exponentially growing perturbations are always paired with exponentially decaying ones, an infinite-dimensional analogue of the fact that all unstable equilibria of a planar Hamiltonian system are saddle points.  The reflection symmetry of $\sigma$ through the real axis is a consequence of the reality of the Klein-Gordon equation and its linearization.  Sometimes systems with the four-fold symmetry of the spectrum as in this case are said to have \emph{full Hamiltonian symmetry}.  

\subsubsection{Spectral bounds}
We continue by establishing bounds on the part of the spectrum $\sigma$ disjoint from the imaginary axis.  
\begin{lemma} \label{lem:bound} 
There exists a constant $C > 0$ depending only on the wave speed $c \neq \pm 1$ and the potential $V$ such that for each $\lambda \in \sigma$ with $\Re \lambda \neq 0$,  we have:
\begin{itemize}
\item[(i)] $|\lambda| \leq C$ if $c^2<1$ (subluminal case), and
\item[(ii)] $|\Re \lambda| \leq C$ if $c^2>1$ (superluminal case).
\end{itemize}
\end{lemma} 
\begin{proof}
Suppose that $\lambda \in \sigma$, i.e., there exists $\theta=\theta(\lambda)\in\mathbb{R}$ for which there is a nontrivial solution $w(z) = w(z,\lambda)$ of the boundary-value problem \eqref{eq:spectral} with
\eqref{eq:periodicode}.
Let us denote
\begin{equation}\label{eq:innerprod}
 \langle u,v \rangle := \int_0^T u(z)^* v(z) \, dz, \qquad \|u\|^2 = \< u,u \>\ge 0,
\end{equation}
and let $M>0$ be defined by
\begin{equation}
M:=\max_{u\in\mathbb{R}}|V''(u)|.
\label{eq:M-Vpp-define}
\end{equation}

Multiplying the differential equation in \eqref{eq:spectral} 
by $w(z)^*$ and integrating over the fundamental period interval $[0,T]$ gives
\begin{equation}
\label{dos}
(c^2 -1) \<w,w_{zz}\> - 2c \lambda \<w, w_{z}\> + \lambda^2 \|w\|^2 + \<w, V''(f) w \> = 0. 
\end{equation}
Integrating by parts, taking into account the boundary conditions \eqref{eq:periodicode},
we observe that
\begin{equation}
 \<w, w_{zz}\> = - \<w_{z},w_{z}\> + w^*(T,\lambda) w_{z}(T,\lambda) - w^*(0,\lambda) w_{z}(0,\lambda) = - \|w_z\|^2 \in \R.
 \label{eq:ip-w-wzz}
\end{equation}
Moreover, $\<w, w_{z}\>$ is purely imaginary:
\begin{equation}
\begin{split}
 \Re \<w, w_z\> = \tfrac{1}{2} \int_0^T (w^* w_z+ w_z^* w) \, dz &= \tfrac{1}{2} \int_0^T \frac{d}{dz} |w|^2 \, dz \\ &= \tfrac{1}{2}(|w(T,\lambda)|^2 - |w(0,\lambda)|^2) = 0.
 \end{split}
 \label{eq:ip-w-wz-imag}
\end{equation}
Therefore, taking the imaginary part of \eqref{dos} using $\Im \lambda^2 = 2 (\Re \lambda)(\Im \lambda)$, and recalling that $\Re \lambda \neq 0$ by assumption, we obtain
\begin{equation}
\label{cuatro}
(\Im \lambda) \|w\|^2 = c\, \Im \<w, w_z\>. 
\end{equation}
Observe also that upon applying the Cauchy-Schwarz inequality to \eqref{cuatro} we get $|\Im \lambda| \|w\|^2 \leq |c| \|w\|\|w_z\|$ yielding, in turn,
\begin{equation}
 \label{cinco}
(\Im \lambda)^2 \|w\|^2 \leq c^2 \|w_z\|^2,
\end{equation}
because $\|w\|^2>0$.
Using the identities \eqref{eq:ip-w-wzz} and \eqref{eq:ip-w-wz-imag}, the real part of equation \eqref{dos} is
\begin{equation}
 \label{seis}
(1 - c^2) \|w_z\|^2 + |\lambda|^2 \|w\|^2 + \< w, V''(f) w \> = 0,
\end{equation}
where we have taken into account \eqref{cuatro} and used $\Re \lambda^2 = (\Re \lambda)^2 - (\Im \lambda)^2$.
On the other hand, we can also multiply
the differential equation \eqref{eq:spectral} by $w_z^*$ and integrate over the fundamental period $[0,T]$, which yields
\begin{equation}
\label{siete}
(c^2 -1) \< w_z, w_{zz}\> - 2c\lambda \|w_z\|^2 + \lambda^2 \<w_z,w\> + \< w_z, V''( f) w \> = 0. 
\end{equation}
According to \eqref{eq:ip-w-wz-imag}, we have $\<w_z,w\> = i \Im \<w_z,w\> = -i \Im \<w, w_z\>$,
and by a virtually identical calculation, it also holds that
$\Re \<w_z, w_{zz}\> = 0$.
Thus, taking the real part of \eqref{siete} and noticing that 
$
 \Re(\lambda^2 \<w_z,w\>) = - (\Im \lambda^2)\Im \<w_z,w\> = 2(\Re \lambda)(\Im \lambda) \Im \<w, w_z\>
$,
we obtain
\begin{equation}
2(\Re \lambda) \big( c^2 \|w_z\|^2 - (\Im \lambda)^2 \|w \|^2 \big) = c\, \Re \< w_z, V''( f) w \>,
\end{equation}
where we have multiplied by $c$ and substituted from \eqref{cuatro}.
Taking absolute values, and using the inequality \eqref{cinco}, we then obtain
\begin{equation}
\label{nueve}
2|\Re \lambda| \big( c^2 \|w_z\|^2 - (\Im \lambda)^2\|w\|^2 \big) = |c| |\Re \< w_z, V''( f) w \> |. 
\end{equation}
From the Cauchy-Schwarz inequality and $ab\le \tfrac{1}{2}a^2+\tfrac{1}{2}b^2$, an upper bound for the right hand side of \eqref{nueve} is 
\begin{equation}
\begin{split}
 |c| |\Re \< w_z, V''( f) w \> |& = |c| |\Re \< \sqrt{|c|} w_z, V''( f) \frac{w}{\sqrt{|c|}} \> | \\
& \leq |c|\cdot\left(\sqrt{|c|}\|w_z\|\right)\left(\frac{M}{\sqrt{|c|}}\|w\|\right)\\
&\leq \tfrac{1}{2}c^2 \|w_z\|^2 + \tfrac{1}{2}M \|w \|^2,
\end{split}
\end{equation}
where $M$ is defined by \eqref{eq:M-Vpp-define}.
Therefore,  \eqref{nueve} implies the inequality
\begin{equation}
\label{diez}
2|\Re \lambda| c^2 \|w_z\|^2 \leq \tfrac{1}{2}c^2 \|w_z\|^2  + \left( 2 |\Re \lambda|(\Im \lambda)^2 +\tfrac{1}{2} M\right) \|w\|^2. 
\end{equation}

To prove statement (i), assume that $c^2<1$ (subluminal case).  Then
equation \eqref{seis} implies that
\begin{equation}
 |\lambda|^2 \|w\|^2 \leq - \< w, V''(f) w \> \leq M \|w\|^2.
\end{equation}
Since $\|w\|>0$, this yields $|\lambda| \leq \sqrt{M}$, which completes the proof of statement (i).

To prove statement (ii), assume now that $c^2>1$ (superluminal case).  Solving \eqref{seis} for $\|w_z\|^2$ gives
\begin{equation}
\label{once}
c^2 \|w_z\|^2 =\left( \frac{c^2}{c^2 -1}\right) \Big( |\lambda|^2 \|w\|^2 + \<w, V''( f) w \> \Big),
\end{equation}
and since $c^2>1$, this equation immediately implies two inequalities:
\begin{equation}
c^2\|w_z\|^2\ge\frac{c^2}{c^2-1}\left(|\lambda|^2-M\right)\|w\|^2\quad\text{and}\quad
c^2\|w_z\|^2\le\frac{c^2}{c^2-1}\left(|\lambda|^2+M\right)\|w\|^2.
\end{equation}
Using the first of these inequalities to give a lower bound  for the left-hand side of  \eqref{diez}, 
and using the second to give an upper bound for the right-hand side of \eqref{diez}, yields an inequality 
which, upon dividing by $\|w\|^2>0$, involves $\lambda$ but not $w$ and can be written in the form
\begin{equation}
2|\Re\lambda|\frac{c^2}{c^2-1}\left(|\lambda|^2-M-\frac{c^2-1}{c^2}(\Im\lambda)^2\right)
\le \tfrac{1}{2}\frac{c^2}{c^2-1}\left(|\lambda|^2+M\right)+\tfrac{1}{2}M.
\label{eq:superluminal-inequality}
\end{equation}
Since for $c^2>1$ we have
\begin{equation}
|\lambda|^2-M-\frac{c^2-1}{c^2}(\Im\lambda)^2 = (\Re\lambda)^2+\frac{1}{c^2}(\Im\lambda)^2-M
\ge \frac{1}{c^2}|\lambda|^2-M,
\end{equation}
the inequality \eqref{eq:superluminal-inequality} implies also 
\begin{equation}
|\Re\lambda|\le\tfrac{1}{4}\left(c^2+\frac{Mc^4+Mc^2+(c^2-1)M}{|\lambda|^2-Mc^2}\right).
\label{eq:superluminal-inequality-II}
\end{equation}
Now, either $|\lambda|^2<2Mc^2$, in which case $|\Re\lambda|\le \sqrt{2M}|c|$, or
$|\lambda|^2\ge 2Mc^2$, in which case \eqref{eq:superluminal-inequality-II} implies
\begin{equation}
|\Re\lambda|\le\tfrac{1}{4}\left(2c^2+1+\frac{c^2-1}{c^2}\right).
\end{equation}
Hence if $c^2>1$ we always have that
\begin{equation}
|\Re\lambda|\le \max\left\{\sqrt{2M}c,\tfrac{1}{4}\left(2c^2+1+\frac{c^2-1}{c^2}\right)\right\}
\end{equation}
which completes the proof of statement (ii).
\end{proof}

\subsection{Dynamical interpretation of the spectrum}
\label{section:dynamics}
The linearized Klein-Gordon equation \eqref{linearperturb} can be treated by Laplace transforms because it has been written in a galilean frame that makes the traveling wave $f$ stationary.  The Laplace transform pair is:
\begin{equation}
\mathscr{V}(\lambda)=\int_0^\infty v(t)e^{-\lambda t}\,dt\quad\text{and}\quad v(t)=\frac{1}{2\pi i}\int_B\mathscr{V}(\lambda)e^{\lambda t}\,d\lambda
\end{equation}
where $B$ is a contour (a \emph{Bromwich path}) in the complex $\lambda$-plane that is an upward-oriented vertical path lying to the right of all singularities of $\mathscr{V}$.  Applying the Laplace transform to \eqref{linearperturb} formulated with localized initial conditions $v(\cdot,0)\in H^1(\R)$ and $v_t(\cdot,0)\in L^2(\R)$, we obtain
\begin{equation}
(c^2-1)\mathscr{V}_{zz}-2c\lambda\mathscr{V}_z + (\lambda^2+V''(f(z)))\mathscr{V}=
\lambda v(z,0)+v_t(z,0)-2cv_z(z,0).
\end{equation}
Recalling the operator $\cT(\lambda):D\subset Y\to Y$ defined by \eqref{eq:coeffA} and \eqref{eq:defofT}, this can be written as the first-order system
\begin{equation}
\cT(\lambda)\begin{pmatrix}\mathscr{V}\\\mathscr{V}_z\end{pmatrix}=\frac{1}{c^2-1}\begin{pmatrix}0\\ \lambda v(\cdot,0)+v_t(\cdot,0)-2cv_z(\cdot,0)\end{pmatrix}.
\end{equation}
The given forcing term on the right hand side lies in the space $Y=L^2(\R,\C^2)$, and therefore if $\lambda$ lies in the resolvent set $\zeta$ (see Definition~\ref{defspect}), we may solve for $\mathscr{V}$ with the help of the bounded inverse $\cT(\lambda)^{-1}$:
\begin{equation}
\begin{pmatrix}\mathscr{V}\\\mathscr{V}_z\end{pmatrix}=\frac{1}{c^2-1}\cT(\lambda)^{-1}
\begin{pmatrix}0\\\lambda v(\cdot,0) + v_t(\cdot,0)-2cv_z(\cdot,0)\end{pmatrix},\quad\lambda\in\zeta\subset\C.
\end{equation}
Since Lemma~\ref{lem:bound} shows that the spectrum has a bounded real part for all periodic traveling waves $f$, we may then recover the solution of the initial-value problem $v(z,t)$ by 
applying the inverse Laplace transform formula, because the resolvent set $\zeta$ is guaranteed to contain an appropriate Bromwich path $B$.  It is easy to show by analyzing the contour integral over $B$ that this yields a solution $v$ lying in $H^1(\R)$ for each $t>0$, that is, a spatially localized perturbation remains localized for all time.

The inverse $\cT(\lambda)^{-1}$ is not only bounded when $\lambda\in\zeta$, but also it is an analytic mapping from $\zeta\subset\C$ into the Banach space of bounded operators on $Y$ \cite{Kat1}.  Cauchy's Theorem therefore allows the Bromwich path to be deformed arbitrarily within the resolvent set $\zeta$.    To analyze the behavior of the solution $v$ for large time to determine stability, one deforms $B$ to its left, toward the imaginary axis, as far as possible.  The only obstruction is the spectrum $\sigma\subset\C$, i.e., the complement of $\zeta$.  Generically, points of $\sigma$ with $\Re\lambda>0$ will give contributions to $v$ proportional to $e^{\lambda t}$, implying exponential growth of $v$ and hence dynamical instability.  The only way such instabilities can be avoided in general is if $\sigma$ is confined to the imaginary axis. 
This discussion suggests the utility of the following definition:
\begin{definition}[spectral stability and instability]
A periodic traveling wave solution $f$ of the Klein-Gordon equation \eqref{eqnlKG} is said to be \emph{spectrally stable} if $\sigma\subset i\R$.  Otherwise (i.e., if $\sigma$ contains points $\lambda$ with $\Re\lambda \neq 0$, and hence contains points $\lambda$ with $\Re\lambda>0$ by Proposition~\ref{prop:spectral-symmetry}) $f$ is \emph{spectrally unstable}.
\label{def:spectral-instability}
\end{definition}
It is clear that spectral instability is closely related to the exponential growth of localized solutions of the linearized Klein-Gordon equation \eqref{linearperturb}.  As usual, the issue of determining whether spectral stability implies dynamical stability of the periodic traveling wave $f$ under the fully nonlinear perturbation equation \eqref{eq:perturbexact} is more subtle and will not be treated here.

\subsection{Linearization about periodic traveling waves of infinite speed}
\label{section-linearization-infinite-speed}
Recall the discussion in Remark~\ref{remark:infinite-speed} of $t$-periodic solutions (modulo $2\pi$) of the Klein-Gordon equation \eqref{eqnlKG} that are independent of $x$.  Although one can think of these solutions as arising from those parametrized by $(E,c)\in\region$ in the limit $|c|\to\infty$, this approach is not well-suited to stability analysis as it has been developed above, because the co-propagating galilean frame in which we have formulated the stability problem becomes meaningless in the limit.  

On the other hand, one can study the stability properties of the infinite-velocity solutions in a stationary frame.  Indeed, if $f=f(t)$ is such an exact solution, then the substitution of $u=f+v$ into
the Klein-Gordon equation \eqref{eqnlKG} yields the nonlinear equation
\begin{equation}
v_{tt}-v_{xx} + V'(f(t)+v)-V'(f(t))=0
\end{equation}
governing the perturbation $v$.  Linearizing about $v=0$ yields
\begin{equation}
v_{tt}-v_{xx}+V''(f(t))v=0.
\label{eq:linearized-pde-infinite-speed}
\end{equation}
Seeking a solution in $L^2(\R,\C)$ as a function of $x$ for each $t$ (for localized perturbations) leads, in view of the fact that the non-constant coefficient $V''(f(t))$ is independent of $x$, to a treatment of \eqref{eq:linearized-pde-infinite-speed} by Fourier transforms.  Letting 
\begin{equation}
\hat{v}(k)=\frac{1}{\sqrt{2\pi}}\int_\R v(x)e^{-ikx}\,dx\quad\text{and}\quad
v(x)=\frac{1}{\sqrt{2\pi}}\int_\R \hat{v}(k)e^{ikx}\,dk
\end{equation}
denote the Fourier transform pair, taking the transform in $x$ of \eqref{eq:linearized-pde-infinite-speed} yields
\begin{equation}
\hat{v}_{tt} + V''(f(t))\hat{v}=-k^2\hat{v}.
\label{eq:Hill-infinite-speed}
\end{equation}
This ordinary differential equation is an example of Hill's equation, about which more will be said in the next section.  For now, we simply formulate a definition of spectral stability for this limiting case of periodic traveling waves that is an analogue of Definition~\ref{def:spectral-instability}.
\begin{definition}[spectral stability and instability for infinite speed waves]
\label{def:stable-infinite-speed}
A nontrivial $t$-periodic solution $f$ of the Klein-Gordon equation \eqref{eqnlKG} that is independent of $x$ is \emph{spectrally stable} if all solutions of \eqref{eq:Hill-infinite-speed} are polynomially bounded in $t$ for every $k\in\R$.  Otherwise, $f$ is \emph{spectrally unstable}.
\end{definition}

\section{A related Hill's equation} 
\label{secHill}
In this section, we describe the spectrum of a related Hill's equation that plays an important role in the analysis of $\sigma$.  The method of arriving at this auxiliary equation is simple:  
if $w$ is a solution to the differential equation \eqref{eq:spectral},
then setting (Scott's transformation, \cite{Sco1})
\begin{equation}\label{eq:transform}
 y(z) = e^{-c\lambda z /(c^2-1)} w(z)
 \end{equation}
one obtains a solution $y$ of the related equation:
\begin{equation}\label{eq:hill} 
y_{zz} + P(z) y = \nu y,\quad P(z):=\frac{V''(f(z))}{c^2-1},\quad \nu=\nu(\lambda):=\left( \frac{\lambda}{c^2-1} \right)^2. 
\end{equation}
One apparent advantage is that this equation can be interpreted as an eigenvalue problem of standard form with eigenvalue $\nu$, whereas the differential equation \eqref{eq:spectral} is a quadratic pencil in $\lambda$.
Moreover, since $f(z)$ is periodic modulo $2\pi$ with period $T$, we have $P(z+T)=P(z)$ for all $z\in\R$, and hence equation \eqref{eq:hill} is an instance of \emph{Hill's equation} with potential $P$ of period $T$.  The latter equation is well-studied, hence motivating an attempt to glean information about the spectrum $\sigma$ of $\cT$ related to \eqref{eq:spectral} from knowledge of the solutions of Hill's equation \eqref{eq:hill}.

\subsection{The Floquet spectrum of Hill's equation}  
\label{section-Hill-spectrum}
A general reference for the material in this section is the book by Magnus and Winkler \cite{MW66} (see also \cite{McKMoer75}).  Let $\cH$ be the formal differential (Hill's) operator
\begin{equation}\label{eq:defH}
\cH:=\frac{d^2}{dz^2}+P(z), 
\end{equation}
Considering $\cH$ acting on a suitable dense domain in $L^2(\mathbb{R},\mathbb{C})$, Definition~\ref{defspect} applies to the associated linear pencil $\cT^\mathrm{H}(\nu):=\cH-\nu$ to define the spectrum as a subset of the complex $\nu$-plane, here denoted $\Sigma^\mathrm{H}$.  As a consequence of the fact that $\cH$ is essentially selfadjoint with respect to the $L^2(\mathbb{R},\mathbb{C})$ inner product, $\Sigma^\mathrm{H}\subset\mathbb{R}$.  By the same arguments as applied to the quadratic pencil $\cT$ (see Lemma~\ref{lem:allcontinuous} and \S\ref{secEvans}), the spectrum $\Sigma^\mathrm{H}$ contains no ($L^2$) eigenvalues and coincides with the Floquet spectrum obtained as the union of discrete partial spectra $\Sigma^\mathrm{H}_\theta$ parametrized by $\theta\in\mathbb{R}\pmod{2\pi}$.  The partial spectrum $\Sigma^\mathrm{H}_\theta$ is the set of complex numbers $\nu=\nu^{(\theta)}$ for which there exists a nontrivial solution of the boundary-value problem
\begin{equation}
\cH y(z)=\nu^{(\theta)} y(z),\quad \begin{pmatrix} y(T) \\ y_z(T) \end{pmatrix} =e^{i\theta} \begin{pmatrix} y(0) \\ y_z(0) \end{pmatrix} ,\quad\theta\in\mathbb{R}.
\label{eq:HillEV}
\end{equation}
Again, $\Sigma_\theta=\Sigma_{\theta+2\pi n}$ for all $n\in\Z$.
Since $P(z)$ is real, it is clear that $\Sigma_{-\theta}^\mathrm{H}=\Sigma_\theta^\mathrm{H}$. The spectrum $\Sigma^\mathrm{H}$ is then the union
\begin{equation}
\Sigma^\mathrm{H}:=\bigcup_{-\pi<\theta\le\pi}\Sigma_\theta^\mathrm{H} = \bigcup_{0\le\theta\le\pi}\Sigma_\theta^\mathrm{H}.
\end{equation}
The numbers $\nu^{(0)}\in\Sigma_0^\mathrm{H}$ (resp., $\nu^{(\pi)}\in\Sigma_\pi^\mathrm{H}$) are typically called the \emph{periodic eigenvalues} (resp., \emph{antiperiodic eigenvalues}) of Hill's equation. 

Although $T$ has been defined as the smallest number for which $f_z(z+T)=f_z(z)$, it need only be an integer multiple of the fundamental period of $P$.  For example, in the sine-Gordon case where $V(u)=-\cos(u)$, $T$ is equal to the fundamental period of $P$ for rotational waves, but it is \emph{twice} the fundamental period of $P$ for librational waves (and the same is true whenever some translate of $V$ is an even function).  Note, however, that if $T_0$ is the fundamental period of $P$ and $T=NT_0$ for $N\in\mathbb{Z}_+$, then the partial spectra corresponding to taking the period to be $T_0$ rather than $T$ in \eqref{eq:HillEV}, denoted $\Sigma_\theta^{\mathrm{H}0}$ for $\theta\in\mathbb{R}$, are related to the partial spectra defined by \eqref{eq:HillEV} for $T=NT_0$ by
$\Sigma_\theta^{\mathrm{H}0}=\Sigma^\mathrm{H}_{N\theta}$.
It follows upon taking unions over $\theta\in\mathbb{R}$ that 
$\Sigma^{\mathrm{H}0}=\Sigma^\mathrm{H}$.

\begin{remark}
In the study of the spectrum $\sigma$ of $\cT$ for the linearized Klein-Gordon equation, it turns out to be a useful idea to pull back the Hill's spectrum $\Sigma^\mathrm{H}$ from the complex $\nu$-plane to the complex $\lambda$-plane via the relation between these two spectral variables defined in \eqref{eq:hill}.  That is, we introduce the subset $\sigma^\mathrm{H}$ of the complex $\lambda$-plane defined by
\begin{equation}
\lambda\in\sigma^\mathrm{H} \quad\Leftrightarrow\quad \nu=\nu(\lambda):=\left(\frac{\lambda}{c^2-1}\right)^2\in\Sigma^\mathrm{H}.
\label{eq:Hill-spectrum-pull-back}
\end{equation}
Since $\lambda=0$ corresponds to $\nu=0$, and since the relationship \eqref{eq:transform} between $w$ and $y$ degenerates to the identity for $\lambda=0$, it holds that $0\in\sigma$ if and only if $0\in\sigma^\mathrm{H}$.  The two spectra $\sigma$ and $\sigma^\mathrm{H}$ are, however, certainly \emph{not} equal, although it turns out that there are relations between them that we will develop shortly.  The discrepancy between these two spectra in spite of the explicit relation \eqref{eq:transform} between \eqref{eq:spectral} and \eqref{eq:hill} appears because the mediating factor $e^{-c\lambda z/(c^2-1)}$ fails to have modulus one for $|\Re\lambda|\neq 0$ and hence can convert uniformly bounded solutions into exponentially growing or decaying solutions and vice-versa.  
The incorrect identification of $\sigma$ with $\sigma^\mathrm{H}$ was the key logical flaw in \cite{Sco1}.
\label{remark:Hill-spectrum-pull-back}
\myendrmk
\end{remark}

The set $\Sigma^\mathrm{H}\subset\mathbb{R}$ is bounded above. It consists of the union of closed intervals
\begin{equation}
\Sigma^\mathrm{H}=\bigcup_{n=0}^\infty  [\nu^{(0)}_{2n+1},\nu_{2n+2}^{(\pi)}]\cup[\nu^{(\pi)}_{2n+1},\nu_{2n}^{(0)}]
\end{equation}
where the sequences $\Sigma_0^\mathrm{H}:=\{\nu^{(0)}_{j}\}_{j=0}^\infty$ and $\Sigma_\pi^\mathrm{H}:=\{\nu^{(\pi)}_{j}\}_{j=1}^\infty$ decrease to $-\infty$ and satisfy the inequalities
\begin{equation}\label{eq:oscillate}
\cdots<\nu_4^{(\pi)}\le\nu_3^{(\pi)}<\nu_2^{(0)}\le\nu_1^{(0)}<\nu_2^{(\pi)}\le\nu_1^{(\pi)}<\nu_0^{(0)}.
\end{equation}
\begin{remark}
Each of the inequalities $\nu_{j+1}^{(0)}\le\nu_j^{(0)}$ or $\nu_{j+1}^{(\pi)}\le \nu_j^{(\pi)}$ that holds strictly indicates the presence of a \emph{gap} in the Hill's spectrum.  The generic situation is that all of the inequalities hold strictly, and hence there are an infinite number of gaps.  The theory of so-called \emph{finite-gap potentials} shows that there are deep connections with the subject of algebraic geometry that arise when one considers periodic potentials $P$ for which there are only a finite number of gaps in $\Sigma^\mathrm{H}$.  The most elementary result is that the only periodic potentials $P$ for which there are no gaps in the spectrum $\Sigma^\mathrm{H}$ are the constant potentials \cite[Theorem 7.12]{MW66}.  Moreover, if $\Sigma^\mathrm{H}$ has exactly one gap, then $P$ is necessarily a non-constant (Weierstra\ss) elliptic function, i.e., $P'(z)^2$ is a cubic polynomial in $P(z)$ (Hochstadt's Theorem \cite[Theorem 7.13]{MW66}).  In the one-gap case, Hill's equation is therefore a special case of  \emph{Lam\'e's equation}, and the elliptic function $P$ corresponds to a Riemann surface of genus $g=1$ (an elliptic curve).  More generally, Dubrovin has shown \cite{DB75} that if $\Sigma^\mathrm{H}$ has exactly $n$ gaps, then $P$ is necessarily a hyperelliptic function of genus $g=n$, and $P'(z)^2$ is a polynomial in $P(z)$ of degree exactly $2g+1$.  
\label{remark:Hill-gaps}
\myendrmk
\end{remark}

Differentiating with respect to $z$ the equation \eqref{eqnlp} satisfied by $f$, one easily verifies that the function $y(z):=f_z(z)$ is a nontrivial $T$-periodic solution of Hill's equation with $\nu=0$.  That is, $y(z)$ is a nontrivial solution of the boundary-value problem \eqref{eq:HillEV} with $\theta=0$ and $\nu=0$.
Hence one of the periodic eigenvalues $\nu^{(0)}_{j}$ coincides with $\nu=0$, and the value of $j$ is determined by oscillation theory (see \cite[Theorem 8.3.1]{CL55} or  \cite[Theorem 2.14 (Haupt's Theorem)]{MW66}\footnote{The statement of Haupt's theorem in \cite{MW66} contains some typographical errors; the second sentence of Theorem 2.14 should be corrected to read ``if $\lambda=\lambda_{2n-1}'$ or $\lambda=\lambda_{2n}'$, then $y$ has exactly $2n-1$ zeros in the half-open interval $0\le x<\pi$.''}), i.e., by the number of roots of $y(z)=f_z(z)$ per period. 
If $f$ is a librational wave, then $f_z$ has exactly two zeros per period and hence either $\nu^{(0)}_{1}=0$ or $\nu^{(0)}_2=0$.  On the other hand, if $f$ is rotational , then $f_z$ has no zeros at all and hence $\nu^{(0)}_{0}=0$. 
Therefore, we have the following dichotomy for the Hill's spectrum $\Sigma^\mathrm{H}$.
\begin{itemize}
\item 
For rotational $f$,  $\Sigma^\mathrm{H}$ is a subset of the closed negative half-line, and $\nu^{(0)}_0=0$ belongs to the spectrum.
\item
For librational $f$, the positive part of  $\Sigma^\mathrm{H}$ consists of the intervals $[\nu^{(0)}_1,\nu^{(\pi)}_2]\cup[\nu^{(\pi)}_1,\nu_0^{(0)}]$ (which may merge into a single interval if $\nu^{(\pi)}_1=\nu^{(\pi)}_2$).  It is possible that $\nu^{(0)}_1=0$, but not necessary; otherwise $\nu^{(0)}_2=0$ and $\nu^{(0)}_1>0$.  
\end{itemize}
In both cases the negative part of the Hill's spectrum is unbounded.

\begin{remark}\label{rem:specH} It turns out that for librational waves $f$, the following facts hold
true.
\begin{itemize}
\item If $(c^2-1)T_E>0$, then $\nu^{(0)}_1=0$ while $\nu^{(0)}_2<0$.  Hence an
an interval of the Hill's spectrum $\Sigma^\mathrm{H}$ abuts the origin $\nu=0$ on the right and there is a gap for small negative $\nu$.  This is the case for the sine-Gordon equation, cf.\@ Remark~\ref{remark:Chicone-sine-Gordon}.
\item If $(c^2-1)T_E<0$, then $\nu^{(0)}_2=0$ while $\nu^{(0)}_1>0$.  Hence an interval of the Hill's spectrum $\Sigma^\mathrm{H}$ abuts the origin $\nu=0$ on the left and there is a gap for small positive $\nu$.
\item If $T_E=0$, then $\nu^{(0)}_1=\nu^{(0)}_2=0$, and there is no gap in the Hill's spectrum near $\nu=0$.
\end{itemize}
The reason for this is that the sign of the slope of the \emph{Hill 
discriminant} (i.e., the trace of the monodromy matrix corresponding to \eqref{eq:hill}) at $\nu=0$ can be computed 
in terms of the product $(c^2-1)T_E$
(cf.\@ 
Remark~\ref{rem:mu0}). Letting $\Delta^\mathrm{H}(\nu)$ denote 
the Hill discriminant, the periodic partial spectrum $\Sigma_0^\mathrm{H}$ is characterized as follows:  $\Delta^\mathrm{H}(\nu) = 2$ if and only if $\nu = 
\nu_j^{(0)}$ for some $j$, and $\Delta_\nu^{\mathrm{H}}(\nu_j^{(0)}) = 0$ if and only if $\nu_{j+1}^{(0)} = \nu_j^{(0)}
$ (i.e., if and only if $\nu_j^{(0)}$ is a periodic eigenvalue of geometric multiplicity 2). 
Moreover, $-2\le\Delta^\mathrm{H}(\nu)\le 2$ is the condition defining $\Sigma^\mathrm{H}$.  Thus, the sign of the derivative $\Delta_\nu^{\mathrm{H}}(\nu_j^{(0)})$ determines on which end of a spectral gap $\nu_j^{(0)}$ must lie (or whether in fact there is no gap there at all). Taken together with  
the oscillation theory informed by the number of zeros of $f_z$ over a period, this describes precisely where the periodic eigenvalue $\nu=0$ falls in the chain of inequalities \eqref{eq:oscillate}.  See Figure~\ref{fig:hillspec}.
\myendrmk
\end{remark}
\begin{figure}[h]
\begin{center}
\includegraphics{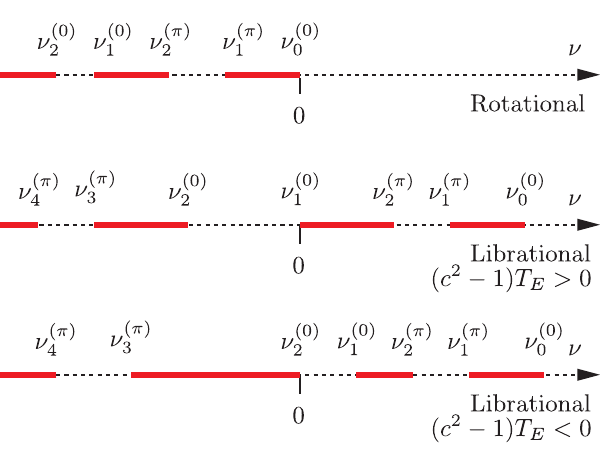}
\end{center}
\caption{A qualitative sketch of the spectrum $\Sigma^\mathrm{H}$.}
\label{fig:hillspec}
\end{figure}

As the boundary-value problem \eqref{eq:HillEV} is selfadjoint for all real $\theta$ and since the resolvent is compact (Green's function $g_\theta(z,\xi)=g_\theta(\xi,z)^*$ is a continuous and hence Hilbert-Schmidt kernel on $[0,T]^2$), it follows from the Spectral Theorem that for each $\theta\in\mathbb{R}$ the nontrivial solutions of \eqref{eq:HillEV} associated with the values of $\nu$ in the partial spectrum $\Sigma_\theta^\mathrm{H}$ form an orthogonal basis of $L^2(0,T)$.  Moreover the corresponding generalized Fourier expansion of a smooth function $u(z)$ satisfying the condition $u(z+T)=e^{i\theta}u(z)$ is uniformly convergent.  From these facts it follows that
\begin{equation}
 \langle u,\cH u\rangle\le
\|u\|^2\max\Sigma_\theta^\mathrm{H},\quad\forall u\in C^{2}(\mathbb{R}),\quad u(z+T)=e^{i\theta}u(z),
\end{equation}
where the inner product is defined in \eqref{eq:innerprod}. 
Therefore, whenever $f$ is a rotational wave,
\begin{equation}
\langle u,\cH u\rangle\le 0\quad\forall u\in C^{2}(\mathbb{R}),\quad u(z+T)=e^{i\theta}u(z),
\label{eq:Lposdef}
\end{equation}
because $\Sigma_\theta^\mathrm{H}\subset\Sigma^\mathrm{H}\subset\mathbb{R}_-$. We summarize the results of this discussion in the following: 
\begin{proposition}\label{cor:hnegdef}
Hill's differential operator $\cH$ is negative semidefinite in the case that $f$ is a rotational traveling wave.
For librational traveling waves $f$, $\cH$ is indefinite.
\end{proposition}

Finally, we may easily use the theory of Hill's equation described above to prove the following result.  Recall the notion of spectral stability for traveling waves of infinite speed described in \S\ref{section-linearization-infinite-speed}, and in particular Definition~\ref{def:stable-infinite-speed}.
\begin{theorem}[instability criterion for infinite speed waves]
A time-periodic solution $f$ of the Klein-Gordon equation \eqref{eqnlKG} that is independent of $x$ is spectrally stable if and only if $\Sigma^\mathrm{H}\cap\R_-=\R_-$, that is, if the part of $\Sigma^\mathrm{H}$ on the negative half-axis has no gaps.
\label{theorem:infinite-speed}
\end{theorem}
\begin{proof}
Referring to the equation \eqref{eq:Hill-infinite-speed}, we see that $k\in\R$ corresponds to $\nu=-k^2\le 0$.  If such a value of $\nu$ lies in $\Sigma^\mathrm{H}$, then either it is a periodic or antiperiodic eigenvalue, in which case the solutions of \eqref{eq:Hill-infinite-speed} are linearly bounded, or the general solution of \eqref{eq:Hill-infinite-speed} is uniformly bounded.  On the other hand, if $\nu$ does not lie in the spectrum $\Sigma^\mathrm{H}$, then there exists a solution of \eqref{eq:Hill-infinite-speed} that grows exponentially in time.  Therefore, to have spectral stability in the sense of Definition~\ref{def:stable-infinite-speed} it is necessary that the entire negative real axis in the $\nu$-plane consist of spectrum, that is, all of the inequalities in the sequence \eqref{eq:oscillate} corresponding to negative $\nu$ must degenerate to equalities.  
\end{proof}

\begin{corollary}  All time-periodic and $x$-independent solutions of the Klein-Gordon equation \eqref{eqnlKG} that are of rotational type are spectrally unstable.
\label{corollary:rotation-infinite-speed-unstable}
\end{corollary}
\begin{proof}
In the rotational case, $\Sigma^\mathrm{H}=\Sigma^\mathrm{H}\cap\R_-$, so the criterion of Theorem~\ref{theorem:infinite-speed} reduces to the condition that $\Sigma^\mathrm{H}$ should have
no gaps for spectral stability.  But $f$ is necessarily non-constant, so $\Sigma^\mathrm{H}$ has at least one gap.  
\end{proof}

\subsection{Relating spectra} \label{sec:relate} 
Recall \eqref{eq:Hill-spectrum-pull-back} defining the set $\sigma^\mathrm{H}$ in the complex $\lambda$-plane related to the Hill's spectrum $\Sigma^\mathrm{H}$.  It was pointed out in Remark~\ref{remark:Hill-spectrum-pull-back} that $\sigma\neq\sigma^\mathrm{H}$ despite the simple transformation \eqref{eq:transform} relating the differential equations \eqref{eq:spectral} (to which $\sigma$ is associated) and \eqref{eq:hill} (to which $\sigma^\mathrm{H}$ is associated).  In particular $\sigma^\mathrm{H}$ is necessarily the union of an unbounded subset of the imaginary axis and a bounded subset (possibly just $\{0\}$) of the real axis, but
as we will show in \S\ref{secmodinst}, $\sigma$ can contain points not on either the real or imaginary axes.  However, we have also shown that $0\in\sigma\cap\sigma^\mathrm{H}$, and hence the two sets are not disjoint.  

Parallel with the development leading up to \eqref{eq:monodromy}, we write Hill's equation \eqref{eq:hill} as a first-order system by introducing the vector unknown $\by:=(y,y_z)^\mathsf{T}\in\mathbb{C}^2$.  Letting $\bF^\mathrm{H}(z,\nu)$ denote the corresponding fundamental solution matrix normalized by the initial condition $\bF^\mathrm{H}(0,\nu)=\Id$, the monodromy matrix for Hill's equation is defined by $\bM^\mathrm{H}(\nu):=\bF^\mathrm{H}(T,\nu)$.  Just as the Floquet multipliers $\mu$ were defined as the eigenvalues of $\bM(\lambda)$, so are the multipliers
$\mu^\mathrm{H}$ defined as the eigenvalues of $\bM^\mathrm{H}(\nu)$.  Note that as both monodromy matrices are $2\times 2$, the multipliers are roots of a quadratic equation in each case, and hence there are two of them (counted with multiplicity).  The Hill's spectrum $\Sigma^\mathrm{H}$ can be characterized as the set of values of $\nu$ for which at least one of the Floquet multipliers $\mu^\mathrm{H}(\nu)$ lies on the unit circle in the complex plane.  When $\nu\in\mathbb{R}$, the Floquet multipliers $\mu^\mathrm{H}(\nu)$ either are real numbers or form a complex-conjugate pair, and hence if $\nu\in\Sigma^\mathrm{H}\subset\mathbb{R}$, the multipliers are either a non-real conjugate pair of unit modulus, or they coalesce at $\mu^\mathrm{H}=1$ (for periodic eigenvalues $\nu$) or $\mu^\mathrm{H}=-1$ (for antiperiodic eigenvalues $\nu$).  

We first establish a result that shows how the Floquet multipliers $\mu$ are transformed under an exponential mapping of $w$ that generalizes the one given in \eqref{eq:transform}.  Our purpose here is two-fold:  firstly we use this result to relate the Floquet multipliers $\mu$ and $\mu^\mathrm{H}$, and then we will apply it to complete the proof of Proposition~\ref{prop:spectral-symmetry} (the part asserting that $\sigma=-\sigma$).  Let $w$ be a solution of the differential equation \eqref{eq:spectral}, and let a function $r=r(z)$ be defined by the invertible transformation
\begin{equation}\label{eq:gentran}
r(z) = e^{-\alpha(\lambda) z} w(z),
\end{equation}
where $\alpha$ depends on $\lambda$ (and possibly other parameters in the system) but not on $z$.  Then $r$ satisfies 
\begin{equation}\label{eq:requation}
r_{zz} + \left( 2 \alpha(\lambda) - \frac{2 \lambda c}{c^2-1} \right)r_z + \left(\alpha(\lambda)^2 - \frac{2 \lambda c}{c^2-1} \alpha(\lambda) + \frac{ \lambda^2 + V''(f(z))}{c^2-1} \right)r = 0,
\end{equation}
which is easily written as a first-order system:
\begin{equation}
\label{eq:rsys}
\br_z
 = \bA^\alpha(z,\lambda)\br,\quad \br:=
\begin{pmatrix} r \\ r_z \end{pmatrix},
\end{equation}
with coefficient matrix
\begin{equation}
\bA^\alpha(z,\lambda):=\begin{pmatrix}  0 & 1 \\ 
 \displaystyle \frac{2 \lambda c}{c^2-1} \alpha(\lambda) - \alpha(\lambda)^2 - \frac{ \lambda^2 + V''(f(z))}{c^2-1} & \displaystyle  \frac{2 \lambda c}{c^2-1} -  2 \alpha(\lambda) 
\end{pmatrix}. 
\end{equation}
(Notice that if we choose $\alpha(\lambda) = \lambda c/(c^2-1)$, then \eqref{eq:gentran} reduces to \eqref{eq:transform} and therefore \eqref{eq:requation} reduces to \eqref{eq:hill}.)  Let $\bF^\alpha(z,\lambda)$ denote the fundamental solution matrix of \eqref{eq:rsys} with initial condition $\bF^\alpha(0,\lambda)=\Id$, let $\bM^\alpha(\lambda):=\bF^\alpha(T,\lambda)$ denote the corresponding monodromy matrix, and let $\mu^\alpha$ denote a Floquet multiplier of \eqref{eq:rsys}, i.e., an eigenvalue of $\bM^\alpha(\lambda)$.
\begin{lemma} \label{lem:floquetrelation}
The following identities hold:
\begin{equation}
\begin{aligned}
\bF^\alpha(z,\lambda)&=e^{-\alpha(\lambda)z}\bB(\lambda)\bF(z,\lambda)\bB(\lambda)^{-1}\\
\bM^\alpha(\lambda)&=e^{-\alpha(\lambda)T}\bB(\lambda)\bM(\lambda)\bB(\lambda)^{-1},
\end{aligned}\quad
\text{where}\quad
\bB(\lambda):=\begin{pmatrix}1 & 0 \\ -\alpha(\lambda) & 1\end{pmatrix}.
\label{eq:F-Falpha}
\end{equation}
Furthermore, to each Floquet multiplier $\mu=\mu(\lambda)$ there corresponds a Floquet multiplier $\mu^\alpha=\mu^\alpha(\lambda)$ via the relation $\mu^\alpha(\lambda)=e^{-\alpha(\lambda)T}\mu(\lambda)$.
\end{lemma}
\begin{proof}
The relation \eqref{eq:gentran} connecting solutions $w$ of \eqref{eq:spectral} with solutions $r$ of \eqref{eq:requation} induces a corresponding relation on
the vectors $\bw$ and $\br$ solving the first-order systems \eqref{eq:firstorder} and \eqref{eq:rsys} respectively:
$ \br= 
e^{-\alpha(\lambda)z} \bB(\lambda)\bw$. 
The first identity in \eqref{eq:F-Falpha} then follows immediately (the final factor of $\bB(\lambda)^{-1}$ fixes the initial condition for $\bF^\alpha$).  Evaluating at $z=T$ yields the second identity in \eqref{eq:F-Falpha}, from which the statements concerning the  multipliers follow by computation of the eigenvalues.
\end{proof}
By taking $\alpha(\lambda):=\lambda c/(c^2-1)$, Lemma~\ref{lem:floquetrelation} yields the following immediate corollary.  Let 
\begin{equation}
q:=\frac{cT}{c^2-1}\in\mathbb{R}.
\label{eq:q-def}
\end{equation}
\begin{corollary}\label{corollary:floquetpandq} 
The multipliers $\mu(\lambda)$ and $\mu^\mathrm{H}(\nu)$ are related as follows:
\begin{equation}
\mu^\mathrm{H}(\nu(\lambda))=e^{-q\lambda}\mu(\lambda), \quad
\nu(\lambda)=\left(\frac{\lambda}{c^2-1}\right)^2.
\end{equation}
\end{corollary}
As an application of this result, we now show that $\sigma$ and $\sigma^\mathrm{H}$ agree (only) on the imaginary axis.
\begin{proposition}\label{prop:imag} $\sigma \cap i \R = \sigma^\mathrm{H} \cap i \R$, and if $\lambda \in \sigma \cap \sigma^\mathrm{H}$ then $\lambda \in i \R$. 
\end{proposition}
\begin{proof} The first statement follows immediately from Corollary~\ref{corollary:floquetpandq}. To prove the second statement, suppose that $\lambda \in \sigma^H$, which implies that either $\lambda$ is real and nonzero, or purely imaginary. We show that the case that $\lambda$ is real and nonzero is inconsistent with $\lambda\in\sigma$.  Indeed, since $\lambda\in\sigma^\mathrm{H}$,
the Floquet multipliers $\mu^\mathrm{H}(\nu(\lambda))$ both have unit modulus, and it follows that all solutions of Hill's equation \eqref{eq:hill} are either bounded and quasi-periodic or exhibit linear growth (the latter only for $\nu(\lambda)$ in the periodic or antiperiodic partial spectra).  
Accounting for the assumption that $\Re\lambda\neq 0$, applying the transformation \eqref{eq:transform} (a bijection between the solution spaces of the differential equations \eqref{eq:spectral} and \eqref{eq:hill}) shows that every nonzero solution $w$ of \eqref{eq:spectral} is exponentially unbounded for large $|z|$.  It follows that $\lambda\not\in\sigma$. 
\end{proof}

The proof of Proposition~\ref{prop:imag} used the fact that \emph{both} Floquet multipliers $\mu^\mathrm{H}(\nu)$ of Hill's equation have unit modulus when $\nu\in\Sigma^\mathrm{H}$.  Quite a different phenomenon can occur for the spectrum $\sigma$, as the following result shows.
\begin{proposition}\label{prop:pabel} Let $c\neq 0$, and suppose that $\lambda \in \sigma$ with $\Re(\lambda) \neq 0$. Then the Floquet multipliers $\mu(\lambda)$ are distinct, and exactly one of them has unit modulus.
\end{proposition}
\begin{proof}
According to Abel's Theorem, the product of the Floquet multipliers is
\begin{equation}
\det(\bM(\lambda)) = \exp\left(\int_0^T\tr(\bA(z,\lambda))\,dz\right) = e^{2q\lambda}.
\label{eq:Abels-Identity}
\end{equation}
By assumption, at least one of the multipliers has unit modulus because $\lambda\in\sigma$.
If also the other multiplier has unit modulus, then $|\det(\bM(\lambda))|=1$ and therefore $\Re\lambda=0$ because $q\neq 0$.  Hence we arrive at a contradiction, and it follows that the second multiplier cannot have unit modulus.
\end{proof}
A result that is more specialized but of independent interest is the following:
\begin{proposition} Let $c\neq 0$ and suppose that $\lambda\in\sigma$ is a nonzero real number.  Then $\lambda$ is either a periodic eigenvalue or an antiperiodic eigenvalue, and the Floquet multipliers $\mu(\lambda)$ are distinct.
\label{prop:nonzero-real}
\end{proposition}
\begin{proof}  Since $\lambda\in\mathbb{R}$, the differential equation \eqref{eq:spectral} has real coefficients and hence the multipliers either form a complex-conjugate pair or are both real.  Furthermore, since $\lambda\in\sigma$, in the conjugate-pair case both multipliers lie on the unit circle, a contradiction with Proposition~\ref{prop:pabel} because $\Re\lambda\neq 0$.  Therefore, the multipliers are both real, and $\lambda\in\sigma$ implies that one of them is $\pm 1$.
According to Abel's Theorem \eqref{eq:Abels-Identity}, the product of the Floquet multipliers is a positive number not equal to $1$ and hence the second multiplier can be neither $1$ nor $-1$.  (That the multipliers are distinct also follows from Proposition~\ref{prop:pabel}.)
\end{proof}
Next, we consider situations in which the Floquet multipliers of either equation \eqref{eq:spectral} or Hill's equation \eqref{eq:hill} fail to be distinct.
\begin{proposition} \label{prop:sum} Let $c\neq 0$, and suppose that for some $\lambda \in \C$ both Floquet multipliers $\mu(\lambda)$ coincide.  Then 
\begin{itemize}
\item[(i)] both Floquet multipliers $\mu^\mathrm{H}(\nu(\lambda))$ also coincide, and the common values are $\mu^\mathrm{H}(\nu(\lambda)) = \pm 1$ and $\mu(\lambda) = \pm e^{q\lambda}$;
\item[(ii)] either $\lambda$ is real and nonzero and $\lambda\not\in\sigma$, or $\lambda$ is imaginary  and $\lambda\in\sigma\cap\sigma^\mathrm{H}$. 
\end{itemize}
\end{proposition}
\begin{proof}
The fact that the two multipliers $\mu(\lambda)$ of equation \eqref{eq:spectral} coincide exactly when the two multipliers $\mu^\mathrm{H}(\nu(\lambda))$ coincide is a simple consequence of Corollary~\ref{corollary:floquetpandq}.  Applying Abel's Theorem \eqref{eq:Abels-Identity} to the system \eqref{eq:rsys} with $\alpha(\lambda)=\lambda c/(c^2-1)$ shows that the product of the multipliers $\mu^\mathrm{H}$ is always $1$ and therefore if they coincide they must take the values $1$ or $-1$ (together).  Again using Corollary~\ref{corollary:floquetpandq} shows that
the corresponding coincident value of the multipliers $\mu$ must be either $e^{q\lambda}$ or $-e^{q\lambda}$, which proves (i).

According to (i), $\nu(\lambda)\in\Sigma^\mathrm{H}\subset\mathbb{R}$ (being a periodic or antiperiodic eigenvalue), and it follows that $\lambda\in\sigma^\mathrm{H}$.  Because
$\nu(\lambda)\in\mathbb{R}$, either $\lambda$ is either purely imaginary, in which case Proposition~\ref{prop:imag} shows that $\lambda\in\sigma$, or $\lambda$ is a nonzero real number.
In the latter case, the fact that the multipliers $\mu(\lambda)$ are not distinct shows that $\lambda\not\in\sigma$ according to Proposition~\ref{prop:pabel}.  This completes the proof of (ii).
\end{proof}
\begin{remark}\label{rem:negsim} 
As a second application of Lemma~\ref{lem:floquetrelation}, we now complete the proof of 
Proposition~\ref{prop:spectral-symmetry} by showing that $\lambda\in\sigma$ implies $-\lambda\in\sigma$.
The transformation \eqref{eq:gentran} with the choice
\begin{equation}
\alpha(\lambda)=\frac{2\lambda c}{c^2-1}
\end{equation}
connects the differential equation \eqref{eq:spectral} with the equation (cf. \eqref{eq:requation})
\begin{equation}\label{eq:lamplus}
r_{zz} + \frac{2 c \lambda}{c^2-1} r_z + \frac{\lambda^2 + V''(f(z))}{c^2-1} r = 0,
\end{equation}
which is simply \eqref{eq:spectral} with $\lambda$ replaced by $-\lambda$.  Lemma~\ref{lem:floquetrelation} therefore implies that the Floquet multipliers $\mu(\pm\lambda)$ are related as follows:
\begin{equation}
\mu(-\lambda)=e^{-2q\lambda}\mu(\lambda).
\label{eq:mu-lambda-minus-lambda}
\end{equation}
Since $\tr(\bM(\lambda))$ is the sum of the multipliers, using \eqref{eq:Abels-Identity} shows that the function $\tilde{D}(\lambda,\mu)$ related to $D(\lambda,\mu)$ defined by \eqref{eq:det0} by
\begin{equation}
\tilde{D}(\lambda,\mu):=\frac{e^{-q\lambda}}{\mu}D(\lambda,\mu) \quad\text{satisfies}\quad
\tilde{D}(-\lambda,\mu)=\tilde{D}(\lambda,1/\mu).
\label{eq:tilde-D}
\end{equation}
It follows easily that $\lambda\in\sigma_\theta$ implies $-\lambda\in\sigma_{-\theta\pmod{2\pi}}$ and hence $\lambda\in\sigma$ implies $-\lambda\in\sigma$.
\myendrmk
\end{remark}

Finally, we note that Propositions~\ref{prop:pabel}, \ref{prop:nonzero-real}, and  \ref{prop:sum} have all required the condition $c\neq 0$, but much more information is available if $c=0$:
\begin{proposition}
If $c=0$, then $\sigma=\sigma^\mathrm{H}$.
\label{prop:zero-velocity-agree}
\end{proposition}
\begin{proof}
If $c=0$, then the relation \eqref{eq:transform} reduces to the identity, and hence the differential equations \eqref{eq:spectral} and \eqref{eq:hill} --- whose spectra are, respectively, $\sigma$ and $\sigma^\mathrm{H}$ (the latter pulled back to the $\lambda$-plane as in  \eqref{eq:Hill-spectrum-pull-back}) --- are identical.
\end{proof}

\section{Analysis of the monodromy matrix}\label{sec:monodromy}
In this section we study the monodromy matrix $\bM(\lambda)$ of equation \eqref{eq:spectral}  with the goal of analyzing the $L^2$ spectrum of the problem \eqref{eq:spectral} in a neighborhood of the origin $\lambda = 0$. When $\lambda = 0$, equation \eqref{eq:spectral} coincides with  Hill's equation \eqref{eq:hill}, and hence $\bM(0)=\bM^\mathrm{H}(\nu(0))=\bM^\mathrm{H}(0)$.  We start by setting $\lambda=0$ and explicitly finding the fundamental solution matrix $\bF(z,0)=\bF^\mathrm{H}(z,0)$.  Then we compute the power series expansion of  $\bF(z,\lambda)$ based at $\lambda=0$.  Since $\bF(z,\lambda)$ is entire for bounded $z$, this series has an infinite radius of convergence and setting $z=T$ then gives the corresponding power series expansion of the monodromy matrix $\bM(\lambda)$.  This series is also an asymptotic series in the limit $\lambda\to 0$, and hence a finite number of terms suffice to approximate the spectrum $\sigma$ in a vicinity of the origin.

\subsection{Solutions for $\lambda = 0$}
\label{secsolsat0}
Let us assume that the periodic traveling wave solution $u=f$ of the Klein-Gordon equation \eqref{eqnlKG} has been made unique given parameter values $(E,c)\in\region$ by the imposition of either \eqref{eq:sub-normalize} (for subluminal waves) or \eqref{eq:super-normalize} (for superluminal waves).  The functions $u_0:\region\to\R$ and $v_0:\region\to\R$ given by
\begin{equation}
u_0(E,c):=f(0)\quad\text{and}\quad v_0(E,c):=f_z(0)
\label{eq:u0v0define}
\end{equation}
are then well-defined on $\region$. Furthermore, $v_0>0$, and $u_0$ is constant on 
$\region_<^\mathrm{lib}\cup\region_<^\mathrm{rot}$ and also on $\region_>^\mathrm{lib}\cup\region_>^\mathrm{rot}$, with different values both satisfying $V'(u_0)=0$.  It follows from \eqref{eqnlp} that $f_{zz}(0)=0$, and by periodicity, $f_{zz}(T)=0$ also.
\subsubsection{First construction of $\bF(z,0)$ and $\bM(0)$}
\label{section:first-construction}
\begin{lemma}
\label{lem:solat0}
Suppose that $\lambda=0$.  The two-dimensional complex vector space of solutions to the first-order system \eqref{eq:firstorder} is spanned by
\begin{equation}
 \label{eq:defY0Y1}
\bw^z(z) := \begin{pmatrix}
          f_z \\ f_{zz}
         \end{pmatrix} \quad \text{and} \quad \bw^E(z) := \begin{pmatrix} f_E \\ f_{Ez} \end{pmatrix}.
\end{equation}
\end{lemma}
\begin{proof}
Lemma \ref{lemfC2E} guarantees that $f$ is a $C^2$ function of $E$ and $z$. Differentiating equation \eqref{eqnlp} with respect to $z$ yields
$
(c^2 -1) f_{zzz} + V''(f) f_z = 0 
$.
This proves that $\bw = \bw^z(z)$ is a solution to \eqref{eq:firstorder} when $\lambda=0$. On the other hand, differentiating \eqref{eqnlp} with respect to $E$ one gets $
(c^2 -1) f_{zzE} + V''(f) f_E = 0$,
proving that $\bw = \bw^E(z)$ solves \eqref{eq:firstorder} for $\lambda=0$ as well. To verify independence of $\bw^z(z)$ and $\bw^E(z)$, differentiate equation \eqref{eq:waveeq} with respect to $E$:
\begin{equation}
\label{eqderivE}
(c^2 -1) f_z f_{zE} = 1 - V'(f) f_E. 
\end{equation}
Combining this with \eqref{eqnlp}, one obtains
\begin{equation}
\label{wronskian}
 \det (\bw^z(z), \bw^E(z)) = f_z f_{Ez} - f_E f_{zz} = \frac{1}{c^2 -1} \neq 0.
\end{equation}
Hence the Wronskian never vanishes and therefore $\bw^z(z)$ and $\bw^E(z)$ are linearly independent for all $z$ and for all $(E,c)\in\region$.
\end{proof}

Let us define the matrix solution 
\begin{equation}
\label{eq:defofQ}
\bQ_0(z) := \left( \bw^z(z), \bw^E(z) \right) 
\end{equation}
to the first-order system \eqref{eq:firstorder} for $\lambda=0$.  Since $\bQ_0$ is invertible by Lemma~\ref{lem:solat0}, the normalized fundamental solution matrix for \eqref{eq:firstorder}
at $\lambda = 0$ is
\begin{equation}
 \bF(z,0) = \bQ_0(z)\bQ_0(0)^{-1}.
 \label{eq:FQ0}
\end{equation}
Note that as a consequence of the Wronskian formula \eqref{wronskian}, $\bQ_0(z)^{-1}$ is given by
\begin{equation}
\label{inverseQ0}
\bQ_0(z)^{-1} = (c^2 -1) \begin{pmatrix}
                        f_{Ez}(z) & -f_E(z) \\ -f_{zz}(z) & f_z(z)
                       \end{pmatrix}.
\end{equation}
The matrix $\bQ_0(0)$ may be expressed 
in terms of the functions $u_0$ and $v_0$ defined on $\region$ by \eqref{eq:u0v0define} as
\begin{equation}
\bQ_0(0)=\begin{pmatrix}v_0 & \partial_Eu_0\\
0 & \partial_Ev_0\end{pmatrix}=\begin{pmatrix}v_0 & 0\\
0 & \partial_Ev_0\end{pmatrix},
\label{eq:Q0}
\end{equation}
where $\partial_E$ denotes the partial derivative with respect to $E$, because $u_0$ is piecewise constant on $\region$.
Similarly from \eqref{inverseQ0}, we have
\begin{equation}
\bQ_0(0)^{-1}=
\begin{pmatrix}(c^2-1)\partial_Ev_0 & 0\\
0 & (c^2-1)v_0\end{pmatrix}.
\label{eq:Q0-inverse}
\end{equation}
(The identity 
\begin{equation}
(c^2-1)v_0\partial_Ev_0=1
\label{eq:v0-partialE-identity}
\end{equation}
implied by combining \eqref{eq:Q0} and \eqref{eq:Q0-inverse}
can also be obtained by differentiation of \eqref{eq:waveeq} with respect to $E$ at $z=0$.)  It follows from \eqref{eq:defofQ}, \eqref{eq:FQ0}, and \eqref{eq:Q0-inverse} that the fundamental solution matrix $\bF(z,0)$ is
\begin{equation}
\bF(z,0)= \dfrac{1}{v_0}\begin{pmatrix}f_z(z) &(c^2-1)v_0^2f_E(z)\\
f_{zz}(z) & (c^2-1)v_0^2f_{Ez}(z)\end{pmatrix}.
\label{eq:F0-1}
\end{equation}

The corresponding monodromy matrix $\bM(0)$ is obtained  by setting $z=T$ in $\bF(z,0)$.
To simplify the resulting formula, we now express $f_z(T)$, $f_{zz}(T)$, $f_{E}(T)$ and $f_{Ez}(T)$ in terms of the functions $u_0$ and $v_0$ defined on $\region$.  Since $f_z$ and $f_{zz}$ are periodic functions with period $T$, we obviously have
\begin{equation}
f_z(T)=f_z(0)=v_0\quad\text{and}\quad f_{zz}(T)=f_{zz}(0)=0.
\label{eq:QTfirstcol}
\end{equation}
To express $f_E(T)$ and $f_{Ez}(T)$ in terms of $u_0$ and $v_0$, first note that 
since $f(T)=f(0)\pmod{2\pi}$, we may write $u_0$ as $u_0=f(T)$; differentiation 
with respect to $E$ and taking into account that the period $T$ depends on $E$ yields
\begin{equation}
\begin{split}
\partial_E u_0 &= f_E(T) + T_E f_z(T)\\
&= f_E(T) + T_E f_z(0)\quad\text{(because $f_z$ has period $T$)}\\
&= f_E(T) + T_E v_0. 
\end{split}
\end{equation}
Therefore, since $u_0$ is piecewise constant on $\region$,
\begin{equation}
f_E(T)=-T_Ev_0.
\label{eq:QT12}
\end{equation}
Similarly, since $f_z$ has period $T$, we can write $v_0=f_z(T)$, and then differentiation yields $
\partial_Ev_0=T_Ef_{zz}(T) + f_{Ez}(T)$,
and therefore as $f_{zz}(T)=0$,
\begin{equation}
f_{Ez}(T)=\partial_Ev_0.
\label{eq:QT22}
\end{equation}
Substituting \eqref{eq:QTfirstcol}, \eqref{eq:QT12}, and \eqref{eq:QT22} into $\bM(0)=\bF(T,0)$ with $\bF(z,0)$ given by \eqref{eq:F0-1}, and using the identity \eqref{eq:v0-partialE-identity}, we obtain the following.

\begin{proposition}
\label{prop:Jordan}
The monodromy matrix $\bM(0)$ is given by
\begin{equation}
\bM(0)=\begin{pmatrix} 1 & -(c^2-1)T_Ev_0^2\\0 & 1\end{pmatrix}.
\label{eq:M0-explicit}
\end{equation}
In particular, $\bM(0)$ is not diagonalizable unless $T_E=0$.
\end{proposition}

\subsubsection{Alternate construction of $\bF(z,0)$ and $\bM(0)$}
\label{sec:alternate}
We can also obtain the (unique) fundamental solution matrix for $\lambda=0$ and the corresponding (unique) monodromy matrix by an alternate method that avoids
explicit differentiation with respect to the energy $E$.  Indeed, to obtain $\bF(z,0)$, it suffices to find the particular solutions $w^1(z)$ and $w^2(z)$ of the differential equation \eqref{eq:spectral} for $\lambda=0$, namely
\begin{equation}
\label{eq:SpectralODEzero}
(c^2-1) w_{zz} + V''(f(z)) w = 0, 
\end{equation}
that satisfy the
initial conditions
\begin{equation}
\label{eq:w1w2-ICs}
w^1(0)=w^2_z(0)=1\quad\text{and}\quad w^1_z(0)=w^2(0)=0.
\end{equation}
Then 
\begin{equation}
\label{Phiat0}
 \bF(z,0) = \begin{pmatrix}
              w^1(z) & w^2(z) \\ w^1_z(z) & w^2_z(z)
             \end{pmatrix}.
\end{equation}

In order to obtain the particular solutions $w^1(z)$ and $w^2(z)$, we will first develop the general solution of the differential equation \eqref{eq:SpectralODEzero}.  
In \S\ref{section:first-construction} (cf.\@ Lemma~\ref{lem:solat0}) it was observed that one solution of this equation is given by
$w=f_z$.  The general solution of this second-order equation is readily obtained by the classical 
method of reduction of order, which is based on the substitution $w=gf_z$, where $g$ is to be regarded as a new unknown. 
Then $g$ can be found by a quadrature, and the general solution of \eqref{eq:SpectralODEzero} is
\begin{equation}
w(z)=Af_z(z) \int_0^z\frac{dy}{f_z(y)^2} + Bf_z(z),
\label{eq:wzero}
\end{equation}
where $A$ and $B$ are arbitrary constants. It follows from \eqref{eq:wzero} that the corresponding formula for the derivative is
\begin{equation}
w_z(z) = Af_{zz}(z)\int_0^z\frac{dy}{f_z(y)^2} +\frac{A}{f_z(z)} + Bf_{zz}(z). 
\end{equation}
 \begin{remark}
\label{remark:smooth}
Note that since $f_z(0)\neq 0$ according to the normalization \eqref{eq:sub-normalize}--\eqref{eq:super-normalize}, the general solution formula \eqref{eq:wzero}
makes sense in a neighborhood of $z=0$, and for rotational waves of both sub- and superluminal types one has $f_z(z)\neq 0$ for all $z$ and hence the formula \eqref{eq:wzero} requires no additional explication.
However, in the case of librational waves $f_z(z)$ will have exactly two zeros within the fundamental period interval $z\in (0,T)$, and therefore the integrand $1/f_z^2$ becomes singular near these points.  On the other hand, the zeros of $f_z$ are necessarily simple.  Indeed,  $f_z(z)$ is a nontrivial solution of the linear second-order equation \eqref{eq:SpectralODEzero}, and therefore if $f_z(z_0)=f_{zz}(z_0)=0$ for some $z_0\in\mathbb{R}$, then by the Existence and Uniqueness Theorem we would also have $f_z(z)=0$ for all $z$ in contradiction to the nontriviality of $f_z$.   This argument shows that all apparent singularities corresponding to the zeros of $f_z$ in the general solution formula \eqref{eq:wzero} for the equation \eqref{eq:SpectralODEzero} are necessarily removable.  Still, care must be taken in the use of the formula \eqref{eq:wzero} for $z$ near roots of $f_z$ as will be discussed below in detail.
\myendrmk
\end{remark}

To find the constants $(A_1,B_1)$ and $(A_2,B_2)$ corresponding to the fundamental pair of solutions $w^1(z)$
and $w^2(z)$ respectively, one simply imposes the initial conditions \eqref{eq:w1w2-ICs}.  The result of this calculation is that 
\begin{equation}
A_1=0\quad\text{and}\quad B_1=\frac{1}{v_0}, \quad \text{while} \quad A_2=v_0\quad\text{and}\quad B_2=0.
\label{eq:AB-1}
\end{equation}
With these values in hand, one obtains a formula for $\bF(z,0)$:
\begin{equation}
\bF(z,0)=\begin{pmatrix}\displaystyle \frac{f_z(z)}{v_0} & \displaystyle v_0f_z(z)\int_0^z\frac{dy}{f_z(y)^2}\\
\displaystyle\frac{f_{zz}(z)}{v_0} & \displaystyle v_0f_{zz}(z)\int_0^z\frac{dy}{f_z(y)^2} +\frac{v_0}{f_z(z)}\end{pmatrix}.
\label{eq:F0-2-near-zero}
\end{equation}

Now, \eqref{eq:F0-2-near-zero} is a valid formula for $\bF(z,0)$ for $z$ in any interval containing $z=0$ that also contains no zeros of $f_z$.  Although the matrix elements will extend continuously to the nearest zeros of $f_z$ (cf.\@ Remark~\ref{remark:smooth}), formula \eqref{eq:F0-2-near-zero} is no longer correct if $z$ is allowed to pass \emph{beyond} a zero of $f_z$.  Therefore, to allow $z$ to increase by a period to $z=T$ for the purposes of computing the monodromy matrix, we require an alternate formula for $\bF(z,0)$ that is valid for $z$ near $T$.  There is only an issue to be addressed in the librational wave case, because $f$ is monotone for rotational waves.  Hence we consider the librational case in more detail now.

Let $f$ be a librational wave, and let $z_0\in (0,T)$ denote the smallest positive zero of $f_z$.  
%
First suppose that $f_z$ is an analytic function of $z$ in some horizontal strip of the complex plane containing the real axis.  Then we must have $f_{zzz}(z_0)=0$ wherever $f_z(z_0)=0$ (by the differential equation \eqref{eq:SpectralODEzero}), and it follows by a quick computation that $dy/f_z(y)^2$ is a meromorphic differential having a double pole \emph{with zero residue} at $z_0$.  In this situation, the contour integral $\int_0^z dy/f_z(y)^2$ defines a single-valued meromorphic function near $z=z_0$ (with a simple pole).  In this analytic situation, the solution $w(z)$ may be continued to the real interval $z>z_0$ simply by choosing a path of integration from $y=0$ to $y=z$ that avoids the double pole of the integrand at $y=z_0$, and all such paths are equivalent by the Residue Theorem.  
Thus, the same solution that is given by \eqref{eq:wzero} for $z$ near $z=0$ is given for $z$ near $T$ by 
\begin{equation}
w(z)=Af_z(z)\int_{T}^z\frac{dy}{f_z(y)^2} + (B+\delta A)f_z(z),
\label{eq:wzeronearZ}
\end{equation}
where $\delta$ is defined by the formula
\begin{equation}
\delta:=\int_0^T\frac{dy}{f_z(y)^2}
\label{eq:delta-analytic}
\end{equation}
in which the integral is interpreted as a complex contour integral over an arbitrary contour in the strip of analyticity of $f_z$ that connects the specified endpoints and avoids the double-pole singularities of the integrand.  Note in particular that although $f_z(y)^{-2}$ is certainly positive for real $y\in (0,T)$, $\delta$ need not be positive because a real path of integration is not allowed if $f_z$ has any zeros.  Of course \eqref{eq:wzeronearZ} and \eqref{eq:delta-analytic} are also valid for rotational waves (in which case $f$ is monotone); in this situation a real path of integration is indeed permitted in \eqref{eq:delta-analytic} and consequently $\delta>0$.

If the rather special condition of analyticity of $f_z$ is dropped, a more general strategy is to write a different formula for the solution $w(z)$ in each maximal open subinterval of $(0,T)$ on which $f_z\neq 0$; in each such interval $I$ one writes 
\begin{equation}
w(z)=A_If_z(z)\int_{z_I}^z\frac{dy}{f_z(y)^2}+B_If_z(z),\quad z\in I,
\end{equation}
for some fixed point $z_I\in I$.  One then chooses the constants $A_I$ and $B_I$ for the different subintervals of $(0,T)$ to obtain continuity of $w$ and $w_z$ across each zero of $f_z$.  In this way, one finds after some calculation that exactly the same solution of \eqref{eq:SpectralODEzero} represented for $z$ near $z=0$ by the formula \eqref{eq:wzero} is again represented (regardless of whether $f$ is rotational or librational) for $z$ near 
$z=T$ by the formula \eqref{eq:wzeronearZ}, but
where the quantity $\delta$ is now defined as follows:
\begin{multline}
\delta:=\int_0^T\Bigg(\frac{1}{f_z(y)^2} -\mathop{\sum_{f_z(z_k)=0}}_{0<z_k<T}\frac{1}{f_{zz}(z_k)^2(y-z_k)^2}\Bigg)\,dy\\
{} -\mathop{\sum_{f_z(z_k)=0}}_{0<z_k<T}\frac{1}{f_{zz}(z_k)^2}\left(\frac{1}{z_k}+\frac{1}{T-z_k}\right).
\label{defDeltah}
\end{multline}
Here the integral is a real integral.  The quantity on the right-hand side should be regarded as the correct regularization (finite part) of the divergent integral \eqref{eq:delta-analytic} in the case that $f_z$ has zeros in $(0,T)$.

Note that there are either zero (for rotational waves) or two (for librational waves) terms in each sum.  In the case when $f_z$ is analytic, the formula \eqref{defDeltah} reduces to the (expected) simple form \eqref{eq:delta-analytic}, but this form applies in general and it also makes very clear that $\delta$ is a real quantity that can in principle have any sign. 

Given the formula \eqref{eq:wzeronearZ} and the constants \eqref{eq:AB-1}, we obtain a formula for the fundamental solution matrix $\bF(z,0)$ valid for $z$ near $T$:
\begin{equation}
\label{eq:F0-2-T}
\bF(z,0)=\begin{pmatrix}
\displaystyle \frac{f_z(z)}{v_0} & \displaystyle v_0f_z(z)\int_T^z\frac{dy}{f_z(y)^2} +\delta v_0 f_z(z)\\
\displaystyle \frac{f_{zz}(z)}{v_0} & \displaystyle v_0f_{zz}(z)\int_T^z\frac{dy}{f_z(y)^2}+\frac{v_0}{f_z(z)}
\end{pmatrix}.
\end{equation}
Substituting $z=T$ into \eqref{eq:F0-2-T} yields a second formula for the monodromy matrix $\bM(0)$:
\begin{equation}
\bM(0)=\begin{pmatrix}1 & \delta v_0^2\\0 & 1\end{pmatrix}.
\end{equation}
Comparing with Proposition~\ref{prop:Jordan} (noting the uniqueness of the monodromy matrix), and taking into account that $v_0\neq 0$, we obtain an identity:
\begin{equation}
\delta = -(c^2-1)T_E
\label{eq:delta-identity}
\end{equation}
that could also have been obtained by direct but careful computation of $T_E$ from the formulae
\eqref{periodsuplib}, \eqref{periodsublib}, \eqref{periodsuprot}, and \eqref{periodsubrot}.

\begin{remark}
As an illustration of the direct calculation of $\delta$ from \eqref{defDeltah} in practice, we suppose that $f$ is a librational traveling wave solution of \eqref{eqnlKG} with sine-Gordon potential $V(u)=-\cos(u)$.  Then as $f_z$ has exactly two zeros per fundamental period $T$, 
to calculate $\delta$ we should in general use the formula \eqref{defDeltah} representing the finite part of a divergent integral.  As will be seen shortly, however, $f_z$ is an analytic function for real $z$ in the sine-Gordon case that extends to the complex plane as a meromorphic function (in fact an elliptic function) that satisfies the functional identity $f_z(z+T)=f_z(z)$.  Therefore, it is also possible to calculate $\delta$ by contour integration using the formula \eqref{eq:delta-analytic}.  Let $\eta>0$ be a sufficiently small number.   We choose an integration contour consisting of three segments:  a segment from $y=0$ to $y=i\eta$ followed by a segment from $y=i\eta$ to $y=T+i\eta$ followed by a segment from $y=T+i\eta$ to $y=T$.
By periodicity of $f_z(\cdot)$ the contributions of the first and third arcs will cancel and we then will have
\begin{equation}
\delta=\int_{i\eta}^{T+i\eta}\frac{dy}{f_z(y)^2}.
\end{equation}

To calculate $\delta$ we now represent the integrand in terms of elliptic functions.  
First consider the superluminal case, in which $f$ will oscillate about a mean value of $f=0$.
Making the substitution
\begin{equation}
\sin(\tfrac{1}{2}f(\xi\sqrt{c^2-1}))=\sqrt{\frac{1+E}{2}}s(\xi)
\end{equation}
into equation \eqref{eq:waveeq} yields
\begin{equation}
\left(\frac{d\xi}{ds}\right)^2=\frac{1}{(1-s^2)(1-ms^2)},\quad m:=\frac{1+E}{2}\in (0,1).
\end{equation}
This equation is solved by the Jacobi elliptic function $\mathrm{sn}$ (cf.\@\cite[Chapter 22]{DLMF}):
\begin{equation}
s(\xi)=\mathrm{sn}(\xi;m)\quad\text{and hence}\quad
\sin(\tfrac{1}{2}f(z))=\sqrt{m}\,\mathrm{sn}\left(\frac{z}{\sqrt{c^2-1}};m\right).
\end{equation}
Therefore
\begin{equation}
\begin{split}
f_z(z)^2 = 2\frac{E+\cos(f(z))}{c^2-1} & =\frac{4m}{c^2-1}\left(1-\mathrm{sn}\left(\frac{z}{\sqrt{c^2-1}};m\right)^2\right) \\
&=\frac{4m}{c^2-1}\mathrm{cn}\left(\frac{z}{\sqrt{c^2-1}};m\right)^2.
\end{split}
\end{equation}
Also, the period $T$ of $f_z$ in this case is expressed in terms of the complete elliptic integral
of the first kind $K$ \cite[Chapter 19]{DLMF}:
\begin{equation}
T=2\sqrt{c^2-1}K,\quad K :=K(m).
\end{equation}
Now let us choose $\eta=\sqrt{c^2-1}K'$ where $K':=K(1-m)$.  Then using the identity
\begin{equation}
\mathrm{cn}(w+iK';m)^2=-\frac{\mathrm{dn}(w;m)^2}{\mathrm{sn}(w;m)^2}
\end{equation}
we have
\begin{equation}
\begin{split}
\delta=\int_{i\sqrt{c^2-1}K'}^{2\sqrt{c^2-1}K+i\sqrt{c^2-1}K'}
\frac{dy}{f_z(y)^2}&=\sqrt{c^2-1}\int_{iK'}^{2K+iK'}\frac{dw}{f_z(\sqrt{c^2-1}w)^2}\\ &=
\frac{(c^2-1)^{3/2}}{4m}\int_{iK'}^{2K+iK'}\frac{dw}{\mathrm{cn}(w;m)^2}\\ &=-\frac{(c^2-1)^{3/2}}{4m}\int_0^{2K}\frac{\mathrm{sn}(v;m)^2}{\mathrm{dn}(v;m)^2}\,dv
\end{split}
\end{equation}
Now since $\mathrm{dn}(v;m)$ is real and bounded away from zero for real $v$ while $\mathrm{sn}(v;m)$ is real and bounded for real $v$, it follows that
\begin{equation}
\delta<0\quad\text{for superluminal librational waves with $V(u)=-\cos(u)$}.
\end{equation}
Exactly the same calculation shows that $\delta<0$ for subluminal librational waves as well, since the latter case can be related to the former simply by the mapping $f\mapsto f+\pi$.
\myendrmk
\end{remark}

\subsection{Series expansion of the fundamental solution matrix $\bF(z,\lambda)$ about $\lambda=0$}
\label{sec:series-expansion}
The Picard iterates for the fundamental solution matrix $\bF(z,\lambda)$ converge uniformly on $(z,\lambda)\in [0,T]\times K$, where $K\subset\C$ is an arbitrary compact set.  Since the coefficient matrix $\bA(z,\lambda)$ is entire in $\lambda$ for each $z$, it follows that $\bF(z,\lambda)$ is
an entire analytic function of $\lambda\in\C$ for every $z\in [0,T]$.  Hence the fundamental solution matrix $\bF(z,\lambda)$ has a convergent Taylor expansion about every point of the complex $\lambda$-plane.  In particular, the series about the origin has the form
\begin{equation}
\bF(z,\lambda)=\sum_{n=0}^\infty \lambda^n\bF_n(z),\quad z\in [0,T]
\label{eq:F-Taylor}
\end{equation}
for some coefficient matrices $\{\bF_n(z)\}_{n=0}^\infty$, and this series has an infinite radius of convergence.  Setting $\lambda=0$ gives
$
\bF_0(z)=\bF(z,0)
$,
which has already been computed (two ways) in \S\ref{secsolsat0}.  
Our current goal is to 
explicitly compute $\bF_1(z)$ and $\bF_2(z)$.  The benefit of the latter finite computation is that as a convergent power series, the series \eqref{eq:F-Taylor} may equally well be interpreted as an asymptotic series in the Poincar\'e sense in the limit $\lambda\to 0$.  Thus, a finite number of terms are sufficient to obtain increasing accuracy in this limit, and the order of accuracy is determined by the number of retained terms.  We will also obtain the first few terms of the corresponding expansion of the monodromy matrix $\bM(\lambda)$, simply by evaluation of the terms of \eqref{eq:F-Taylor} for $z=T$.

Rather than proceed directly, we instead expand the fundamental solution matrix $\bF^\mathrm{H}(z,\nu)$ and corresponding monodromy matrix $\bM^\mathrm{H}(\nu)$ for Hill's equation \eqref{eq:hill}, and then apply Lemma~\ref{lem:floquetrelation} with $\alpha(\lambda)=\lambda c/(c^2-1)$.  
The initial-value problem satisfied by $\bF^\mathrm{H}(z,\nu)$ is
\begin{equation}
\bF_z^\mathrm{H}(z,\nu)-\bA^\mathrm{H}(z,\nu)\bF^\mathrm{H}(z,\nu)=0,\quad
\bF^\mathrm{H}(0,\nu)=\Id,
\end{equation}
where the coefficient matrix is entire in the spectral parameter $\nu$ and is given by
\begin{equation}
\bA^\mathrm{H}(z,\nu)=\begin{pmatrix}0 & 1\\\nu-P(z) & 0\end{pmatrix},\quad P(z):=\frac{V''(f(z))}{c^2-1}.
\end{equation}
Again, $\bF^\mathrm{H}(z,\lambda)$ is given by a power series with infinite radius of (uniform, for $z\in [0,T]$) convergence:
\begin{equation}
\label{eq:Hill-series}
\bF^\mathrm{H}(z,\nu)=\sum_{n=0}^\infty \nu^n\bF_n^\mathrm{H}(z).
\end{equation}
Since Hill's equation \eqref{eq:hill} coincides with the equation \eqref{eq:spectral} for $\lambda=0$, the zero-order term is given by
$\bF^\mathrm{H}_0(z)=\bF_0(z)=\bF(z,0)$, 
and it has been obtained already in \S\ref{secsolsat0}.  Setting $z=0$ in \eqref{eq:Hill-series}, one obtains
\begin{equation}
\sum_{n=1}^\infty\nu^n\bF_n^\mathrm{H}(0)=0,\quad\nu\in\C,
\end{equation}
because $\bF^\mathrm{H}(0,\nu)=\Id$ for all $\nu$.  Since $\nu$ is arbitrary, it follows that $\bF^\mathrm{H}_n(0)=0$ for all $n\ge 1$.  By collecting the terms with the same powers of $\nu$, it follows that the subsequent terms in the series \eqref{eq:Hill-series} satisfy the initial-value problems
\begin{equation}
\bF_{nz}^\mathrm{H}(z)-\bA^\mathrm{H}(z,0)\bF^\mathrm{H}_n(z)=\sigma_-\bF^\mathrm{H}_{n-1}(z),\; \bF^\mathrm{H}_n(0)=0,\;\text{where}\;\sigma_-:=\begin{pmatrix}0 & 0\\1 & 0\end{pmatrix},
\end{equation}
for $n\ge 1$.  These equations can be solved in turn by means of the method of variation of parameters, i.e., by making the substitution $\bF_n^\mathrm{H}(z)=\bF_0(z)\bG_n(z)$, where $\bG_n(z)$ is to be regarded as a new unknown satisfying the induced initial condition $\bG_n(0)=0$.
Since $\bF_{0z}(z)=\bA^\mathrm{H}(z,0)\bF_0(z)$, the terms involving the non-constant coefficient matrix $\bA^\mathrm{H}(z,0)$ cancel, and then $\bG_n(z)$ can be obtained by a quadrature.  This method results in the following explicit recursive formulae for the coefficients $\bF^\mathrm{H}_n(z)$:
\begin{equation}
\bF^\mathrm{H}_n(z)=\bF_0(z)\int_0^z\bF_0(y)^{-1}\sigma_-\bF^\mathrm{H}_{n-1}(y)\,dy,\quad n\ge 1.
\end{equation}
In particular, for $n=1$ this gives
\begin{equation}
\bF^\mathrm{H}_1(z)=\bF_0(z)\int_0^z\bF_0(y)^{-1}\sigma_-\bF_0(y)\,dy,
\end{equation}
and setting $z=T$ in \eqref{eq:Hill-series} we obtain
the corresponding series for $\bM^\mathrm{H}(\nu)$:
\begin{equation}
\bM^\mathrm{H}(\nu)=\bM(0)+\sum_{n=1}^\infty\nu^n\bM^\mathrm{H}_n,\quad \bM^\mathrm{H}_1=\bM(0)\int_0^T\bF_0(y)^{-1}\sigma_-\bF_0(y)\,dy.
\label{eq:Hill-M1}
\end{equation}

Now, invoking Lemma~\ref{lem:floquetrelation} with $\alpha(\lambda)=\lambda c/(c^2-1)$, and recalling the quadratic mapping $\lambda\mapsto\nu(\lambda)$ defined in \eqref{eq:hill}, we obtain the first two correction terms in the series \eqref{eq:F-Taylor}:
\begin{equation}
\begin{split}
\bF_1(z)&=\frac{cz\bF_0(z)}{c^2-1}+\frac{c[\sigma_-,\bF_0(z)]}{c^2-1}\\
\bF_2(z)&=\tfrac{1}{2}\frac{c^2z^2\bF_0(z)}{(c^2-1)^2} + \frac{c^2z[\sigma_-,\bF_0(z)]}{(c^2-1)^2}
-\frac{c^2\sigma_-\bF_0(z)\sigma_-}{(c^2-1)^2} \\& \quad\quad\quad{}+\frac{\bF_0(z)}{(c^2-1)^2}\int_0^z
\bF_0(y)^{-1}\sigma_-\bF_0(y)\,dy,
\end{split}
\label{eq:F1-F2}
\end{equation}
where $[\bA,\bB]:=\bA\bB-\bB\bA$ denotes the matrix commutator.
Setting $z=T$ in the series formula \eqref{eq:F-Taylor} gives the series for the monodromy matrix $\bM(\lambda)$, also an entire function of $\lambda$:
\begin{equation}
\bM(\lambda)=\sum_{n=0}^\infty \lambda^n\bM_n,\quad \bM_n:=\bF_n(T).
\label{eq:Monodromy-series}
\end{equation}
Of course $\bM_0=\bF_0(T)=\bM(0)$, and this matrix was obtained by two equivalent calculations in \S\ref{secsolsat0}.  From \eqref{eq:F1-F2} we then obtain
\begin{equation}
\begin{split}
\bM_1&=q\bM(0)+\frac{c[\sigma_-,\bM(0)]}{c^2-1}\\
\bM_2&=\tfrac{1}{2}q^2\bM(0)+\frac{cq[\sigma_-,\bM(0)]}{c^2-1}
-\frac{c^2\sigma_-\bM(0)\sigma_-}{(c^2-1)^2}\\&\quad\quad\quad{} + \frac{\bM(0)}{(c^2-1)^2}\int_0^T
\bF_0(y)^{-1}\sigma_-\bF_0(y)\,dy.
\end{split}
\label{eq:M1-M2}
\end{equation}


\section{Stability indices}
\label{section:stability-indices}
In this section, we use the asymptotic expansion of the monodromy matrix $\bM(\lambda)$ valid for small $\lambda$ that was obtained in the preceding section to determine stability properties
of the periodic wavetrain $f$ solving the Klein-Gordon equation \eqref{eqnlKG}.  These stability properties are conveniently described in terms of two \emph{indices}, or signs.  These provide tests for stability that are easy to implement in practice.
\subsection{The parity index}
\label{sec:parity-index}
Following Bronski and Johnson \cite{BrJo2}, we define a \emph{parity index} (also called \emph{orientation index}) for periodic waves, which is analogous to the stability index for solitary waves \cite{San,San3}. Basically, it compares the Evans function near $\lambda = 0$ with its asymptotic behavior for large $\lambda$ along the real line. If the signs are different, then there must be an odd number of zeroes along the positive real line, indicating the existence of real unstable eigenvalues. 

We have already seen by two different calculations in \S\ref{secsolsat0} that $\lambda=0$ belongs to the spectrum $\sigma$.  In fact, $\lambda=0$ belongs to the periodic partial spectrum $\sigma_0$, as both Floquet multipliers coincide at $\mu=1$ (with the formation of a Jordan block in the monodromy matrix $\bM(0)$ in the generic case $T_E\neq 0$).  At a physical level, this is related to the translation invariance of the periodic traveling wave present because \eqref{eqnlKG} is an autonomous equation.  Recall (cf.\@ Remark~\ref{remark:periodic}) that the periodic eigenvalues (the points of the periodic partial spectrum $\sigma_0$) are the roots of the (entire) periodic Evans function
$D(\lambda,\mu)$ with $\mu=1\in S^1$, where $D$ is defined in \eqref{eq:det0}.  Expanding out the determinant and setting $\mu=1$ gives the formula
\begin{equation}
\label{Donreal}
 D(\lambda,1) = 1 - \tr(\bM(\lambda)) + \det (\bM(\lambda)). 
\end{equation}
To define the parity index we will consider the restriction of this formula to $\lambda\in\R$.
\begin{lemma}
\label{Dreal}
The restriction of the periodic Evans function $D(\lambda,1)$ to $\lambda\in\R$ is a real-analytic function.  Moreover, for $\lambda \in \R_+$ with $\lambda \gg 1$ sufficiently large, we have $\sgn (D(\lambda,1)) = \sgn(c^2-1)$. 
\end{lemma}
\begin{proof}
The system \eqref{eq:firstorder} has real coefficients whenever $\lambda \in \R$. Therefore the fundamental solution matrix $\bF(z,\lambda)$ is real for real $\lambda$ and $z\in [0,T]$.  By evaluation at $z=T$ the same is true for the elements of the monodromy matrix $\bM(\lambda)$,
and this proves the real-analyticity.   
When $\lambda$ is large in magnitude, then $\lambda^2+V''(f(z))\approx \lambda^2$, and hence
the first-order system \eqref{eq:firstorder} can be approximated by a constant-coefficient system:
\begin{equation}
\label{constcoefsyst}
\bw_z=\bA^\infty(\lambda)\bw,\quad \bA^\infty:=\begin{pmatrix}0 & 1\\ \displaystyle -\frac{\lambda^2}{c^2-1} & \displaystyle\frac{2c\lambda}{c^2-1}\end{pmatrix}.
\end{equation}
The fundamental solution matrix of this approximating system is the matrix exponential $\bF^\infty(z,\lambda)=e^{z\bA^\infty(\lambda)}$, and the corresponding monodromy matrix is 
$\bM^\infty(\lambda)=e^{T\bA^\infty(\lambda)}$.
The eigenvalues of $\bA^\infty(\lambda)$ are $\lambda/(c\pm 1)$, and hence those of $\bM^\infty(\lambda)$ are $e^{\lambda T/(c\pm 1)}$.  The periodic Evans function associated with the approximating system is therefore
\begin{equation}
D^\infty(\lambda,1)=1-\tr(\bM^\infty(\lambda))+\det(\bM^\infty(\lambda))=(e^{\lambda T/(c+1)}-1)(e^{\lambda T/(c-1)}-1),
\end{equation}
and this real-valued function of $\lambda\in\R$ clearly has the same sign as does $c^2-1$ for 
large positive $\lambda$.
The coefficient matrix $\bA^\infty(\lambda)$ of system \eqref{constcoefsyst} is an accurate approximation of that of the system  \eqref{eq:firstorder} uniformly for $z\in [0,T]$, so the respective Evans functions $D^\infty(\lambda,1)$ and $D(\lambda,1)$ are close to each other in the limit $\lambda\to\infty$ \cite{San}. This shows that $D(\lambda,1)$ has the same sign as does $c^2-1$ for $\lambda \gg 1$.
\end{proof}

As is typical for equations with Hamiltonian structure, the first derivative of the Evans function vanishes at $\lambda = 0$, and we need to compute (at least) the second derivative at the origin to obtain local information for small $\lambda$. 
\begin{lemma}
\label{lemma:Evans-small-lambda}
The periodic Evans function $ D(\cdot,1): \R \to \R$ vanishes to even order at $\lambda=0$, and satisfies
\begin{equation}
 D(0,1) =  D_\lambda(0,1) = 0, \quad 
 D_{\lambda\lambda}(0,1) =  2(q^2-\kappa),
\label{secondderiv}
\end{equation}
where $q$ is defined by \eqref{eq:q-def} and where 
\begin{equation}
\kappa:=\frac{M_{12}(0)}{(c^2-1)^2}\int_0^TF_{11}(y,0)^2\,dy.
\label{defDelta}
\end{equation}
\begin{proof}
Setting $\mu=1$ in \eqref{eq:tilde-D} shows that $e^{-q\lambda}D(\lambda,1)$ is even in $\lambda$.
Hence if $D(\lambda,1)$ vanishes at the origin it necessarily does so to even order.
Recalling \eqref{eq:Abels-Identity} (which also
shows that the approximation of $\det(\bM(\lambda))$ by $\det(\bM^\infty(\lambda))$
in the proof of Lemma~\ref{Dreal} is exact)  and expanding through order $O(\lambda^2)$ gives
\begin{equation}
\det(\bM(\lambda))=1+2q\lambda + 2q^2\lambda^2 + O(\lambda^3),\quad \lambda\to 0.
\label{eq:det-M-expansion}
\end{equation}
Taking the trace in the series  \eqref{eq:Monodromy-series} with the help of 
Proposition~\ref{prop:Jordan} and  \eqref{eq:M1-M2}, 
\begin{equation}
\tr(\bM(\lambda))=2 + 2q\lambda +\left(q^2+\kappa\right)\lambda^2 + O(\lambda^3),\quad\lambda\to 0.
\label{eq:tr-M-expansion}
\end{equation}
Substitution of these expansions into \eqref{Donreal} gives the expansion of the periodic Evans function about $\lambda=0$ as
\begin{equation}
D(\lambda,1)=\left(q^2-\kappa\right)\lambda^2 + O(\lambda^3)\,\quad\lambda\to 0,
\end{equation}
and this completes the proof of the Lemma.
\end{proof}
\end{lemma}
The idea of comparing the behavior of the real-valued function $D(\cdot,1)$ for large and small positive $\lambda$ as described by Lemma~\ref{Dreal} and Lemma~\ref{lemma:Evans-small-lambda} respectively suggests the definition of an index relating these two extremes:
\begin{definition}[parity index $\gamma$]
Suppose $D(\cdot,1)$ vanishes to even order $2p\ge 2$ at $\lambda=0$. The \emph{parity (or orientation) index} is given by
\begin{equation}
\label{pariryindx}
\gamma := \sgn \left( (c^2-1)  \partial_\lambda^{2p}D(0,1)\right).
\end{equation}
\end{definition}

We then have the following instability criterion.
\begin{proposition}
\label{thmparitycriterion}
If $\gamma=1$ (resp., $\gamma = -1$) then the number of positive real points in the periodic partial spectrum $\sigma_0\subset\sigma$, i.e., periodic eigenvalues, is even (resp., odd) when counted according to multiplicity. In particular, if $\gamma=-1$ there is at least one positive real periodic eigenvalue and hence the underlying periodic wave $f$ solving the Klein-Gordon equation \eqref{eqnlKG} is spectrally unstable, with the corresponding exponentially growing solution of the linearized equation \eqref{linearperturb} having the same spatial period $T$ as $f_z$.  
\end{proposition}
\begin{proof}
If $\gamma=1$, then $D(\lambda,1)$ has the same sign for sufficiently small and sufficiently large strictly positive $\lambda$, while if $\gamma=-1$ the signs are opposite for small and large $\lambda$.   Since $D(\lambda,1)$ is real-analytic for real $\lambda$ it clearly has an even number of positive roots for $\gamma=1$ and an odd number of positive roots for $\gamma=-1$, with the roots weighted by their multiplicities.  By Proposition~\ref{prop:deteq0}, these roots correspond to points in the spectrum $\sigma$, and since $\mu=1$, they are periodic eigenvalues.
\end{proof}

The case $\gamma=1$ is, of course, inconclusive for stability, because it only guarantees that the number of real positive (periodic) eigenvalues is even (possibly zero). This is why $\gamma$ is  called a \emph{parity} index \cite{San3}.  The alternative name of orientation index is suggested by the way that $\gamma$ compares the asymptotic directions of the graph of $D(\cdot,1)$.
We can immediately apply Proposition~\ref{thmparitycriterion} to establish the following: 
\begin{theorem} \label{th:sublibunstable}
Let the potential $V$ satisfy Assumptions~\ref{assumptionsV} and \ref{assumptionsnormalize}.  Then, subluminal librational periodic traveling wave solutions of the Klein-Gordon equation \eqref{eqnlKG} for which $T_E< 0$ are spectrally unstable, having a positive real periodic eigenvalue $\lambda\in \sigma_0\subset\sigma$.  
\end{theorem}
\begin{proof}
We use Proposition~\ref{prop:Jordan} to express $M_{12}(0)$ in the formula \eqref{secondderiv} for $D_{\lambda\lambda}(0,1)$ in terms of $c$ and $T_E$.  Noting that the integral of $F_{11}(y,0)^2$ is strictly positive, it then follows from the hypotheses that $(c^2-1)T_E> 0$, and hence that $D_{\lambda\lambda}(0,1)>0$ because $\kappa<0$. Hence $D(\cdot,1)$ vanishes to precisely second order at the origin and $\gamma=-1$ as $c^2-1<0$;  instability is thus predicted according to Proposition~\ref{thmparitycriterion}.
\end{proof}

As a consequence of Remark~\ref{remark:Chicone-sine-Gordon}, we have the following.
\begin{corollary}
\label{corollary:parity-index-sine-Gordon}
All subluminal librational periodic traveling wave solutions of the sine-Gordon equation ($V(u)=-\cos(u)$) are spectrally unstable, having a positive real periodic eigenvalue $\lambda\in\sigma_0\subset\sigma$.
\end{corollary}

\begin{remark}
In the case of superluminal librational waves satisfying instead $T_E> 0$,  the index $\gamma$ satisfies (by similar arguments) 
$\gamma=1$.
The number of positive periodic eigenvalues is necessarily even in this case, but as this number can be zero, instability cannot be concluded.  For rotational waves of both types, as well as for subluminal (resp., superluminal) librational waves with $T_E>0$ (resp., $T_E<0$), Proposition~\ref{prop:rotational-monotone} indicates that the two terms in the sum \eqref{secondderiv} for $D_{\lambda\lambda}(0,1)$ have opposite signs, and therefore without further information we cannot even reliably calculate the index $\gamma$ for these remaining cases.
\label{remark:parity-index-troubles}
\myendrmk
\end{remark}

A parallel theory can be developed for the ``antiperiodic Evans function'' $D(\lambda,-1)$.  This function is also real-valued for real $\lambda$, however it is easy to check that it is strictly positive for $\lambda=0$ and in the limit $\lambda\to \pm\infty$, and therefore the only conclusion is that the number of positive real antiperiodic eigenvalues is always even (possibly zero).

The instability detected by the parity index corresponds to a perturbation with spatial period $T$ and with strictly positive exponential growth rate $\lambda$.
A similar phenomenon occurs in the stability analysis of solitary waves \cite{BD1,PW1}, with the corresponding Evans function having a root at some real value (growth rate) bounded away from zero.  We will turn out attention next toward the introduction of an index capable of detecting instabilities of a different type, namely those having arbitrarily small exponential growth rates.

\subsection{The modulational instability index}
\label{secmodinst}
The asymptotic analysis of the monodromy matrix $\bM(\lambda)$ in the limit $\lambda\to 0$ 
already played a role in the theory of the parity index (cf.\@ Lemma~\ref{lemma:Evans-small-lambda}).  Now we revisit the corresponding expansions to extract information about the behavior of the spectrum $\sigma$ (and not just the periodic partial spectrum $\sigma_0$) in a \emph{complex} neighborhood of the origin in the complex plane.  We will also compute expansions of the Floquet multipliers near the origin, and of the function $D(\lambda,\mu)$ in a full complex neighborhood of $(\lambda,\mu)=(0,1)$.

Recall that, given $\lambda\in\C$, the Floquet multipliers $\mu=\mu(\lambda)$ are defined as the roots of the characteristic equation $D(\lambda,\mu)=0$, i.e., they are the eigenvalues of the monodromy matrix $\bM(\lambda)$.  The quadratic formula gives the multipliers in the form
\begin{equation}
\label{valuesFloquetm}
\mu=\mu_\pm(\lambda) = \frac{1}{2} \Big( \tr (\bM(\lambda)) \pm \big((\tr (\bM(\lambda)))^2 - 4 \det( \bM(\lambda))\big)^{1/2} \Big). 
\end{equation}
To analyze the multipliers near $\lambda=0$, we first calculate the quadratic discriminant with the help of the expansions of $\det(\bM(\lambda))$ and $\tr(\bM(\lambda))$ recorded in \eqref{eq:det-M-expansion} and \eqref{eq:tr-M-expansion} respectively.  We obtain:
\begin{equation}
(\tr(\bM(\lambda)))^2-4\det(\bM(\lambda))=4\kappa\lambda^2 + O(\lambda^3),\quad\lambda\to 0,
\end{equation}
where $\kappa$ is given by \eqref{defDelta}.

It is obvious that the quadratic discriminant is an analytic function of $\lambda$ that vanishes to at least second order at $\lambda=0$.
More generally, it is easy to apply the reasoning described in Remark~\ref{rem:negsim} to show that
the quadratic discriminant necessarily vanishes to even order at $\lambda=0$.  This implies that,
by proper accounting of the square roots, the two Floquet multipliers may be regarded as analytic functions of $\lambda$ in a complex neighborhood of $\lambda = 0$.
This is an unusual situation in eigenvalue perturbation theory, given that when $\lambda=0$ the ``unperturbed'' monodromy matrix is generally not diagonalizable according to Proposition~\ref{prop:Jordan}.  
In general, the double eigenvalue $\mu(0) = 1$ of $\bM(0)$ would be expected to bifurcate under perturbation in a way that can be described by a Puiseux (fractional power) series about $\lambda = 0$ \cite[pg.\@ 65]{Kat1}. Thanks to the special structure of the differential equation behind the monodromy matrix, however, this expansion gets simplified into a standard power series in $\lambda$ in the present context.

Again using \eqref{eq:tr-M-expansion}, the first few terms in the Taylor series about $\lambda=0$ of the (analytic) Floquet multipliers are:
\begin{equation}
\label{asymptoticFloq}
\mu_\pm(\lambda) = 1 + \left(q \pm \kappa^{1/2} \right) \lambda + O(\lambda^2),\quad\lambda\to 0. 
\end{equation}
Noting that $\kappa$ is proportional to the monodromy matrix element $M_{12}(0)$ by strictly positive factors because $F_{11}(z,0)$ is a differentiable function satisfying $F_{11}(0,0)=1$, this formula motivates the definition of another instability index:
\begin{definition}[modulational instability index $\rho$]
\label{defmodinstindx}
The \emph{modulational instability index} is given by 
\begin{equation}
\label{defmodinst}
\rho := \sgn(M_{12}(0)),
\end{equation}
with the understanding that $\rho:=0$ if $M_{12}(0)=0$.
\end{definition}

Before we discuss the implications of this definition, we explain how the modulational instability index can be effectively computed.
Firstly, according to Proposition~\ref{prop:Jordan} we have the following:
\begin{proposition}
The modulational instability index $\rho$ may be explicitly expressed in terms of the wave speed $c$ and the derivative of the period $T$ with respect to energy $E$ as
\begin{equation}
\rho=\sgn(-(c^2-1)T_E),
\end{equation}
with the understanding that $\rho=0$ if $T_E=0$.
In particular, from Proposition~\ref{prop:rotational-monotone}, $\rho=1$ for rotational waves of any speed.
\label{corollary:index-period}
\end{proposition}
Next, recall the quantity $\delta$ defined in \S\ref{sec:alternate} in terms of the wave profile $f$ either by the contour integral \eqref{eq:delta-analytic} or more generally by the finite-part expression \eqref{defDeltah}.
Combining \eqref{eq:delta-identity} with Proposition~\ref{corollary:index-period} yields another formula for the modulational instability index $\rho$:
\begin{proposition}
The modulational instability index $\rho$ may be explicitly expressed in terms of the wave profile function $f$ alone as
\begin{equation}
\rho=\sgn(\delta),
\end{equation}
with the understanding that $\rho=0$ if $\delta=0$.  Here $\delta$ is defined in terms of $f(\cdot)$ by \eqref{defDeltah}, or in special situations, by \eqref{eq:delta-analytic}.
\label{corollary:index-profile}
\end{proposition}

It seems clear that the modulational instability index $\rho$ has something to do with the spectrum near the origin because it is the sign of $\kappa$ which appears under a radical in the small-$\lambda$ approximation \eqref{asymptoticFloq} of the Floquet multipliers.  To make this precise, we return to the notion that curves of spectrum (cf.\@ Proposition~\ref{prop:curves}) may be parametrized implicitly by $\theta\in\R$ via the equation $D(\lambda,e^{i\theta})=0$.  Since the Floquet multipliers $\mu$ bifurcate from $\mu=1$ corresponding to $\theta=0$, we will require an expansion of the analytic function $D(\lambda,e^{i\theta})$ near $(\lambda,\theta)=(0,0)$.  

\begin{lemma}\label{lem:expandevans}
The periodic Evans function $D(\lambda,e^{i\theta})$ is analytic in the variables $(\lambda,\theta)\in\C^2$ and has the following expansion in a neighborhood of $(\lambda,\theta) = (0,0)$:
\begin{equation}
\label{expansionDelta}
D(\lambda,e^{i\theta}) = - \kappa \lambda^2 + \left( i \theta - q \lambda \right)^2 + O(3), 
\end{equation}
where $O(3)$ denotes terms of order three or higher in $(\lambda,\theta)$, and where $\kappa$ and $q$ are defined by \eqref{defDelta} and \eqref{eq:q-def} respectively.
\end{lemma}
\begin{proof}
This follows from the formula
$D(\lambda,e^{i\theta}) =  e^{2i\theta} - \tr (\bM(\lambda)) e^{i\theta} + \det (\bM(\lambda))$
upon expanding the exponentials in power series about $\theta=0$, substituting the expansions
\eqref{eq:det-M-expansion} and \eqref{eq:tr-M-expansion}, and using the definition \eqref{defDelta}.
\end{proof}

Using this expansion of $D(\lambda,e^{i\theta})$, we analyze how solutions to $D(\lambda,e^{i\theta}) = 0$ with $(\lambda,\theta)\in \C \times \R$ bifurcate from $(0,0)$. We will see that the modulational instability index $\rho$ determines exactly whether or not the spectral curves are tangent to the imaginary axis at the origin.  Since the spectrum is understood in the case $c=0$ by Proposition~\ref{prop:zero-velocity-agree}, and since $q=0$ if and only if $c=0$, we will assume that $q\neq 0$.
\begin{lemma}
\label{lemcrossortan}
If $\rho = 1$ but $\kappa\neq q^2>0$, then the equation $D(\lambda,e^{i\theta}) = 0$ parametrically describes (for small real $\theta$) 
two smooth curves 
passing through the origin tangent to the imaginary axis in a neighborhood of the origin the complex $\lambda$-plane.  

If $\rho = -1$ then the equation $D(\lambda,e^{i\theta})=0$ instead parametrically describes two distinct smooth curves 
that cross at the origin with tangent lines making
acute non-zero angles with the imaginary axis.
\end{lemma}
\begin{proof}
The result will follow from an application of the Implicit Function Theorem. 
It is easy to check from the expansion \eqref{expansionDelta} that while $D(\lambda,e^{i\theta})$ vanishes for $\lambda=\theta=0$, so does $D_\lambda(\lambda,e^{i\theta})$. Therefore we cannot use the Implicit Function Theorem directly to solve for $\lambda$ in terms of $\theta$.
However, the issue at hand is not that there is no smooth solution $\lambda=\lambda(\theta)$ of the equation $D(\lambda,e^{i\theta})=0$, but rather that there are two such solutions which can be separated by an appropriate change of coordinates.  

Let us introduce the ``blow-up'' transformation of coordinates 
\begin{equation}
(\lambda,\theta)\mapsto (s,\theta)\quad\text{where}\quad \lambda=is\theta.  
\label{eq:monoidal-transformation}
\end{equation}
Substituting $\lambda=is\theta$ into the expansion \eqref{expansionDelta} yields
$D(is\theta,e^{i\theta})=\theta^2\hat{D}(s,\theta)$, where
\begin{equation}
\hat{D}(s,\theta):=\kappa s^2 - \left(1-qs\right)^2+\theta R(s,\theta),
\label{eq:D-zero-blown-up}
\end{equation}
where $R$ is an entire analytic function of its arguments.  Omitting the ``exceptional fiber'' associated with $\theta=0$, the equation $D=0$ becomes $\hat{D}=0$.
We want to solve for $s=s(\theta)$ near $\theta=0$; since $\kappa\neq 0$ and $\kappa\neq q^2$,  when $\theta=0$,
$\hat{D}=0$  is a quadratic equation in $s$ with two distinct roots $s=s_0^\pm:=(q\pm\kappa^{1/2})^{-1}$.
This implies that
both partial derivatives $\hat{D}_s(s_0^+,0)$ and $\hat{D}_s(s_0^-,0)$ are nonzero.
Therefore, the 
Implicit Function Theorem applies and guarantees the existence of
two unique analytic functions $s=s^\pm(\theta)$ solving $\hat{D}(s^\pm(\theta),\theta)=0$ for small $\theta$ with $s^\pm(0)=s_0^\pm$.
Recalling the change of coordinates \eqref{eq:monoidal-transformation}, 
there are two curves of spectrum $\sigma$ through the origin $\lambda=0$ having parametric form
\begin{equation}
\lambda=\lambda^\pm(\theta):=i\theta s^\pm(\theta)=is_0^\pm\theta +O(\theta^2),\quad\theta\to 0.
\end{equation}
Note that since $s_0^\pm\neq 0$, these parametric equations both define smooth curves with well-defined tangents at $\lambda=0$.  Finally, we note that 
$s_0^\pm$ are both real if and only if $\kappa>0$, in which case both curves of spectrum $\lambda=\lambda^\pm(\theta)$ passing through the origin are tangent to the imaginary axis.  (It is possible in this case that the two curves actually coincide.)
If instead $\kappa<0$, then $s_0^\pm$ form a complex-conjugate pair with nonzero real and imaginary parts: $s_0^\pm=a\pm ib$, with $a,b\neq 0$.  Therefore, the two curves of spectrum $\lambda=\lambda^\pm(\theta)$ cross at the origin with the tangent lines making acute nonzero angles with the imaginary axis of size $\phi=\arctan(|b/a|)$.
This completes the proof of the lemma.
\end{proof}
If $\rho=0$, then it also holds that $\sigma$ consists locally of two curves tangent to the imaginary axis at $\lambda=0$, with the degree of tangency increasing with the (even) order of vanishing of $(\tr(\bM(\lambda)))^2-4\det(\bM(\lambda))$ at $\lambda=0$.
The local structure of the spectrum near the origin is illustrated in Figure~\ref{fig:crossandtan}.
\begin{figure}[h]
\begin{center}
\includegraphics{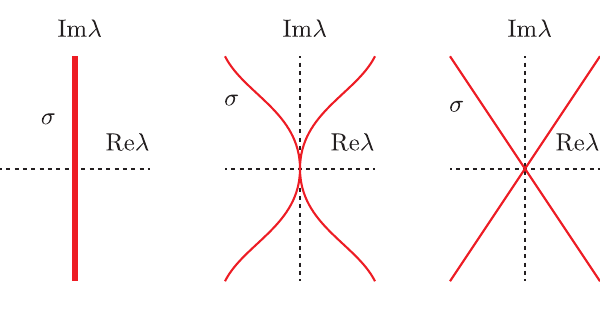}
\end{center}
\caption{A qualitative sketch of the three generic possibilities for the spectrum $\sigma$ in a neighborhood of the origin.  
Left panel:  two curves coincide exactly with the imaginary axis.  Center panel:  two distinct curves of spectrum each of which is tangent to the imaginary axis.  Right panel:  two distinct curves of spectrum crossing the imaginary axis transversely.
The left two panels illustrate possible configurations in the case $\rho=1$ with $\kappa\neq q^2$ or $\rho=0$, while the right-hand panel illustrates the configuration if $\rho=-1$.
In the terminology of Definition~\ref{def:modulational-instability}, the left-hand panel illustrates modulational stability, the central panel illustrates weak modulational instability, and the right-hand panel illustrates strong modulational instability.  
}
\label{fig:crossandtan}
\end{figure}
\begin{remark}
The special case when $\kappa=q^2$ (equivalently, $D_{\lambda\lambda}(0,1)=0$ by \eqref{secondderiv}) is more complicated.   If $D(\cdot,1)$ vanishes at $\lambda=0$ precisely to (even) order $2p>2$, then it can be shown that $D(\lambda,e^{i\theta})$ has the expansion
\begin{equation}
D(\lambda,e^{i\theta})=-\theta^2-2iq\lambda\theta + C\lambda^{2p}+O(\theta^3) + O(\lambda\theta^2)+O(\lambda^2\theta) + O(\lambda^{2p+1}),
\end{equation}
where
\begin{equation}
C:=\frac{\partial_\lambda^{2p}D(0,1)}{(2p)!}=\frac{(2q)^{2p}-\tr(\bM^{(2p)}(0))}{(2p)!}\neq 0.
\end{equation}
When $\lambda=0$, the equation $D=0$ has a double root at $\theta=0$.  To unfold this double root, let $\theta= t\lambda$ for a new unknown $t$, and obtain $D(\lambda,e^{it\lambda})=\lambda^2\check{D}(\lambda,t)$, where
\begin{equation}
\check{D}(\lambda,t)=-t^2-2iqt+C\lambda^{2p-2} + O(\lambda t^3)+O(\lambda t^2)+O(\lambda t)+O(\lambda^{2p-1}).
\end{equation}
Omitting the exceptional fiber, we solve $\check{D}(\lambda,t)=0$ for $t$ when $\lambda$ is small.
At $\lambda=0$ there are two simple roots, $t=-2iq$ and $t=0$, so by the Implicit Function Theorem there are two analytic solutions of $\check{D}(\lambda,t)=0$:  $t=t_1(\lambda)$ and $t=t_2(\lambda)$ satisfying $t_1(0)=-2iq$ and $t_2(0)=0$.  The solution $t=t_2(\lambda)$ can be seen to satisfy $t_2(\lambda)=C\lambda^{2p-2}/(2iq)+O(\lambda^{2p-1})$ as $\lambda\to 0$.  Returning to the original equation $D(\lambda,e^{i\theta})=0$ we have found two analytic solutions $\theta=\theta(\lambda)$:
\begin{equation}
\theta=\theta_1(\lambda)=-2iq\lambda+O(\lambda^2)\quad\text{and}\quad
\theta=\theta_2(\lambda)=\frac{C}{2iq}\lambda^{2p-1}+O(\lambda^{2p}).
\end{equation}
If $q\neq 0$ (i.e., $c\neq 0$), then $\theta=\theta_1(\lambda)$ can be solved for $\lambda$ to yield a curve of spectrum given by $\lambda=i\theta/(2q)+O(\theta^2)$, which is clearly tangent to the imaginary axis for $\theta\in\mathbb{R}$.  On the other hand, the equation $\theta=\theta_2(\lambda)$ is a normal form of the type described in Remark~\ref{remark:normal-forms}, and solving for $\lambda$ in terms of $\theta$ one finds $2p-1$ curves of spectrum crossing at the origin with equal angles, and with one of the curves tangent to the imaginary axis.  Since $p>1$, there exist curves of spectrum emanating from the origin into the left and right half-planes making nonzero angles with the imaginary axis.  Representative plots are shown in Figure~\ref{fig:DegenerateSpectrum}.  The mechanism for the sudden appearance,  as lower-order derivatives of $D(\cdot,1)$ tend to zero, of spectral arcs through the origin that are not tangent to $i\mathbb{R}$ is the collision of smooth arcs of unstable spectrum   with the origin.  Therefore spectral instability is expected not just for $D_{\lambda\lambda}(0,1)=0$ but also for $D_{\lambda\lambda}(0,1)$ sufficiently small.
\begin{figure}[h]
\begin{center}
\includegraphics{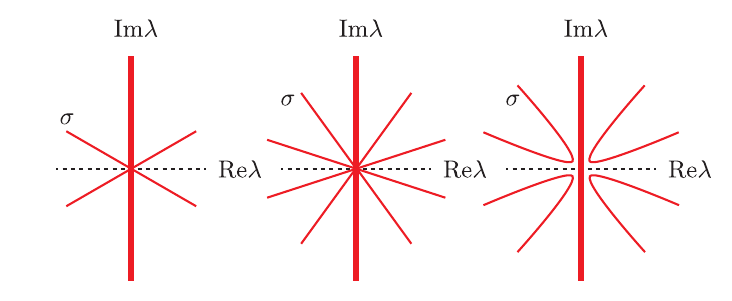}
\end{center}
\caption{The spectrum near the origin when $D(\cdot,1)$ vanishes to order $2p$ with $p=2$ (left panel) and $p=3$ (center panel).  Two curves are tangent to $i\mathbb{R}$ (illustrated as if they coincide with multiplicity two), and other curves of simple spectrum make nonzero angles with $i\mathbb{R}$.  According to Definition~\ref{def:modulational-instability}, strong modulational instability is predicted for all $p>1$.  The right panel illustrates the spectrum near the origin for $\partial^2_\lambda D(0,1)$ and $\partial^4_\lambda D(0,1)$ both small but nonzero.}
\label{fig:DegenerateSpectrum}
\end{figure}
\label{remark:one-curve}
\myendrmk
\end{remark}

Of course the presence of any spectrum $\sigma$ that is not purely imaginary implies spectral instability according to Definition~\ref{def:spectral-instability}. The particular type of instability detected by the condition $\rho=-1$ is called a \emph{modulational instability}, for which we give the following formal definition.
\begin{definition}[modulational stability and instability]
A periodic traveling wave solution $f$ of the Klein-Gordon equation \eqref{eqnlKG} is said to be
\emph{modulationally unstable} (or, to have a \emph{modulational instability}) if for every neighborhood $U$ of the origin $\lambda=0$, $(\sigma\setminus i\R)\cap U \neq\emptyset$.  Otherwise, $f$ is said to be \emph{modulationally stable}.
For an angle $\theta\in (0,\pi/2)$, let $S_\theta$ denote the union of the open sectors given by the inequalities $|\arg(\lambda)|<\theta$ or $|\arg(-\lambda)|<\theta$ (note $0\not\in S_\theta$).  A modulational instability is called \emph{weak} if for every $\theta\in (0,\pi/2)$ and for every neighborhood $U$ of the origin, $\sigma\cap U\cap S_\theta=\emptyset$.  A modulational instability that is not weak is called \emph{strong}.
\label{def:modulational-instability}
\end{definition}
Note that weak modulational instabilities correspond to the existence of curves of non-imaginary spectrum arbitrarily close to the origin that are nonetheless tangent to the imaginary axis at the origin.
There are two reasons for the terminology introduced in Definition~\ref{def:modulational-instability}.  Firstly, the unstable modes associated with the spectrum near $\lambda=0$ have Floquet multipliers very close to $1$, which in turn implies that the corresponding Floquet exponents $R(\lambda)$ (cf.\@ \eqref{eq:Bloch-form}) can be chosen to be very small imaginary numbers.   This makes the factor $e^{R(\lambda)z/T}$,  that modulates the periodic function $\bz(z,\lambda)$ in the Bloch representation of the mode, very slowly-varying.  The complex growth rate $\lambda$ of the mode is of course very small as well.  Hence the slow exponential growth of such a mode may appear to an observer as the formation of a slowly-varying modulation on the background of the underlying unstable periodic wave $f$.  (Note that the combination of slow spatial modulation and slow instability can only occur for unstable spectrum in a neighborhood of the origin, which implies that there are no additional instabilities that may appear as slow spatiotemporal modulations of the unstable wave $f$ that are not captured by Definition~\ref{def:modulational-instability}.)  Secondly, as we will show in \S\ref{secWhitham} below (cf.\@ Theorem~\ref{theorem:Whitham-mod-inst}), the presence/absence of a modulational instability turns out to be precisely correlated with the ellipticity/hyperbolicity of the system of modulation equations that arise in Whitham's fully nonlinear theory \cite{Wh} of slowly-modulated waves.  For this reason in particular, a modulational instability is sometimes called an instability of \emph{Whitham type}.   There is also a connection between the presence/absence of a modulational instability and the focusing/defocusing type of a certain nonlinear Schr\"odinger equation that arises in the weakly nonlinear modulation theory of near-equilibrium librational waves, as we will see in \S\ref{sec:NLS} (cf. Theorem~\ref{theorem:NLS}).

Observe also that a periodic traveling wave $f$ can be modulationally stable according to Definition~\ref{def:modulational-instability} without being spectrally stable in the sense of Definition~\ref{def:spectral-instability}, because $\sigma$ may coincide exactly with the imaginary axis in a neighborhood of the origin while containing values of $\lambda$ with $\Re\lambda\neq 0$ elsewhere.  For example, the parity index $\gamma$ defined in \S\ref{sec:parity-index} is designed to detect points of $\sigma$ that are real and not close to the origin, and hence that correspond to unstable modes exhibiting fairly rapid exponential growth in time. Also, in \S\ref{sec:super-rotation-unstable} we will show that superluminal rotational waves are spectrally unstable, although as we will see immediately they exhibit at worst weak modulational instability in the generic case $\kappa\neq q^2>0$ (and numerics suggest that they are in fact modulationally stable at least in the sine-Gordon case; see \cite{JMMP1}).

Proposition~\ref{corollary:index-period} and Lemma~\ref{lemcrossortan} imply the following instability result. 
\begin{theorem}\label{th:librational}
Let $V$ be a potential satisfying Assumptions~\ref{assumptionsV} and \ref{assumptionsnormalize}.
A librational periodic traveling wave solution of the Klein-Gordon equation \eqref{eqnlKG} for which $(c^2-1)T_E>0$ holds is strongly modulationally unstable.  
\end{theorem}
\begin{proof}
Computing the modulational instability index in the form given by Proposition~\ref{corollary:index-period}, one has $\rho=-1$ for all such waves.  Hence Lemma~\ref{lemcrossortan} guarantees the existence of two smooth curves of $\sigma$ crossing at $\lambda=0$, neither of which is tangent to the imaginary axis.
\end{proof}
Since modulational instability implies spectral instability, we have the following.
\begin{corollary}
All librational waves satisfying $(c^2-1)T_E>0$ are spectrally unstable.
\label{corollary:libration-spectral-instability}
\end{corollary}

Finally, recalling Remark~\ref{remark:Chicone-sine-Gordon}, we have the following.
\begin{corollary}
\label{corollary:sine-Gordon-libration-spectral-instability}
All librational traveling wave solutions of the sine-Gordon equation ($V(u)=-\cos(u)$) are strongly modulationally unstable and hence spectrally unstable.
\end{corollary}

\begin{remark}
By Proposition~\ref{prop:rotational-monotone}, the modulational instability index is $\rho=1$ for rotational waves of both sub- and superluminal types.  Hence in the generic case of $\kappa\neq q^2>0$, the spectrum $\sigma$ is locally tangent to the imaginary axis at the origin $\lambda=0$, but this is inconclusive for stability because these curves could fail to be confined to the imaginary axis, or because there could be other parts of the spectrum $\sigma$ with nonzero real parts far from  the origin.  In other words, there could be either a weak modulational instability, or an instability of non-modulational  type, neither of which can be detected by the modulational instability index.  
In the special case of $\kappa=q^2>0$, there is a strong modulational instability, according to Remark~\ref{remark:one-curve}.
\label{remark:modulational-instability-index-troubles}
\myendrmk
\end{remark}

\subsection{Application of the modulational instability index to the Hill discriminant}
\label{section:mod-index-Hill}
The index $\rho$ has other applications as well.  We will now use it to settle an issue left unresolved earlier (see Remark~\ref{rem:specH}).
The quantity $\Delta^\mathrm{H}(\nu):=\tr(\bM^\mathrm{H}(\nu))$ is called the \emph{Hill discriminant}.  Evaluating for $\nu=0$, where the equations \eqref{eq:spectral} and \eqref{eq:hill} coincide, gives $\Delta^\mathrm{H}(0)=\tr(\bM(0))=2$ by Proposition~\ref{prop:Jordan}.
The sign of the derivative $\Delta_\nu^{\mathrm{H}}(0)$ can be computed explicitly:
\begin{lemma}
\label{lemma:sign-of-discriminant}
The Hill discriminant satisfies $\sgn(\Delta_\nu^\mathrm{H}(0))=\rho$ (and $\Delta_\nu^\mathrm{H}(0)=0$ if $\rho=0$).
\end{lemma}
\begin{proof}
The monodromy matrix $\bM^\mathrm{H}(\nu)$ for Hill's equation was analyzed in a neighborhood of $\nu=0$ in \S\ref{sec:series-expansion}, where it was shown that the elements of $\bM^\mathrm{H}(\nu)$ admit convergent power series expansions about $\nu=0$.  By differentiation of the power series formula for $\bM^\mathrm{H}(\nu)$ given in \eqref{eq:Hill-M1}, it follows upon setting $\nu=0$ that
\begin{equation}
\bM^\mathrm{H}_\nu(0)=\bM_1^\mathrm{H}=\bM(0)\int_0^T\bF_0(y)^{-1}\sigma_-\bF_0(y)\,dy.
\end{equation}
Taking the trace and noting that $\bF_0(z)=\bF(z,0)$, one obtains
\begin{equation}
\Delta_\nu^\mathrm{H}(0)=M_{12}(0)\int_0^TF_{11}(y,0)^2\,dy.
\end{equation}
Since $\bF(z,0)$ is a fundamental solution matrix of a system of equations with real coefficients, and since $\bF(z,0)$ is normalized to the identity at $z=0$, $F_{11}(z)$ is a real-valued differentiable function satisfying $F_{11}(0,0)=1$, and hence the integral is strictly positive.
Therefore $\sgn(\Delta_\nu^\mathrm{H}(0))=\sgn(M_{12}(0))=\rho$ according to Definition~\ref{defmodinstindx}. 
\end{proof}
\begin{remark} \label{rem:mu0} 
Combining Proposition~\ref{corollary:index-period} with
Lemma~\ref{lemma:sign-of-discriminant} shows that if $f$ is a librational wave, then 
\begin{itemize}
\item
$\Delta_\nu^{\mathrm{H}}(0)<0$ if $(c^2-1)T_E>0$ and hence $\nu_2^{(0)}<\nu_1^{(0)}=0$;
\item
$\Delta_\nu^{\mathrm{H}}(0)>0$ if $(c^2-1)T_E<0$ and hence $0=\nu_2^{(0)}<\nu_1^{(0)}$;
\item
$\Delta_\nu^{\mathrm{H}}(0)=0$ if $T_E=0$, and hence $\nu_2^{(0)}=\nu_1^{(0)}=0$.
\end{itemize}
Here we recall that $\nu_1^{(0)}$ and $\nu_2^{(0)}$ are particular periodic eigenvalues of the Hill's spectrum for the librational wave $f$, as defined in \S\ref{section-Hill-spectrum}.
This concludes the characterization of the Hill's spectrum $\Sigma^\mathrm{H}$ corresponding to librational waves for $\nu$ near $\nu=0$ (cf.\@ Remark~\ref{rem:specH}).
\myendrmk
\end{remark}

\section{Stability properties of rotational waves} 
\label{secfurther}
According to Remark~\ref{remark:parity-index-troubles}, the parity index $\gamma$ is not easy to calculate for rotational waves. According to Remark~\ref{remark:modulational-instability-index-troubles}, the modulational instability index satisfies $\rho=1$ for rotational waves, but unfortunately this is inconclusive for stability.  Therefore, we have so far not been able to determine the spectral stability properties of rotational waves.  In this section, we remedy this by proving the following:
\begin{theorem} \label{th:rotational} Suppose that the potential $V$ satisfies Assumptions~\ref{assumptionsV} and \ref{assumptionsnormalize}.  Then we have the following: 
\begin{itemize}
\item[(i)] all periodic traveling waves of the Klein-Gordon equation \eqref{eqnlKG} of superluminal rotational type are spectrally unstable;
\item[(ii)] all periodic traveling waves of the Klein-Gordon equation \eqref{eqnlKG} of subluminal rotational type are spectrally stable. 
\end{itemize}
\end{theorem}
The proofs of statements (i) and (ii) are quite different and will be presented below in \S\ref{sec:super-rotation-unstable} and
\S\ref{sec:sub-rotation-stable} respectively.  The proof of (i) also implies spectral instability results for certain librational waves; these will be described briefly in Remarks~\ref{rem:libstab} and \ref{remark:librational-general-speed-Hill-gap}.
\subsection{Superluminal rotational waves} 
\label{sec:super-rotation-unstable}
We will prove spectral instability of superluminal rotational waves by showing the existence of a point $\lambda\in\sigma$ with $\Re\lambda\neq 0$.  Our method is based on the introduction of
a spectrum-detecting function $G:\C\to\R$ defined by
\begin{equation}
G(\lambda):=\log|\mu_+(\lambda)|\log|\mu_-(\lambda)|,
\label{eq:G-function-define}
\end{equation}
where $\mu=\mu_\pm(\lambda)$ are the two Floquet multipliers (eigenvalues of the monodromy
matrix $\bM(\lambda)$) associated with the superluminal rotational wave.  The function $G$ is well-defined on $\C$ and continuous, even at the branching points where the two multipliers degenerate, a fact following from the symmetric fashion in which the multipliers enter.  Moreover, according to Proposition~\ref{prop:deteq0}, $G(\lambda)=0$ if and only if $\lambda\in\sigma$.  This fact makes $G$ very similar to an Evans function; however $G$ is most certainly not analytic at any point --- indeed it is a real-valued function of a complex variable.  Analyticity of the Evans function is very useful in certain applications, especially numerical computations where the Argument Principle can be used to confine the spectrum by means of evaluation of the Evans function along a system of curves, resulting in an effective reduction in dimension over an exhaustive search of the complex plane.  However,  many other applications  exploit mere continuity of the Evans function (see the argument behind the utility of the parity index $\gamma$ described in \S\ref{sec:parity-index}, for example). 
We shall use continuity of $G$ to locate points of the spectrum $\sigma$ simply by looking for sign changes and applying the Intermediate Value Theorem.

\begin{lemma}
For each rotational periodic traveling wave solution $f$ of  equation \eqref{eqnlKG}, there exists an imaginary number $\lambda_-\in i\R$ such that $G(\lambda_-)<0$.
\label{lemma:G-negative}
\end{lemma}
\begin{proof}
The product of the Floquet multipliers is given by Abel's Theorem (cf.\@ \eqref{eq:Abels-Identity}).  
Combining this with Corollary~\ref{corollary:floquetpandq}, which relates the Floquet multipliers $\mu$ with those of Hill's equation, it is easy to show that $G$ can equivalently be rendered in the form
\begin{equation}
G(\lambda)=\left(q\Re\lambda\right)^2-\left(\log|\mu^\mathrm{H}(\nu(\lambda))|\right)^2,
\label{eq:G-rewrite}
\end{equation}
where $q\in\R$ is defined by \eqref{eq:q-def} and where $\mu^\mathrm{H}(\nu(\lambda))$ denotes any Floquet multiplier of Hill's equation \eqref{eq:hill} evaluated at $\nu=\nu(\lambda)$.  
It follows immediately that $G(\lambda)\leq 0$ if $\Re\lambda=0$.  We will show that the inequality  holds strictly for certain $\lambda\in i\R$.  

As explained in \S\ref{section-Hill-spectrum}, The Hill's spectrum $\Sigma^\mathrm{H}$ contains the point $\nu=0$ and for rotational waves $f$, $\nu\in\Sigma^\mathrm{H}$ implies $\nu\le 0$.  Since the Hill potential $P(z)$ is non-constant for all rotational waves, according to Remark~\ref{remark:Hill-gaps} the Hill's spectrum necessarily contains at least one gap, i.e.\@ there exists some $\nu<0$ such that $\nu\not\in\Sigma^\mathrm{H}$.  It follows that there exists a pair $\pm\lambda_-$ of nonzero imaginary numbers such that $\nu(\pm \lambda_-)\not\in\Sigma^\mathrm{H}$, where the quadratic mapping $\lambda\mapsto\nu(\lambda)$ is defined in \eqref{eq:hill}.  Since $\nu(\pm\lambda_-)\not\in\Sigma^\mathrm{H}$, by definition neither of the corresponding Floquet multipliers $\mu^\mathrm{H}(\nu(\pm\lambda_-))$ can lie on the unit circle, from which it follows that $G(\pm\lambda_-)<0$ holds (strict inequality).
\end{proof}

\begin{lemma}
For each superluminal periodic traveling wave solution $f$ of equation \eqref{eqnlKG} there exists a non-imaginary number $\lambda_+\in\C\setminus i\R$ such that $G(\lambda_+)>0$.
\label{lemma:G-positive}
\end{lemma}
\begin{proof}
In the proof of Lemma~\ref{Dreal} it was shown by comparing the first-order system \eqref{eq:firstorder} with its constant-coefficient form \eqref{constcoefsyst} that the Floquet multipliers $\mu(\lambda)$ satisfy 
\begin{equation}
\mu(\lambda)=\mu_\pm(\lambda)=e^{\lambda T/(c\pm 1)}(1+o(1)),\quad\text{in the limit $\lambda\to\infty$.}
\end{equation}
Inserting this asymptotic formula to the definition \eqref{eq:G-function-define} of $G$ yields
\begin{equation}
G(\lambda)=\left(\frac{T\Re\lambda}{c+1}+o(1)\right)\left(\frac{T\Re\lambda}{c-1}+o(1)\right),\quad\lambda\to\infty,
\end{equation}
from which it follows that for $|\Re\lambda|$ sufficiently large, $\sgn(G(\lambda))=\sgn(c^2-1)$.
Therefore, if $c^2>1$ there exists $\lambda_+$ with $|\Re\lambda_+|$ large, such that $G(\lambda_+)>0$.
\end{proof}

To complete the proof of statement (i) in Theorem~\ref{th:rotational}, let $C$ be a simple smooth curve with endpoints $\lambda_\pm$ for which $\Re\lambda\neq 0$ holds for all points of $C$ with the exception of the endpoint $\lambda_-$.  Let $\lambda=\lambda(t)$, $-1\le t\le 1$, be a smooth parametrization of $C$ for which $\lambda(\pm 1)=\lambda_\pm$.  Then $g(t):=G(\lambda(t))$ is a continuous function from $[-1,1]$ to $\R$, and $g(-1)<0$ by Lemma~\ref{lemma:G-negative} (because $f$ is rotational) while $g(1)>0$ by Lemma~\ref{lemma:G-positive} (because $f$ is superluminal).  Therefore, by the Intermediate Value Theorem there exists some $t_0\in (-1,1)$ for which $g(t_0)=0$, i.e., there exists some $\lambda_0=\lambda(t_0)$ with $\Re(\lambda_0)\neq 0$ for which $G(\lambda_0)=0$.  Hence $\lambda_0\in\sigma$ and by Definition~\ref{def:spectral-instability}, $f$ is spectrally unstable.

\begin{remark} \label{rem:libstab} 
According to Remark~\ref{rem:specH} and the discussion in \S\ref{section:mod-index-Hill}, the additional assumption that $(c^2-1)T_E>0$
implies that for librational waves of arbitrary finite speed, $\nu_2^{(0)}<\nu_1^{(0)}=0$, and hence there is a gap $(\nu_2^{(0)},\nu_1^{(0)})$ in the Hill's spectrum $\Sigma^\mathrm{H}$
for sufficiently small strictly negative $\nu$.  Therefore, under 
the assumption that $(c^2-1)T_E>0$, Lemma~\ref{lemma:G-negative} also holds for librational waves. By exactly the same argument, we conclude spectral instability of all superluminal librational waves for which $T_E>0$.
Of course we already have established spectral instability of such waves by different means (cf.\@ Theorem~\ref{th:librational} and Corollary~\ref{corollary:libration-spectral-instability}).
Note that even without the condition $T_E>0$ it can (and frequently does) happen that $\Sigma^\mathrm{H}$ has negative gaps, in which case spectral instability is again concluded.
\myendrmk
\end{remark}

\begin{remark}
Note that for any librational wave, the fact that the first periodic eigenvalue $\nu_0^{(0)}$ in the Hill's spectrum $\Sigma^\mathrm{H}$ is strictly positive implies, according to \eqref{eq:G-rewrite}, that $G(\lambda)>0$ holds at the two nonzero real preimages of $\nu_0^{(0)}$ under the quadratic mapping $\lambda\to \nu(\lambda)$. 

Combining this fact with the formula $\mathrm{sgn}(G(\lambda))=\mathrm{sgn}(c^2-1)$ holding for $\Re \lambda>0$ sufficiently large as shown in the proof of Lemma~\ref{lemma:G-positive} one sees that the continuous function $G(\lambda)$ has to vanish somewhere in the right half-plane whenever $f$ is a subluminal librational wave, and hence \emph{all} such waves are spectrally unstable.

Likewise, if there is also a gap in the negative part of the Hill's spectrum $\Sigma^\mathrm{H}$, then by the same argument as in the proof of Lemma~\ref{lemma:G-negative}, $G(\lambda)<0$ when $\lambda$ is the imaginary preimage under $\nu$ of a point in the gap.  Therefore again $G$ changes sign at some point with $\Re \lambda>0$ implying spectral instability of the librational wave.
Hence a negative gap in the Hill's spectrum $\Sigma^\mathrm{H}$ always indicates spectral instability of a librational wave, \emph{regardless of its speed}.  

Under some conditions it is possible to guarantee a negative gap in the Hill's spectrum of a librational wave.  Indeed, according to Remark~\ref{rem:specH}, if $f$ is a librational wave then 
$\Sigma^\mathrm{H}$ can have at most two positive gaps (and there are exactly two positive gaps if $(c^2-1)T_E<0$; see the bottom panel of Figure~\ref{fig:hillspec} for an illustration).  Furthermore, according to Remark~\ref{remark:Hill-gaps}, as long as $P(z)=V''(f(z))/(c^2-1)$ is neither constant, nor an elliptic function, nor a hyperelliptic function of genus $2$, then $\Sigma^\mathrm{H}$ has at least three gaps in total.  Hence, for such $f$, there exists a negative gap and therefore spectral instability is deduced.
\label{remark:librational-general-speed-Hill-gap}
\myendrmk
\end{remark}

\subsection{Subluminal rotational waves} 
\label{sec:sub-rotation-stable}
We will prove that all subluminal rotational waves are spectrally stable by a direct calculation
showing that $\lambda\in\sigma$ implies $\Re\lambda=0$.  Recall the boundary-value problem \eqref{eq:spectral} with boundary condition 
\eqref{eq:periodicode} parametrized by $\theta\in\R\pmod{2\pi}$ and characterizing the partial spectrum $\sigma_\theta\subset\sigma$.  The differential equation \eqref{eq:spectral} 
can be written in terms of Hill's operator $\cH$ defined in \eqref{eq:defH} as follows:
\begin{equation} \label{eq:oureqn}
(c^2-1) \cH w(z) - 2 c \lambda w_z(z) + \lambda^2 w(z) = 0.
\end{equation}
Suppose that $\lambda\in\sigma$, and therefore that there exists $\theta\in\R$ such that
$\lambda\in\sigma_\theta$.  Let $w\in C^2(\R)$ denote the corresponding nontrivial solution
of the boundary-value problem consisting of \eqref{eq:spectral} subject to \eqref{eq:periodicode}.  

Multiplying the differential equation \eqref{eq:oureqn} through by $w(z)^*$ and integrating over the fundamental period $[0,T]$ gives
\begin{equation}\label{eq:lambdaquadratic}
(c^2-1)\langle w,\cH w\rangle -2im\lambda+\|w\|^2\lambda^2=0, \quad \text{where} \quad m:=-ic\langle w,w_z\rangle\in\mathbb{R},
\end{equation}
and the notation $\langle\cdot,\cdot\rangle$ and $\|\cdot\|$ is defined in \eqref{eq:innerprod}.
That $m$ is real follows by integration by parts
using the boundary conditions \eqref{eq:periodicode} (cf.\@ \eqref{eq:ip-w-wz-imag}).
The relation \eqref{eq:lambdaquadratic} can be viewed as a quadratic equation for $\lambda$;  solving this equation for $\lambda$ in terms of $\langle w,\cH w\rangle$, $m$, and $\|w\|^2$, we obtain:
\begin{equation}
\lambda=\frac{1}{\|w\|^2}\left[im \pm\sqrt{-m^2-(c^2-1)\|w\|^2\langle w,\cH w\rangle}\right].
\label{eq:quadraticsoln}
\end{equation}
Because $m$ is real, it follows that $\Re\lambda=0$ as long as $f$ is subluminal (implying $c^2-1<0$) and $f$ is a rotational wave (implying the negative semidefiniteness condition $\langle w,\cH w\rangle\le 0$ according to Proposition~\ref{cor:hnegdef}). This completes the proof of statement (ii) of Theorem \ref{th:rotational}.

The proof of statement (ii) given above resembles a well-known result in the theory of linearized Hamiltonian systems, which concerns eigenvalue problems of the form $\cJ\cL u=\lambda u$ where
$\cJ$ is skewadjoint and $\cL$ is selfadjoint.  If $\cJ$ is invertible, this can be written as a generalized eigenvalue problem:  $\cL u=\lambda \cJ^{-1}u$.  The latter can be reformulated in terms of a linear operator pencil $\cT^\mathrm{lin}(\lambda):=\cL-\lambda\cJ^{-1}$.  Taking the inner product with $u$ yields
\begin{equation}
\langle u,\cL u\rangle = \lambda\langle u,\cJ^{-1}u\rangle.
\end{equation}
If $\cL$ is a definite operator (positive or negative), then $\langle u,\cL u\rangle\neq 0$, and hence neither
$\lambda$ nor $\langle u,\cJ^{-1}u\rangle$ can vanish.  Therefore $\lambda$ is a generalized Rayleigh quotient:
\begin{equation}
\lambda = \frac{ \langle u,\cL u\rangle}{\langle u,\cJ^{-1}u\rangle},
\label{eq:RayleighQuotient}
\end{equation}
which by selfadjointness of $\cL$ and skewadjointness of $\cJ$ (and hence of $\cJ^{-1}$) is purely imaginary.  This is the easy case of a linearized Hamiltonian eigenvalue problem; if $\cL$ is indefinite there generally will be non-imaginary eigenvalues, although the spectrum is still symmetric with respect to reflection through the imaginary axis.  

The formula \eqref{eq:quadraticsoln} is evidently a further extension of the generalized Rayleigh quotient  \eqref{eq:RayleighQuotient} to the setting of the quadratic pencil $\cT(\lambda)$ arising in the linearization of the (Hamiltonian) Klein-Gordon equation \eqref{eqnlKG}, and the positivity of $(c^2-1)\cH$ plays the role of the definiteness of $\cL$.

\section{Summary of spectral stability and instability results}
\label{section:summary}
For the convenience of the reader, we now formulate a theorem that summarizes all of the results we have obtained regarding the spectral stability  of periodic traveling waves for the nonlinear Klein-Gordon equation \eqref{eqnlKG}:
\begin{theorem}[spectral stability properties of periodic traveling waves]
\label{theorem:summary}
Let $V$ be a potential satisfying Assumptions~\ref{assumptionsV} and \ref{assumptionsnormalize}.  Then,  the families of periodic traveling waves $f=f(z;E,c)$ parametrized by $(E,c)\in\region$ have the following properties.
\begin{itemize} 
\item[(i)] Subluminal rotational waves are spectrally stable. 
\item[(ii)] Superluminal rotational waves are spectrally unstable.  Generically, however, these waves are not strongly modulationally unstable. The spectrum $\sigma$ is confined to a vertical strip in the complex plane containing the imaginary axis.
\item[(iii)] Subluminal librational waves are spectrally unstable.  If $T_E<0$ then $\gamma=-1$ so there is a positive real periodic eigenvalue in $\sigma$, and also $\rho=-1$ so there is strong modulational instability.  If $T_E\ge 0$ then generically $\rho=1$ so the wave is not strongly modulationally unstable, and although the parity index $\gamma$ is indeterminate there exist positive real points in the spectrum $\sigma$ (necessarily periodic or antiperiodic eigenvalues).
The unstable spectrum $\sigma\setminus i\R$ is bounded.
\item[(iv)] Superluminal librational waves are spectrally unstable, provided either
\begin{itemize}
\item[(a)] $T_E>0$, in which case $\rho=-1$ so there is a strong modulational instability, or
\item[(b)] $T_E\le 0$ but the associated Hill's spectrum $\Sigma^\mathrm{H}$ has a  gap in the negative half-line $\R_-$ (see Remarks \ref{remark:Hill-gaps} and \ref{remark:librational-general-speed-Hill-gap}). 
\end{itemize}
In particular a spectrally stable superluminal librational wave $f$ is necessarily either an elliptic function or a hyperelliptic function of genus $2$.
Regardless of whether case (a) or (b) holds, the spectrum $\sigma$ is confined to a vertical strip in the complex plane containing the imaginary axis.
\end{itemize}
Furthermore, periodic traveling waves of infinite speed are spectrally stable if and only if the associated Hill's spectrum $\Sigma^\mathrm{H}$ has no gaps in the negative half-line $\R_-$; in particular all infinite speed rotational waves are spectrally unstable.
\end{theorem}
\begin{proof}
The basic stability results in statements (i) and (ii) were proved in \S\ref{secfurther} (cf.\@ Theorem~\ref{th:rotational}), and the part of statement (ii) concerning the potential for weak modulational instability is addressed in Remark~\ref{remark:modulational-instability-index-troubles}.  
The quickest proof of the basic instability result in statement (iii) is to appeal to the argument in the first two paragraphs of Remark~\ref{remark:librational-general-speed-Hill-gap}.
The part of statement (iii) concerning the parity index $\gamma$ indicating the presence of a non-modulational instability when $T_E<0$ follows from 
Theorem~\ref{th:sublibunstable}; for the indeterminacy of the parity index when $T_E>0$ see Remark~\ref{remark:parity-index-troubles}.  The parts of statements (iii) and (iv) concerning strong modulational instability follow from Theorem~\ref{th:librational} and Corollary~\ref{corollary:libration-spectral-instability}.  The parts of statement (iv) concerning negative gaps in the Hill spectrum follow from the argument in 
Remark~\ref{remark:librational-general-speed-Hill-gap}.  The parts of statements (ii), (iii), and (iv) concerning bounds on the unstable spectrum $\sigma$ follow from Lemma~\ref{lem:bound}.  Finally, the statement concerning the stability of waves of infinite speed follows from Theorem~\ref{theorem:infinite-speed} and Corollary~\ref{corollary:rotation-infinite-speed-unstable}.
\end{proof}

\section{The modulational instability index $\rho$ and formal modulation theory}
\label{sec:modulation}
\subsection{Whitham's fully nonlinear modulation theory}
\label{secWhitham}
Here we aim to relate the modulational instability index $\rho$ introduced in \S\ref{secmodinst} to Whitham's fully nonlinear theory of slowly modulated periodic waves \cite{Wh1,Wh}. We prove that the index $\rho$ determines the analytical nature of Whitham's system of modulation equations in terms of their suitability for determining a well-posed evolution.  This calculation justifies, in a fashion, the terminology introduced in Definition~\ref{def:modulational-instability}.  For a related analysis in the case of periodic traveling waves in the generalized Korteweg-de Vries equation, see Bronski and Johnson \cite{BrJo2} and Johnson and Zumbrun \cite{JZ1}.

\subsubsection{Derivation of Whitham's modulation equations}
Beginning with the seminal paper \cite{Wh1} in 1965 and continuing for some 20 years, G.\@ B.\@ Whitham introduced, revisited, and refined a fully nonlinear theory of modulated waves that has made a very significant impact in the field.  The theory produces results that are physically satisfying and that have been verified in experiments, and it also has a remarkable mathematical structure, especially when the underlying nonlinear wave equation is completely integrable (e.g., leading to important examples of quasilinear systems with an arbitrary number of equations that have several unusual features, including:  (i) they admit diagonalization via special variables called Riemann invariants, and (ii) general solutions can be represented in implicit form by a 
technique \cite{Tsarev91} that generalizes the classical hodograph transform for $2\times 2$ systems).

Whitham's theory is based on the supposition that nonlinear wave equations that support families of periodic traveling waves should also have other solutions that are close to different representatives of the family in different regions of space-time.  In other words, there should be solutions that, near any given $(x,t)$ can be accurately approximated by a periodic traveling wave solution, but the parameters that single out the wave from the family might be different for different points $(x,t)$.  Thus, the periodic traveling wave is spatiotemporally modulated, having different amplitude, wavelength, or frequency at different points $(x,t)$, while still somehow locally resembling a periodic wavetrain. 

The key to making this idea precise at the formal level at least is to consider it to be an asymptotic theory in the limit where the variation in the wave parameters (these are the constants $E$ and $c$ in the present context) is very gradual compared to the fluctuations of  the wave itself.  The ratio of the microscopic scale (e.g., typical wavelength) to the macroscopic scale (e.g., characteristic width of a wave packet) thus becomes a small dimensionless parameter $\epsilon$,
and this makes available a toolbox of asymptotic methods to study the dynamics of the modulating waves.  The goal of this formal analysis is to deduce effective equations (the \emph{modulation equations}) that describe how the parameters vary slowly.  Such equations should only involve quantities that vary on the macroscopic scale.  

As he developed his theory, Whitham found several different equivalent ways to derive the modulation equations.  His first approach \cite{Wh1} (we will follow this in some detail below) was based on period averaging of densities and fluxes of local conservation laws consistent with the underlying nonlinear wave equation.  He quickly found a second method based on an averaged variational principle \cite{Wh2} that has additional appeal because the ansatz used for the modulating wave includes global phase information that was not captured by his original approach.  One can also apply directly the method of multiple scales to the underlying nonlinear wave equation and obtain the modulation equations as solvability conditions for the existence of uniformly bounded corrections to the leading-order modulated wave ansatz.  This last approach can apply also to weakly dissipative systems that do not have conservation laws or a Lagrangian formulation.  

We now give a brief formal derivation of the modulation equations following Whitham's original method \cite{Wh1} (see \cite{BuM2,SCR,Wh} and the references therein for information about other approaches).  The Klein-Gordon equation \eqref{eqnlKG} implies the following local conservation laws \cite{AnBl,SanKon}:
\begin{equation}
D_1[u]_t + F_1[u]_x=0\quad\text{and}\quad D_2[u]_t+F_2[u]_x=0
\label{eq:cons-laws}
\end{equation}
where the \emph{densities} $D_j[u]$ and \emph{fluxes} $F_j[u]$ are
\begin{equation}
\begin{split}
D_1[u]&:=\tfrac{1}{2}u_t^2 +\tfrac{1}{2}u_x^2 + V(u)\\
D_2[u]&:=-u_tu_x
\end{split}\quad\text{and}\quad
\begin{split}
F_1[u]&:=-u_tu_x\\
F_2[u]&:=\tfrac{1}{2}u_t^2+\tfrac{1}{2}u_x^2-V(u).
\end{split}
\end{equation}
Whitham argues that if the wavetrain represented by the parameters $(E,c)\in\region$ is slowly modulated, so that $E$ and $c$ become slowly-varying functions of $(x,t)$, the conservation laws \eqref{eq:cons-laws} should govern (at leading order) the functions $E=E(x,t)$ and $c=c(x,t)$  after  simply replacing the densities and fluxes by their averages over one period of the wave.  The procedure is therefore the following:  for each \emph{fixed} choice of $(E,c)\in\region$, we first replace $u(x,t)$ by the exact periodic traveling wave $u=f(z;E,c)$ in $D_j[u]$ and $F_j[u]$ and therefore obtain $T$-periodic functions of $z$ that can be averaged over a period.  The corresponding averages are functions of $(E,c)\in\region$ given by
\begin{equation}
\begin{split}
\langle D_1\rangle(E,c)&:=\frac{1}{T}\int_0^T\left(\tfrac{1}{2}(c^2+1)f_z(z;E,c)^2+V(f(z;E,c))\right)\,dz\\
\langle D_2\rangle(E,c)&:=\frac{c}{T}\int_0^Tf_z(z;E,c)^2\,dz
\end{split}
\label{eq:averaged-densities}
\end{equation}
and
\begin{equation}
\begin{split}
\langle F_1\rangle(E,c)&:=\frac{c}{T}\int_0^Tf_z(z;E,c)^2\,dz\\
\langle F_2\rangle (E,c)&:=\frac{1}{T}\int_0^T\left(\tfrac{1}{2}(c^2+1)f_z(z;E,c)^2-V(f(z;E,c))\right)\,dz.
\end{split}
\label{eq:averaged-fluxes}
\end{equation}
Then, one allows $(E,c)$ to depend smoothly on $(x,t)$ and \emph{Whitham's modulation equations} are (by definition, really)
\begin{equation}
\langle D_1\rangle_t + \langle F_1\rangle_x=0\quad\text{and}\quad
\langle D_2\rangle_t +\langle F_2\rangle_x=0.
\label{eq:Whitham-system-1}
\end{equation}
The averaged densities and fluxes may be represented in terms of the function 
\begin{equation}
W(E,c):=(c^2-1)\int_0^Tf_z(z;E,c)^2\,dz.
\label{eq:W-define}
\end{equation}
Indeed, using \eqref{eq:waveeq} to eliminate $V(f)$ for constant $(E,c)$ in the definitions \eqref{eq:averaged-densities}--\eqref{eq:averaged-fluxes}, the definition \eqref{eq:W-define} yields
\begin{equation}
\langle D_1\rangle =\frac{W}{(c^2-1)T}+E,\;\; \langle F_1\rangle=\langle D_2\rangle=
\frac{cW}{(c^2-1)T},\;\;\text{and}\;\;
\langle F_2\rangle = \frac{c^2W}{(c^2-1)T}-E.
\label{eq:averages}
\end{equation}
Here we recall that the period $T$ also is a function of $(E,c)\in\region$.

The function $W=W(E,c)$ has a number of  properties useful for simplifying and analyzing the modulation equations that result from substituting \eqref{eq:averages} into \eqref{eq:Whitham-system-1}:
\begin{lemma}
\label{lemma:W-properties}
Suppose that $(E,c)\in\region$ (any of the four disjoint open components $\region_<^\mathrm{lib}$, $\region_<^\mathrm{rot}$, $\region_>^\mathrm{lib}$, or $\region_>^\mathrm{rot}$).  Then $W=W(E,c)$ defined by \eqref{eq:W-define} satisfies:
\begin{itemize}
\item[(i)] 
\begin{equation}
\sgn(W)=\sgn(c^2-1);
\label{eq:sgnW}
\end{equation}
\item[(ii)] 
\begin{equation}
W_c = \frac{cW}{c^2-1}\quad\text{and}\quad W_{cc}=-\frac{W}{(c^2-1)^2};
\label{eq:WcWcc}
\end{equation}
\item[(iii)]
\begin{equation}
W_E=T.
\label{eq:WE}
\end{equation}
\end{itemize}
\end{lemma}
\begin{proof}
The statement (i) is obvious from the definition \eqref{eq:W-define}.  To prove statements (ii) and (iii), we rewrite $W$ using the differential identity $f_z^2\,dz = f_z\,df$.  More precisely, for rotational waves where $f_z>0$ for all $z$, we solve for $f_z$ in terms of $f$ by taking a positive square root in \eqref{eq:waveeq} and therefore obtain
\begin{equation}
W(E,c)= \sqrt{2}(c^2-1)\int_0^{2\pi}\sqrt{\frac{E-V(f)}{c^2-1}}\,df,\quad (E,c)\in\region_<^\mathrm{rot}\cup\region_>^\mathrm{rot},
\end{equation}
because $f$ increases by $2\pi$ over the period $T$.  For librational waves we may represent $W$ as twice the integral over half the period and assume that $f_z>0$ holds over the half-period of integration.  Therefore,
\begin{equation}
W(E,c)=2\sqrt{2}(c^2-1)\int_{f^-(E)}^{f^+(E)}\sqrt{\frac{E-V(f)}{c^2-1}}\,df,\quad (E,c)\in\region_>^\mathrm{lib},
\end{equation}
and 
\begin{equation}
W(E,c)=2\sqrt{2}(c^2-1)\int_{f^+(E)}^{f^-(E)+2\pi}\sqrt{\frac{E-V(f)}{c^2-1}}\,df,\quad (E,c)\in\region_<^\mathrm{lib}.
\end{equation}
In each case, the identities \eqref{eq:WcWcc} are easy to establish from these formulae, which proves (ii).  To prove (iii), one differentiates with respect to $E$ (using Leibniz' formula in the librational cases, but noting that the contributions from the endpoints of the interval of integration vanish because $V(f^\pm\pmod{2\pi})=E$ by definition of $f^\pm$) and compares with \eqref{periodsuplib}, \eqref{periodsublib}, \eqref{periodsuprot}, and \eqref{periodsubrot}.
\end{proof}

We now simplify the modulation equations by the following systematic steps.  First, use \eqref{eq:WE} to eliminate the period $T$, putting them in the form
\begin{equation}
\begin{split}
\left(\frac{W(E,c)}{(c^2-1)W_E(E,c)}+E\right)_t + \left(\frac{cW(E,c)}{(c^2-1)W_E(E,c)}\right)_x&=0,\\
\left(\frac{cW(E,c)}{(c^2-1)W_E(E,c)}\right)_t +\left(\frac{c^2W(E,c)}{(c^2-1)W_E(E,c)}-E\right)_x&=0.
\end{split}
\label{eq:Whitham-system-2}
\end{equation}
Next, apply the chain rule to isolate the coefficients of the partial derivatives $E_{x,t}$ and $c_{x,t}$.  This places the system in obvious quasilinear form:
\begin{equation}
\bT(E,c)\begin{pmatrix}E\\c\end{pmatrix}_t + \bX(E,c)\begin{pmatrix}E\\c\end{pmatrix}_x=0,
\label{eq:Whitham-system-3}
\end{equation}
where, after clearing out a scalar denominator and using \eqref{eq:WcWcc} to eliminate all partial derivatives of $W$ with respect to $c$,  the coefficient matrices are given by
\begin{equation}
\bT(E,c):=\begin{pmatrix}(c^2-1)c^2(c^2W_E^2-WW_{EE}) & -2c^3WW_E\\
(c^2-1)c(W_E^2-WW_{EE}) & -(c^2+1)WW_E\end{pmatrix}
\end{equation}
and 
\begin{equation}
\bX(E,c):=\begin{pmatrix}
(c^2-1)c^3(W_E^2-WW_{EE}) & -(c^2+1)c^2WW_E\\
(c^2-1)(W_E^2-c^2WW_{EE}) & -2cWW_E\end{pmatrix}.
\end{equation}
Finally, inverting $\bT(E,c)$, the modulation equations take the form
\begin{equation}
\begin{pmatrix}E\\c\end{pmatrix}_t +\bU(E,c)\begin{pmatrix}E\\c\end{pmatrix}_x=0,
\label{eq:Whitham-system-4}
\end{equation}
with coefficient matrix $\bU(E,c)$ given by
\begin{equation}
\bU(E,c):=\frac{1}{c^2W_E^2+WW_{EE}}\begin{pmatrix}
(W_E^2+WW_{EE})c & -WW_E\\(c^2-1)^2W_EW_{EE} & (W_E^2+WW_{EE})c\end{pmatrix}.
\end{equation}

\subsubsection{Interpretation of Whitham's modulation equations.  Relation to the modulational instability index $\rho$}
The Whitham system of modulation equations, say taken in the quasilinear form \eqref{eq:Whitham-system-4}, is supposed to describe the slow variation of the parameters $(E,c)$ determining the shape and period of the periodic traveling wave profile $f=f(z;E,c)$.  If this system is to describe the evolution in time $t$ of initially given modulation profiles $E(x,0)=E_0(x)$ and 
$c(x,0)=c_0(x)$ in a suitable function space, then the correct application of the theory is to pose the Cauchy initial value problem for the system \eqref{eq:Whitham-system-4} with the given initial data.  For a general reference for the theory of such initial-value problems, see \cite{Lax06}.  

The basic tool for treating quasilinear systems like \eqref{eq:Whitham-system-4} is to
apply the method of characteristics, and the first step is the calculation of the \emph{characteristic velocities}, which are defined as the eigenvalues of the coefficient matrix $\bU(E,c)$.  The essential dichotomy that arises is that the eigenvalues of the real matrix $\bU(E,c)$ are either real  or they form a  complex-conjugate pair.  In the real (resp., complex-conjugate) case the system \eqref{eq:Whitham-system-4} is said to be \emph{hyperbolic} (resp., \emph{elliptic}) at $(E,c)$.
The key result in the theory is that the Cauchy initial-value problem is (locally) well-posed if and only if the system is hyperbolic at $(E_0(x),c_0(x))$ for all $x\in\R$.  Whitham interprets the hyperbolicity of the system \eqref{eq:Whitham-system-4} as a kind of modulational stability of the family of periodic traveling waves that is spatially modulated at $t=0$, because the modulation equations make a well-defined prediction at least locally in time for the evolution of the given initial modulation profiles $(E_0(\cdot),c_0(\cdot))$.  By contrast, in the elliptic case, there is generally no solution for $t>0$ of the initial-value problem at all, even in the roughest spaces of generalized functions (distributions).  Indeed, if one linearizes \eqref{eq:Whitham-system-4} about a constant state by ``freezing'' the coefficient matrix $\bU$, the linearized system can be treated by Fourier transforms and in the elliptic case the Fourier mode $e^{ikx}$ experiences exponential growth in time $t$ with a rate that itself grows as $k\to\infty$.  This preferentially amplifies the tails of the Fourier transform so that after an infinitesimal time, the Fourier transform of the solution is not the image of any tempered distribution.    This severe amplification effect again suggests an instability in the underlying nonlinear wave equation \eqref{eqnlKG}.  

Here, we make the interpretation of the hyperbolic/elliptic dichotomy in terms of stability precise with the following result:
\begin{theorem}
\label{theorem:Whitham-mod-inst}
Suppose $c^2W_E^2+WW_{EE}\neq 0$.  
Whitham's system of modulation equations \eqref{eq:Whitham-system-4} is hyperbolic (resp., elliptic) if and only if $\rho=1$ (resp., $\rho=-1$), where $\rho$ is the modulational instability index (cf.\@ Definition~\ref{defmodinstindx} in \S\ref{secmodinst}).
\end{theorem}
\begin{proof}
The characteristic velocities $v$ are obtained as the eigenvalues of $\bU(E,c)$:
\begin{equation}
v=\frac{(W_E^2+WW_{EE})c\pm\sqrt{-(c^2-1)^2WW_E^2W_{EE}}}{c^2W_E^2+WW_{EE}}.
\end{equation}
Obviously, these are real numbers if and only if  $WW_{EE}\le 0$.  But according to Lemma~\ref{lemma:W-properties}, $\sgn(WW_{EE}) = \sgn((c^2-1)T_E)$.  Comparing with the form of the modulational instability index $\rho$ given in Proposition~\ref{corollary:index-period}, the proof is complete.
\end{proof}

Note that if $T_E=0$, then simultaneously $\rho=0$ and the Whitham system is on the borderline between the hyperbolic and elliptic cases with a double real characteristic velocity.  Also, since
$c^2W_E^2+WW_{EE}=\tfrac{1}{2}(c^2-1)^2D_{\lambda\lambda}(0,1)$, the singular case for the Whitham equations corresponds also to (strong) modulational instability according to Remark~\ref{remark:one-curve}.

Theorem~\ref{theorem:Whitham-mod-inst} analytically confirms in the case of the Klein-Gordon   equation \eqref{eqnlKG} what is (at least in full generality) mathematical folklore:  
if the Whitham modulation system is elliptic, then the underlying periodic traveling wave is spectrally unstable. The relationship between the nature of the Whitham modulation system and the linear stability properties of waves has been explored in various other contexts as well.  See, for example, \cite{BrJo2,JZ1,BangM98}.

\subsection{Weakly nonlinear modulation theory for librational waves}
\label{sec:NLS}
Consider a family of librational periodic traveling waves for the Klein-Gordon equation \eqref{eqnlKG}
for which the orbits surround a single equilibrium point $f=u^0$ in the phase portrait of \eqref{eqnlp} in the $(f,f_z)$-plane.  Such a family admits a different type of formal modulation theory based on the assumption that the orbits are close to equilibrium.  In contrast to the quasilinear system \eqref{eq:Whitham-system-4}, this weakly nonlinear modulation theory yields 
instead a model equation of nonlinear Schr\"odinger type governing a slowly-varying complex envelope of an underlying (librational) carrier wave.  Here we will deduce the nonlinear Schr\"odinger equation, discuss the properties of its solutions and how they depend on the signs of the coefficients, and then we will again make a connection with the modulational instability index $\rho$ introduced in \S\ref{secmodinst}.
\subsubsection{Derivation of the cubic nonlinear Schr\"odinger equation}
Let $u^0$ be a non-degenerate critical point of the potential $V$, and assume that $V$ has a sufficient number of continuous derivatives near $u^0$.  We suppose that the solution $u(x,t)$ of the Klein-Gordon equation \eqref{eqnlKG} is close to equilibrium by introducing an artificial small parameter $\epsilon>0$ and writing $u=u^0+\epsilon U$.  We will develop a formal asymptotic expansion for $U$ in the limit $\epsilon\to 0$ with the use of the method of multiple scales \cite[Chapter 10]{AAA}.  Specifically, we seek $U$ as a function $U=U(X_0,X_1,T_0,T_1,T_2;\epsilon)$ depending on the ``fast'' variables $X_0=x$ and $T_0=t$ as well as the ``slow'' variables $X_1=\epsilon x$, $T_1=\epsilon t$, and $T_2=\epsilon^2 t$.
By the chain rule, the Klein-Gordon equation becomes a partial differential equation in $5$ independent variables:
\begin{multline}
U_{T_0T_0}+2\epsilon U_{T_0T_1}+\epsilon^2\left[2U_{T_0T_2}+U_{T_1T_1}\right] + 2\epsilon^3U_{T_1T_2} + \epsilon^4U_{T_2T_2}\\{} -
U_{X_0X_0}-2\epsilon U_{X_0X_1}-\epsilon^2U_{X_1X_1} + \epsilon^{-1}V'(u^0+\epsilon U)=0.
\end{multline}
Expanding $V$ in a Taylor series about $u^0$ yields
\begin{multline}
\cL U + \epsilon\left[2 U_{T_0T_1}-2U_{X_0X_1}+\tfrac{1}{2}V'''(u^0)U^2\right]\\
{}+\epsilon^2\left[2U_{T_0T_2}+U_{T_1T_1}-U_{X_1X_1}+\tfrac{1}{6}V^{(4)}(u^0)U^3\right]=
O(\epsilon^3),
\label{eq:MMS-PDE}
\end{multline}
where the linear operator $\cL$ is defined by $
\cL U:=U_{T_0T_0}-U_{X_0X_0}+V''(u^0)U$.
We next suppose that $U$ has an asymptotic power series expansion in $\epsilon$ of the form
\begin{equation}
U=U^{[0]}+\epsilon U^{[1]} + \epsilon^2U^{[2]} +O(\epsilon^3),\quad\epsilon\to 0.
\end{equation}
Equating the terms in \eqref{eq:MMS-PDE} corresponding to the same powers of $\epsilon$ leads to a hierarchy of equations governing the terms $U^{[n]}$, beginning with:
\begin{equation}
\cL U^{[0]}=0,
\label{eq:epsilon-0}
\end{equation}
\begin{equation}
\cL U^{[1]}=-2U^{[0]}_{T_0T_1}+2U^{[0]}_{X_0X_1}-\tfrac{1}{2}V'''(u^0)U^{[0]2},
\label{eq:epsilon-1}
\end{equation}
\begin{multline}
\cL U^{[2]}=-2U^{[1]}_{T_0T_1}+2U^{[1]}_{X_0X_1}-V'''(u^0)U^{[0]}U^{[1]}\\
{}-2U^{[0]}_{T_0T_2}-U^{[0]}_{T_1T_1}+U^{[0]}_{X_1X_1}-\tfrac{1}{6}V^{(4)}(u^0)U^{[0]3}.
\label{eq:epsilon-2}
\end{multline}

The idea behind the method of multiple scales is to solve these equations in order, choosing the dependence on the slow scales $X_1$, $T_1$, and $T_2$ so as to ensure that $U^{[n]}$ is bounded (in the parlance of the method, we wish to avoid \emph{secular terms} in the solution).
We begin by taking a solution of  \eqref{eq:epsilon-0} in the form
\begin{equation}
U^{[0]}=Ae^{i\theta}+\circledast,\quad A=A(X_1,T_1,T_2),\quad\theta:=kX_0-\omega T_0,
\end{equation}
where the symbol $\circledast$ indicates the complex conjugate of the preceding term.
This is a wavetrain solution of the Klein-Gordon equation linearized about the equilibrium $u^0$ as long as the wavenumber $k$ and the frequency $\omega$ are linked by the linear dispersion relation
\begin{equation}
\omega^2=k^2+V''(u^0).
\label{eq:linear-dispersion-relation}
\end{equation}
With this solution in hand, we see that \eqref{eq:epsilon-1} takes the form
\begin{multline}
\cL U^{[1]}=-\tfrac{1}{2}V'''(u^0)A^2e^{2i\theta}+2i\left[\omega A_{T_1}+k A_{X_1}\right]e^{i\theta}\\{} -V'''(u^0)|A|^2-2i\left[\omega A^*_{T_1}+kA^*_{X_1}\right]e^{-i\theta}-\tfrac{1}{2}V'''(u^0)A^{*2}e^{-2i\theta}.
\end{multline}
It can be shown that $U^{[1]}$ will be bounded as a function of the fast time $T_0$ if and only if the terms proportional to the fundamental harmonics $e^{\pm i\theta}$ are not present on the right-hand side.  Noting that implicit differentiation of \eqref{eq:linear-dispersion-relation} with respect to $k$ gives $k/\omega=\omega'(k)$, we avoid secular terms by demanding that the envelope $A(X_1,T_1,T_2)$ satisfy the equation
\begin{equation}
A_{T_1} + \omega'(k)A_{X_1}=0.
\label{eq:A-T1}
\end{equation}
This equation shows that the envelope $A$ moves rigidly to the right (for fixed $T_2$) at the  \emph{linear group velocity} $\omega'(k)$ associated with the carrier wave $e^{i\theta}$ with wavenumber $k$ and frequency $\omega=\omega(k)$.  A particular solution for $U^{[1]}$ is then
\begin{equation}
\begin{split}
U^{[1]}&=\frac{\tfrac{1}{2}V'''(u^0)A^2}{4\omega^2-4k^2-V''(u^0)}e^{2i\theta}-\frac{V'''(u^0)}{V''(u^0)}|A|^2 +\frac{\tfrac{1}{2}V'''(u^0)A^{*2}}{4\omega^2-4k^2-V''(u^0)}e^{-2i\theta}\\
&=\frac{V'''(u^0)A^2}{6V''(u^0)}e^{2i\theta}-\frac{V'''(u^0)}{V''(u^0)}|A|^2 +\frac{V'''(u^0)A^{*2}}{6V''(u^0)}e^{-2i\theta},
\end{split}
\end{equation}
where we have used the dispersion relation \eqref{eq:linear-dispersion-relation} to eliminate $\omega^2-k^2$.
While it is possible to include homogeneous terms proportional to the fundamental harmonics $e^{\pm i\theta}$, these add nothing further since the amplitude $A$ is still undetermined.  Next, we apply similar reasoning to \eqref{eq:epsilon-2}, which takes the form
\begin{multline}
\cL U^{[2]}=\left[2i\omega A_{T_2}-A_{T_1T_1}+A_{X_1X_1}+\frac{5V'''(u^0)^2-3V''(u^0)V^{(4)}(u^0)}{6V''(u^0)}|A|^2A\right]e^{i\theta} \\{}+ \circledast +\text{other harmonics},
\end{multline}
where the other harmonics are terms with $(X_0,T_0)$-dependence proportional to $e^{in\theta}$ for $n\neq \pm 1$.  Such terms cannot cause $U^{[2]}$ to grow as a function of $T_0$, but the forcing terms proportional to $e^{\pm i\theta}$ must be removed if $U^{[2]}$ is to remain bounded.
We therefore insist that in addition to \eqref{eq:A-T1}, $A$ should satisfy the equation
\begin{equation}
2i\omega A_{T_2}-A_{T_1T_1}+A_{X_1X_1}+\frac{5V'''(u^0)^2-3V''(u^0)V^{(4)}(u^0)}{6V''(u^0)}|A|^2A=0.
\label{eq:NLS-1}
\end{equation}
It is best to go into a frame of reference moving with the group velocity by introducing the variables
$\zeta:=X_1-\omega'(k)T_1$, $T_1,$ and $\tau:=T_2$,
and writing $a(\zeta,T_1,\tau)=A(X_1,T_1,T_2)$.
Then, \eqref{eq:A-T1} simply reads $a_{T_1}=0$, so in fact $a=a(\zeta,\tau)$.  Also, since differentiation of \eqref{eq:linear-dispersion-relation} twice implicitly with respect to $k$ gives
$\omega'(k)^2 + \omega(k)\omega''(k)=1$, the equation \eqref{eq:NLS-1} can be written in the
form
\begin{equation}
ia_\tau + \tfrac{1}{2}\omega''(k)a_{\zeta\zeta} + \beta |a|^2a=0,\quad
\beta:=\frac{5V'''(u^0)^2-3V''(u^0)V^{(4)}(u^0)}{12\omega(k)V''(u^0)}.
\label{eq:NLS-2}
\end{equation}
Therefore, the slow modulation of the complex amplitude in the frame moving with the linear group velocity obeys the \emph{cubic nonlinear Schr\"odinger equation}  \eqref{eq:NLS-2}.  

The above derivation can be simultaneously viewed as a special case of and a higher-order generalization of Whitham's theory as developed in \S\ref{secWhitham}.  Indeed, if we consider near-equilibrium librational waves, then the phase velocity of the wavetrain is nearly constant ($c\approx \omega/k$), and so is the energy ($E\approx V(u^0)$).  Linearizing Whitham's modulational system \eqref{eq:Whitham-system-4} about these constant states amounts to replacing the matrix $\bU(E,c)$ by its constant limit $\bU(V(u^0),\omega/k)$.  In this limit, it is easy to check that $W\to 0$, and hence
\begin{equation}
\bU(V(u^0),\omega/k)=\begin{pmatrix}\omega'(k) & 0\\\displaystyle\left(\frac{\omega}{k}-\omega'(k)\right)^2\frac{T_E}{T} & \omega'(k)\end{pmatrix},
\end{equation}
where $T$ and $T_E$ are evaluated at $c=\omega/k$ and $E=V(u^0)$.  Clearly, the characteristic velocities degenerate to the linear group velocity $\omega'(k)$ in this limit.  In particular, upon linearization, the equation governing $E$ decouples:
\begin{equation}
E_t +\omega'(k)E_x=0.
\label{eq:EnergyTransport}
\end{equation}
The energy can be expressed in terms of the complex amplitude $A$ by substituting $f=u^0 + \epsilon(Ae^{i\theta}+\circledast) + O(\epsilon^2)$ into \eqref{eq:waveeq}; using the dispersion relation \eqref{eq:linear-dispersion-relation} and the substitution $c=\omega/k + O(\epsilon)$ yields
$E=V(u^0) + 2\epsilon^2V''(u^0)|A|^2 + O(\epsilon^3)$.
Therefore, \eqref{eq:EnergyTransport} implies
\begin{equation}
(|A|^2)_{T_1} +\omega'(k)(|A|^2)_{X_1}=0
\end{equation}
in the limit $\epsilon\to 0$.  Of course this equation can be obtained from \eqref{eq:A-T1} simply by multiplying by $A^*$ and taking the real part.  In this way, the statement that the envelope propagates with the linear group velocity can be seen as arising from the Whitham theory in a special limiting case.  Now, recalling that the dynamical interpretation of the Whitham system is inconclusive in the degenerate case where the characteristic velocities agree, it is natural to try to go to higher order in the linearization to attempt to determine the fate of the modulated near-equilibrium librational wave.  Going to higher order brings in a balance between nonlinearity and linear dispersion (the latter is not present in the Whitham system at all) that results in the cubic nonlinear Schr\"odinger equation \eqref{eq:NLS-2}.  Therefore, while \eqref{eq:A-T1} can be viewed as a reduction of the Whitham modulational system \eqref{eq:Whitham-system-4} in a special case, \eqref{eq:NLS-2} can be viewed as a higher-order generalization thereof in the same limiting case.

\subsubsection{Interpretation of the cubic Schr\"odinger equation.  Relation to the modulational instability index $\rho$}
The initial-value problem for the nonlinear Schr\"odinger equation \eqref{eq:NLS-2} is the search for a solution $a(\zeta,\tau)$ for $\tau>0$ corresponding to given initial data $a(\zeta,0)$.
Unlike the quasilinear Whitham modulational system \eqref{eq:Whitham-system-4}, the initial-value problem for \eqref{eq:NLS-2} can be solved globally for general initial data.  In fact, the cubic nonlinear Schr\"odinger equation is well-known to be an integrable system, and the initial-value problem can be solved by means of an associated inverse-scattering transform.  On the other hand, solutions of the equation \eqref{eq:NLS-2} behave quite differently depending on the sign of $\beta\omega''(k)$.  In the \emph{defocusing case} corresponding to $\beta\omega''(k)<0$, an initially localized wave packet will gradually spread and disperse, but for large time (i.e., large $\tau$) the solution $a(\zeta,\tau)$ will qualitatively resemble the initial condition, a situation that can be interpreted as another kind of modulational stability of the carrier wave.  On the other hand, in the \emph{focusing case} corresponding to $\beta\omega''(k)>0$, an initially localized wave packet will break apart into smaller components (``bright'' solitons) that move apart with distinct velocities, a situation that can be interpreted as a sort of modulational instability.  In the weakly nonlinear context, the fact that one has a slowly-varying envelope $A=a(\epsilon(x-\omega'(k)t),\epsilon^2t)$ modulating a carrier wave of fixed wavenumber $k$ yields a natural interpretation in the Fourier domain:  the instability corresponds to the broadening of an initially narrow Fourier spectrum peaked at the wavenumber $k$.  Such instabilities are sometimes called \emph{sideband} instabilities, or by a connection with water wave theory, instabilities of \emph{Benjamin-Feir} type \cite{BenjaminF67}.

We may easily relate this dichotomy of behavior to the modulational instability index $\rho$ introduced in \S\ref{secmodinst}.
\begin{theorem}
Let $u^0$ be a non-degenerate critical point of $V$.  The nonlinear Schr\"odinger equation \eqref{eq:NLS-2} governing weakly nonlinear modulations of librational periodic traveling waves
of the Klein-Gordon equation \eqref{eqnlKG} that are near the equilibrium $u^0$ is of focusing (resp., defocusing) type if and only if  $\rho=-1$ (resp., $\rho=1$), where $\rho$ is the modulational instability index (cf.\@ Definition~\ref{defmodinst} in \S\ref{secmodinst}).
\label{theorem:NLS}
\end{theorem}
\begin{proof}
Eliminating the group velocity $\omega'(k)$ between the identities $k/\omega(k)=\omega'(k)$ and $\omega'(k)^2+\omega(k)\omega''(k)=1$ yields
\begin{equation}
\omega''(k)=\frac{\omega(k)^2-k^2}{\omega(k)^3}=\frac{V''(u^0)}{\omega(k)^3}.
\end{equation}
Therefore, recalling the definition of $\beta$ in \eqref{eq:NLS-2}, 
\begin{equation}
\beta\omega''(k)=\frac{5V'''(u^0)^2-3V''(u^0)V^{(4)}(u^0)}{12\omega(k)^4}.
\end{equation}
From Proposition~\ref{proposition-near-equilibrium} it then follows that 
\begin{equation}
\sgn\left(\beta\omega''(k)\right)=\sgn\left((c^2-1)\left.T_E\right|_{E=V(u^0)}\right).
\end{equation}
Applying Proposition~\ref{corollary:index-period} then shows that indeed $\sgn(\beta\omega''(k))=-\sgn(\rho)$.
\end{proof}

\section{Extensions} 
\label{sec:extensions}
While we have made precise assumptions on the potential function $V$ in the Klein-Gordon equation \eqref{eqnlKG}, these were made primarily to keep the exposition as simple as possible.  Indeed, many of our results also carry over to potentials that violate one or more of the key assumptions.

\subsection{More complicated periodic potentials}
Suppose that $V:\R\to\R$ is a smooth $2\pi$-periodic function, but now suppose also that $V$ has more than two critical points per period.  Generically there will be as many critical values as critical points, so we may assume that the absolute maximum and minimum of $V$ (we may still choose them to be $\pm 1$ without loss of generality) are achieved exactly once per period.  The key new feature that is introduced in this situation is that there are now multiple families of librational waves,  as separatrices associated with the new critical values other than $\pm 1$ appear in the phase portrait of equation \eqref{eq:waveeq}.  At a given speed $c$, several distinct librational orbits can now coexist at the same value of the energy $E$, isolated from one another by components of the new system of separatrices, illustrated in Figure~\ref{fig:PhasePortraitPlot3} with blue curves (color online).
\begin{figure}[h]
\begin{center}
\includegraphics{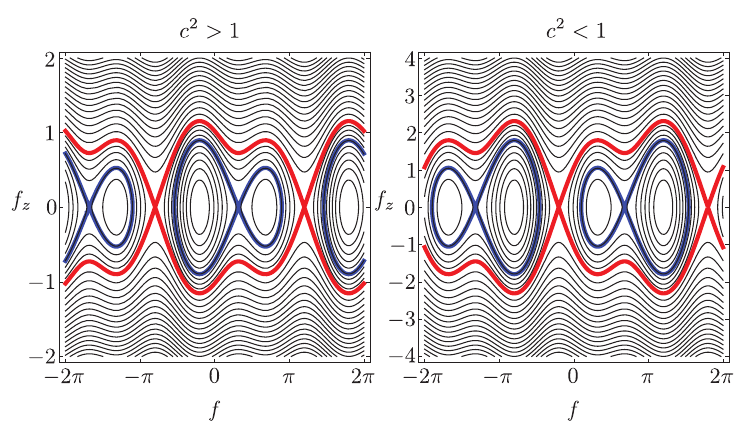}
\end{center}
\caption{Phase portraits of equation \eqref{eq:waveeq} corresponding to $c=2$ (left) and $c=1/2$ (right) for a potential $V(u)=-0.568[\cos(u) -\sin(2u)]$ that has four critical points per period.  The separatrices associated with the critical values $\pm 1$ are shown in red, and the separatrices associated with critical values lying within the open interval $(-1,1)$ are shown in blue.  (Color online.)}
\label{fig:PhasePortraitPlot2}
\end{figure}

Some of the families of librational orbits will surround neutrally stable fixed points (local minima of the effective potential), and such families are amenable to Chicone's theory of sufficient conditions for the monotonicity of the period $T$ of such orbits as a function of the energy $E$.  However, for potentials with more than two critical values per period there will now also be a family of librational orbits that are trapped between two separatrices, e.g., the red and blue curves in Figure~\ref{fig:PhasePortraitPlot3} (color online).  The period function for such a family of orbits \emph{cannot} be monotone as a function of $E$, because the period has to tend to $+\infty$ as $E$ tends to either of the distinct critical values corresponding to the enclosing separatrices.  As pointed out in Remark~\ref{remark:TE-near-separatrix} however, the librational orbits that are near the separatrix that stands between the librational and rotational orbits will necessarily satisfy $(c^2-1)T_E>0$,
and hence those results we have presented that depend on this condition will be applicable.  

The situation for librational waves is certainly made more complicated in the presence of additional critical values of $V$, as there are now many families of orbits to consider, each of which will have its own period function $T$ with its own monotonicity properties.  This is the only new feature, however, and otherwise the theory goes through unchanged. Additionally, there is no effect \emph{whatsoever} on the theory regarding rotational waves.  We summarize these results as follows:
\begin{theorem}\label{th:weaker} Let $V$ be a potential satisfying only condition (a) of Assumption~\ref{assumptionsV}, normalized for convenience according to Assumption~\ref{assumptionsnormalize}.  Then the conclusions of Theorem~\ref{theorem:summary} continue to hold provided that $T_E$ is calculated separately for each family of librational waves that may coexist at the same value of $E$.  
\end{theorem}

\subsection{Non-periodic potentials}
Now we consider the effect of dropping periodicity of the potential $V$ in the Klein-Gordon equation \eqref{eqnlKG}.  The most obvious feature of such non-periodic potentials is that there can no longer be any periodic traveling waves of rotational type at all (even for bounded $V$).  All periodic traveling waves are therefore of librational type in such a situation.  
Phase portraits for equation \eqref{eq:waveeq} corresponding to a quartic potential are illustrated in Figure~\ref{fig:PhasePortraitPlot3}.
\begin{figure}[h]
\begin{center}
\includegraphics{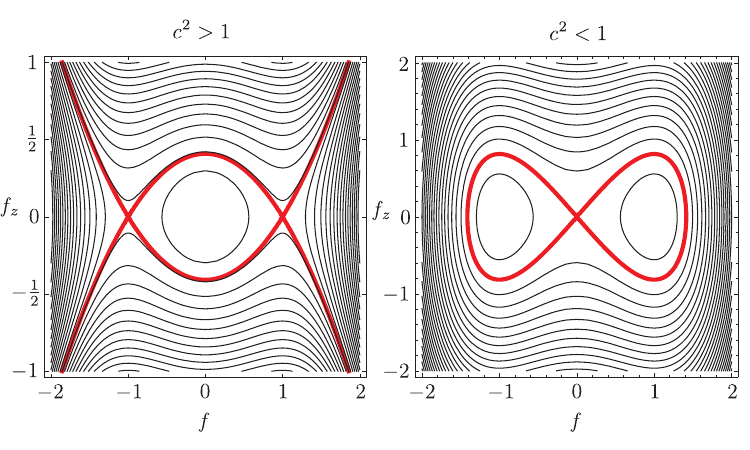}
\end{center}
\caption{Phase portraits of equation \eqref{eq:waveeq} for $c=2$ (left) and $c=1/2$ (right) for
a quartic potential $V(u)=\tfrac{1}{2}u^2-\tfrac{1}{4}u^4$.  Of course, these are just the phase portraits of different types of Duffing oscillators.  There is one family of librational waves for $c^2>1$ and there are three families of librational waves for $c^2<1$.}
\label{fig:PhasePortraitPlot3}
\end{figure}

Aside from this point, however, the stability theory of any family of librational periodic waves is exactly the same as in the case of periodic potentials, with the sole exception being that without an a priori bound on the potential $V$ (cf.\@ \eqref{eq:M-Vpp-define}) the spectral bounds established in Lemma~\ref{lem:bound} may not hold.  Therefore we have the following result.
\begin{theorem}
\label{theorem-non-periodic}
Suppose that $V:\R\to\R$ is a function of class $C^2$, and suppose that for a fixed wave speed $c\neq\pm 1$ there exists a family of librational periodic traveling wave solutions to the nonlinear Klein-Gordon equation parametrized locally by the energy $E$ on which they depend smoothly.  Then, except for the assertions that $\sigma\setminus i\R$ is bounded for subluminal waves and $\Re\sigma$ is bounded for superluminal waves, statements (iii) and (iv) of Theorem~\ref{theorem:summary} characterize the spectral stability properties of the family of  waves.  Also, librational waves of infinite speed are spectrally stable if and only if the associated Hill's spectrum $\Sigma^\mathrm{H}$ has no negative gaps.
\end{theorem}
For polynomial potentials $V$, Chicone's theory is particularly convenient to apply to determine whether a family of librational orbits that surround a stable fixed point (as is the case for the family of orbits in the left-hand panel of Figure~\ref{fig:PhasePortraitPlot3} and for both families of orbits in the right-hand panel of Figure~\ref{fig:PhasePortraitPlot3}).  For example, one may consider cubic nonlinearities arising from a quartic double-well potential:
\begin{equation}
V'(u) = au + bu^3 \quad\text{or}\quad V(u)=\tfrac{1}{2}au^2 + \tfrac14 bu^4 
\end{equation}
with $a,b \neq 0$.  Such potentials are important in the theory of the Klein-Gordon equation \eqref{eqnlKG} as it arises in models for magnetic fluid theory \cite{MaSi}, nonlinear meson theory \cite{Schi1}, solid state physics \cite{BiKrTr}, and in classical sine-Gordon and $\phi^4$-field theories \cite{Makh1}. The stability of periodic traveling waves for equation \eqref{eqnlKG} with this type of potential has been studied by Weissman \cite{Weissm1}, Murakami \cite{Mura1} and Parkes \cite{Prk1}, among others.

After a simple calculation one may show that 
\begin{equation}
 \frac{V(u)}{V'(u)^2} = \frac{1}{2a} - \frac{3b}{4a^2} u^2 + O(u^4),
\end{equation}
for $|u|\ll 1$. Therefore, by the criterion of Chicone \cite{Chi1}, 
the period of the librational oscillations surrounding the critical point $u=0$ (assuming that $(c^2-1)b<0$ making this point a local minimum of the effective potential) will be monotone in a neighborhood of zero inasmuch as $b \neq 0$. (The sign of the cubic term determines whether the nonlinear modification of the Hooke's law spring constant makes the spring ``harder'' or ``softer'', and it is physically reasonable that it should determine whether the period gets longer or shorter with increasing amplitude.)  In any case, knowledge of the sign of $T_E$, however it is obtained,  allows many of our results to be applied to this case immediately.  For example, according to Remark~\ref{remark:TE-near-separatrix}, one automatically has $(c^2-1)T_E>0$ near the outermost separatrix bounding each family of librational orbits.
Therefore, while we leave the details to the interested reader, it is in principle possible to 
recover, for example, the modulational instability observations of Parkes \cite{Prk1} for periodic waves near $u = 0$ for the cubic nonlinearity from a calculation of the modulational instability index $\rho$.

\section*{Acknowledgements}
RM gratefully acknowledges the partial support of the Australian Research Council on grant DP110102775.   PDM thanks the National Science Foundation for support on research grants DMS-0807653 and DMS-1206131, and benefited from a Research Fellowship from the Simons Foundation during the preparation of this work.
RGP thanks Tim Minzoni for useful conversations about Whitham's original paper \cite{Wh1}. The research of RGP was partially supported by DGAPA-UNAM, program PAPIIT, grant no. IN104814. 

\def\cprime{$'$}

%



\end{document}